\documentclass[a4paper,11pt,twoside]{amsart}
\usepackage[english]{babel}
\usepackage[utf8]{inputenc}
\usepackage[a4paper,inner=3cm,outer=3cm,top=4cm,bottom=4cm,pdftex]{geometry}
\usepackage{fancyhdr}
\usepackage{fancyhdr}
\usepackage{titlesec}
\usepackage{enumerate}
\usepackage{color}
\usepackage{bold-extra}
\titleformat{\section}{\normalfont\scshape\centering}{\thesection}{1em}{}
\titleformat{\subsection}{\bfseries}{\thesubsection}{1em}{}

\usepackage{mathrsfs}

\usepackage{comment}
\usepackage{graphics}
\usepackage{aliascnt}
\usepackage[pdftex,citecolor=green,linkcolor=red]{hyperref}

\usepackage{amsmath}
\usepackage{amsfonts}
\usepackage{amssymb}
\usepackage{amsthm}
\usepackage{mathtools}

\newtheorem{theorem}{Theorem}[section]
\newtheorem{corollary}[theorem]{Corollary}

\newtheorem{lemma}[theorem]{Lemma}
\newtheorem{proposition}[theorem]{Proposition}

\theoremstyle{definition}
\newtheorem{definition}[theorem]{Definition}
\newtheorem{remark}[theorem]{Remark}
\numberwithin{equation}{section}

\newcommand\1{\mathbf{1}}

\let\oldpmod\pmod
\renewcommand{\pmod}[1]{\hspace{-0.12cm}\oldpmod {#1}}

\DeclareMathOperator{\sgn}{sgn}

\begin{document}

\title{Linnik's problem for multiplicative functions}

\author{Kaisa Matom\"aki}
\address{Department of Mathematics and Statistics, University of Turku, 20014 Turku, Finland}
\email{ksmato@utu.fi}

\author{Joni Ter\"av\"ainen}
\address{Department of Pure Mathematics and Mathematical Statistics, University of Cambridge, Cambridge CB3 0WB, UK}
\email{joni.p.teravainen@gmail.com}

\begin{abstract} 
We study a multiplicative function analogue of Linnik's problem on the least prime in an arithmetic progression. Let $h\colon \mathbb{N}\to\mathbb{R}\setminus\{0\}$ be a multiplicative function, and let $a \pmod q$ be a reduced residue class. We ask how far one must go before finding square-free integers $n_1,n_2\equiv a \pmod q$ with $h(n_1)<0<h(n_2)$.

We show that one can always find such integers with $n_1,n_2\le q^{2+o(1)}$, unless the sign of $h$ strongly pretends to be a real Dirichlet character modulo $q$. Thus, apart from this natural character obstruction, sign changes of a multiplicative function occur in every reduced residue class at a scale corresponding essentially to the square root barrier.

In the special case of the Liouville function $\lambda$ this improves on a recent result of Ford and Radziwi{\l\l} and matches, up to $q^{o(1)}$ factors, what was previously known conditionally under the generalized Riemann hypothesis.
\end{abstract}

\maketitle

\section{Introduction}


For $q \in \mathbb{N}$ and a multiplicative function $h \colon \mathbb{N} \to \mathbb{R}$, define
\begin{multline*}
  R(h; q) := \min\{N \in \mathbb{N} \colon \text{For all $a \in \mathbb{Z}_q^\times$, there exist square-free integers $n_1, n_2 \leq N$} \\  \text{such that $n_1 \equiv n_2 \equiv a \pmod{q}$ and $h(n_1) < 0 < h(n_2)$}\}.
\end{multline*}
Thus, $R(h;q)$ expresses the threshold for finding a sign change of $h$ in every reduced residue class modulo $q$.
In this paper we study the problem of upper bounding $R(h; q)$, which can be seen as a multiplicative function analogue of Linnik's problem concerning the least prime in a residue class.

If $\chi$ is a real character $\pmod{q}$, then $R(\chi; q)$ does not exist, so we necessarily need to make some assumptions about the function $h$. In this paper we shall show that $R(h; q)$ is at most slightly larger than $q^2$, unless the sign of $h$ pretends to be a real character in a certain strong sense. This is achieved in the following theorem.

\begin{theorem}
\label{thm:generaleasy}
Let $\varepsilon > 0$ be sufficiently small, let $c > 0$, and let $q \in \mathbb{N}$ be sufficiently large in terms of $\varepsilon$ and $c$. Let $h \colon \mathbb{N} \to \mathbb{R}$ be a multiplicative function such that $h(p) \neq 0$ for every $p \nmid q$. Then
\[
  R(h; q) \leq q^{2+\varepsilon},
\]
unless there exists a character $\chi \pmod{q}$ of order at most two such that the sign of $h$ pretends to be $\chi$ in the sense that
\begin{equation} \label{eq:thmConditionEasyInThm}
\sum_{\substack{p \leq q^{1/3} \\ h(p)\chi(p) < 0}} \frac{1}{p} \leq \frac{c}{q^{\varepsilon/100}}.
\end{equation}
\end{theorem}

In many cases, we can further improve the bound $R(h;q)\leq q^{2+\varepsilon}$ to $R(h;q)\leq q^{2+o(1)}$ with a very good control on the $o(1)$-term. For technical reasons, our bound is slightly weaker if $q$ has very many small prime factors. To take this into account, we define
\begin{equation} \label{eq:defB(q)}
  B(q)\coloneqq \min\left\{B \geq 2 \colon \text{For all $z \geq B$, we have } |\{p \leq z\colon p \mid q\}| \leq \frac{z}{10\log z} \right\}.
\end{equation}
Thus $B(q)$ expresses the threshold beyond which the primes dividing $q$ are no longer unusually dense.
Notice that $B(q) \ll 1$ for almost all $q \in \mathbb{N}$ and $B(q) \leq 20 \log q$ for all $q\geq 2$. Theorem~\ref{thm:generaleasy} is a special case of the following theorem (take $Q_1 = q^\varepsilon$).
\begin{theorem}
\label{thm:general}
Let $\varepsilon > 0$ be sufficiently small and let $c > 0$. Let both $q \in \mathbb{N}$ and $Q_1 \in [B(q), q^\varepsilon]$ be sufficiently large in terms of $\varepsilon$ and $c$. Let $h \colon \mathbb{N} \to \mathbb{R}$ be a multiplicative function such that $h(p) \neq 0$ for every $p \nmid q$. Then
\[
  R(h; q)\leq q^2 Q_1,
\]
unless there exists a character $\chi \pmod{q}$ of order at most two such that the sign of $h$ pretends to be $\chi$ in the sense that
\begin{equation} \label{eq:thmCondition}
\sum_{\substack{p \leq q^{1/2} \\ h(p)\chi(p) < 0}} \frac{1}{p} \leq \frac{c}{Q_1^{1/100}}.
\end{equation}
\end{theorem}
We have not tried to optimize the exponents in $q^{1/2}$ and $Q_1^{1/100}$ in~\eqref{eq:thmCondition}. The condition that $h(p) \neq 0$ for every $p \nmid q$ is for convenience; it would be possible to use the same method when $h(p)$ vanishes for some (but not too many) primes $p \nmid q$.

We will obtain the following corollary for the M\"obius function. For the deduction, see Section~\ref{sec:proofofCor}.
\begin{corollary}
\label{cor:Liouville}
Let $\varepsilon_0 > 0$ and let $q$ be sufficiently large in terms of $\varepsilon_0$. Define
\[
  L(q) \coloneqq  \max_{\substack{\chi \pmod q \\ \chi \textnormal{ real}}} \prod_{p > q} \left(1-\frac{\chi(p)}{p}\right).
\] 
Then
\[
  R(\mu; q) \ll q^2 \left(L(q)^{100}+B(q)\right) \ll_{\varepsilon_0} q^{2+\varepsilon_0}.
\]
\end{corollary}
This improves upon a recent result of Ford and Radziwi{\l\l}~\cite{FordRad} concerning sign changes of the Liouville function in arithmetic progressions. They proved that, for any $\varepsilon > 0$, any sufficiently large prime $q$, and any $a \in \mathbb{N}$, one can find integers $m, n \leq q^{5/2+\varepsilon}$ such that $m \equiv n \equiv a \pmod{q}$ and $\lambda(m) = -1$ and $\lambda(n) = 1$.
\begin{remark}
In the definition of $R(h; q)$ we consider only $a \in \mathbb{Z}_q^\times$. On the other hand, for integers $a, q \in \mathbb{N}$ with $(a, q) = r$ and a multiplicative function $h \colon \mathbb{N} \to \mathbb{R}$ such that $h(r) \neq 0$, the least $n$ such that $n \equiv a \pmod{q}$ and $\sgn(h(n)) = \Delta$ (if it exists) is of the form $n_0r$ with $n_0 \equiv a/r \pmod{q/r}$ and $\sgn(h(n_0)) = \Delta \sgn(h(r))$ and thus satisfies the bound $n \leq r \cdot R(h; q/r)$. Hence, for instance for the M\"obius function, we see that for each $\Delta\in \{-1,+1\}$, the least $n\equiv a\pmod q$ with $\mu(n)=\Delta$ is $\ll_{\varepsilon} q^{2+\varepsilon}$ whenever $(a,q)$ is square-free.
\end{remark}

\subsection{Previous results}
We already mentioned the recent work of Ford and Radziwi{\l\l}~\cite{FordRad}, but let us next dive deeper into the history of the topic.

Linnik's theorem~\cite{Linnik1}, \cite{Linnik2} from 1944 asserts that, for every reduced residue class $a \pmod q$, the least prime $p \equiv a \pmod q$ is  $\ll q^{L+o(1)}$ for some absolute constant $L$. Considerable effort has gone into reducing the admissible value of $L$; we refer to~\cite{Heath-Brown2} for a discussion of previous results. Xylouris~\cite{Xylouris} proved that $L=5$ is admissible, using a method due to Heath-Brown~\cite{Heath-Brown2}. When it comes to conditional results, it was shown already in 1934 by Chowla~\cite{chowla-AP} that the generalized Riemann hypothesis (GRH) implies that $L=2$ is admissible. This corresponds to the natural square root barrier for equidistribution problems in arithmetic progressions, and so one does not expect multiplicative methods to do substantially better, even under GRH. Nevertheless, Chowla also conjectured that $L=1$ works, which would be the optimal result. Thus, the known unconditional results are still a significant distance away from the conjectural results.

There has also been substantial recent interest in analogues of Linnik's problem for numbers with few prime factors. Ramar\'e and Walker~\cite{ramare-walker} proved that 
every reduced residue class modulo $q$ contains a product $p_1p_2p_3\leq q^{16}$ of exactly three primes, with $p_1,p_2,p_3\leq q^{16/3}$. More recently, in~\cite{MatoTeraEkq} the present authors showed that for every sufficiently large cube-free modulus $q$, every reduced residue class modulo $q$ can be represented as a product of three primes $p_1p_2p_3\leq q^3$ with $p_1,p_2,p_3 \le q$, and that for all large enough $q$, at least $(2/3-\varepsilon)\varphi(q)$ residue classes admit a representation as a product of two primes $p_1p_2\leq q^2$ with $p_1,p_2 \le q$. These results may be viewed as partial ternary and binary analogues of the conjectural $q^{2+o(1)}$
bound in Linnik's problem.

It is also natural to ask for Linnik-type results for other multiplicatively defined sets, such as the totient numbers (numbers of the form $\varphi(n)$, where $n$ is any natural number). Recently, Jha~\cite{Jha} proved that for any odd modulus $q$ and any reduced residue class $a \pmod q$, there exists a totient value $v\equiv a\pmod q$ with $v \le q^{2+o(1)}$. 

For general bounded multiplicative functions, Klurman, Mangerel and Ter\"av\"ainen~\cite{KMT} established variance bounds in short arithmetic progressions for almost all moduli. As a consequence, they obtained Linnik-type results for products of exactly three primes for almost all moduli, such as the existence of products of three primes $p_1p_2p_3\leq q^{2+\varepsilon}$ in every reduced residue class modulo $q$, for sufficiently smooth moduli $q$, or for all but a small number of exceptional prime moduli $q$. However, the approach of that paper relied on zero-free regions for Dirichlet $L$-functions (which are much wider for smooth moduli), and zero-density estimates (which give good zero-free regions for most moduli), and therefore that approach does not seem to extend to all moduli without significantly stronger information on zero-free regions for Dirichlet $L$-functions.

Finally, in the specific case of the M\"obius function, the bound $q^{2+o(1)}$ was previously available conditionally under GRH. Indeed, GRH gives the expected $q^{2+o(1)}$ bound for the least prime in a reduced residue class, from which one can obtain both signs of $\mu$ in a fixed class by considering a prime and a suitable product $p_1p_2$ of two distinct primes (where for instance $p_2$ is fixed to be the least prime coprime to $q$, and $p_1$ is the least prime in the progression $p_2^{-1}a\pmod q$). Our result therefore recovers unconditionally, for the M\"obius function, the strength that was previously accessible only under GRH.

\subsection{Proof outline}
We will first prove the simpler Theorem~\ref{thm:generaleasy}, and then modify the argument to obtain the full Theorem~\ref{thm:general}. The proofs have three main ingredients: a multiplicative dense model theorem, additive combinatorial
information on triple product sets in $\mathbb{Z}_q^\times$, and estimates for character sums that allow us to verify the conditions of the dense model theorem and transfer from the original sparse problem to the dense model. 

\medskip

\noindent \textbf{Proof of Theorem~\ref{thm:generaleasy}:}
For Theorem~\ref{thm:generaleasy}, we look for integers of the form
\begin{align*}
n=r_1r_2r_3pu,
\end{align*}
where $r_1,r_2,r_3$ are square-free numbers of size about $q^{1/2}$ with no small prime factors, $p$ is a prime in $(q^\varepsilon/e,q^\varepsilon]$, and $u\le q^{1/2}$ is square-free and restricted to an index two coset. Any such number is at most $q^{2+\varepsilon}$ in size, and is square-free with high probability. We will fix the signs of $h(r_j)$ and $h(p)$ and use the factor $u$ to force both signs of $h$ to appear by appealing to lower bounds for square-free numbers on which a multiplicative function has prescribed sign. 

To study the contribution of the variables $r_j$, we introduce functions $f^\Delta$ that detect square-free integers in a fixed interval, free of small prime factors, and with $\sgn(h(n))=\Delta$. These are sparse functions on $\mathbb{Z}$, but Proposition~\ref{prop:f=g+h} allows us to replace them by dense model functions $g^\Delta \colon \mathbb{Z}_q^\times \to [0,1+o(1)]$ that have essentially the same character sums. For showing that the count of solutions is roughly the same, we use the fact that the product $r_1r_2r_3u$  can be split into two subproducts of comparable size and that we have a small prime variable $p$.

The key point is that after this replacement one is no longer dealing with a sparse multiplicative set, but with dense subsets of the finite abelian group $\mathbb{Z}_q^\times$. Writing
\[
A^\Delta:=\{a\in \mathbb{Z}_q^\times:\ g^\Delta(a)\ge \varepsilon^2\},
\]
one is thus led to a problem about triple products of the sets $A^+$ and $A^-$.

The combinatorial input is that large subsets of $\mathbb{Z}_q^\times$ have very rigid product-set behaviour. Roughly speaking, Proposition~\ref{prop:AAA->E1} states that if one of the triple convolutions
\[
\mathbf{1}_{A^\Delta} * \mathbf{1}_{A^\Delta} * \mathbf{1}_{A^\Delta}
\]
is large on all of $\mathbb{Z}_q^\times$, then one can represent every residue class using three factors of the same sign, and the remaining variables $p$ and $u$ are used to adjust the final sign of $h(n)$. If this does not happen,
then Kneser-type arguments of Lemma~\ref{le:KneserAppl} show that both $A^+$ and $A^-$ must be concentrated on cosets of an index two subgroup $H\le \mathbb{Z}_q^\times$, and in fact on opposite cosets.  The rest of the proof shows that in this exceptional case the sign of $h$ must correlate strongly with the quadratic character attached to $H$, giving the alternative in the theorem. 

The role of the short prime variable $p\in (q^\varepsilon/e,q^\varepsilon]$ is especially important in the transference step. After expanding by characters, one needs to compare a sparse convolution built from the $f^\Delta$
with its dense analogue built from the $g^\Delta$. Since the $p$-sum is long enough to admit a good large values estimate, one can split the characters into those for which the prime sum is small and those for which it is large, and control both contributions using mean square estimates for character sums. This is the reason why Theorem~\ref{thm:generaleasy} is significantly simpler than the general theorem.

\medskip

\noindent \textbf{Proof of Theorem~\ref{thm:general}:}
For the full Theorem~\ref{thm:general}, we follow the same broad strategy, but the transference step becomes more delicate because the prime variable may now be much shorter. To compensate for this, we insert an additional factorization and look for integers of the form
\[
n=r_1r_2r_3 \cdot p_1 \cdot u \cdot m,
\]
where $p_1\in (Q_1/e,Q_1]$, the variables $r_j$ lie in suitable $e$-adic intervals of length about $q^{1/2-\varepsilon/4}$, the variable $u$ has length about $q^{1/2+\varepsilon/4}$, and $m$ has length
$q^{\varepsilon/2}$ and is required to possess a prime factor from each of a sequence of disjoint intervals $(P_j,Q_j]$. The purpose of the factor $m$ is to provide a ``ladder'' of prime factors of increasing sizes, as in~\cite{matomaki-radziwill}. We can then use various character sum estimates and a case analysis depending on which character sum is large to conclude.

The most technical part of the paper is the comparison between the sparse and dense convolutions in this general setting. After expanding in characters, one partitions the characters into classes $\mathcal{X}_j$ and $\mathcal{Y}$. For a character in $\mathcal{X}_j$, the prime sum in the range $(P_j,Q_j]$ is small, and one can exploit this by using a pointwise bound on this character sum and mean value estimates and an amplification argument for the remaining character sums. The remaining set $\mathcal{Y}$ consists of characters for which the prime sums are large on every scale. Large values estimates show that there are very few such characters, so their total contribution is negligible.

The dense model theorem is then applied separately to the three variables $r_1,r_2,r_3$ in each $e$-adic interval, producing sets $A_k^\pm\subseteq \mathbb{Z}_q^\times$. The combinatorial analysis is by now similar in spirit to the case of Theorem~\ref{thm:generaleasy}, but with one important new feature: since the variables $r_j$ are restricted to many different intervals, one has to sum over many triples $(k_1,k_2,k_3)$. If for many such triples one has strong triple product expansion,
then one is in the generic case. If not, then for many $k$ the sets $A_k^+$ and $A_k^-$ must each be concentrated in cosets of some index two subgroup $H_k$. At this point the argument splits again. Either many of the $H_k$ are equal, which yields the same quadratic obstruction as before, or else many distinct subgroups occur, and then mixed triple products coming from different $k$'s are forced to expand, bringing us back to the generic case. 

\medskip

\noindent \textbf{Structure of the paper:}
The structure of the paper reflects this strategy of first proving Theorem~\ref{thm:generaleasy} and then Theorem~\ref{thm:general}. In Section~\ref{sec:auxil1} we collect the tools needed for the proof of Theorem~\ref{thm:generaleasy}, including the sign results for multiplicative functions, the dense model theorem, the character sum estimates, and the additive combinatorics lemmas on product sets. In Section~\ref{sec:proofofCor} we deduce Corollary~\ref{cor:Liouville} from the main theorem. Section~\ref{sec:thmgeneraleasy} contains the proof of Theorem~\ref{thm:generaleasy}. In Section~\ref{sec:auxil2} we develop the additional character sum estimates and decomposition lemmas needed for the full theorem, and the final four sections then carry out the general transference argument and complete the proof of Theorem~\ref{thm:general}.

\section*{Acknowledgements}
The authors would like to thank Kevin Ford and Maksym Radziwi{\l\l} for suggesting the problem and for sharing details of their work~\cite{FordRad} before it was publicly available.

The first author was supported by Research Council of Finland grants number 346307, 333707, and 370133. The second author was supported by European Union's Horizon Europe research and innovation programme under Marie Sk\l{}odowska-Curie grant agreement no. 101058904 and ERC grant agreement no. 101162746.

\section{Notation} The letter $p$, with or without subscripts, is reserved for prime numbers. For $z\geq 1$, we write 
$$P(z)\coloneqq \prod_{p<z}p.$$
Define, for $q \in \mathbb{N}$, $N \geq 2$, and an interval $I \subseteq\mathbb{R}$,
\[
[N]_q = \{n \in \mathbb{N} \colon n\in [1,N] \text{ and } (n, q) = 1\} \quad \text{and} \quad [I]_q = \{n \in \mathbb{N} \cap I \colon (n, q) = 1\}.
\]
We write $\tau_k$ for the $k$-fold divisor function and abbreviate $\tau_2=\tau$. Using M\"obius inversion, we see that for any finite interval $I \subseteq \mathbb{R}$ and $q \in \mathbb{N}$ we have
\begin{equation}
\label{eq:Simplegcd}
|[I]_q| = \sum_{\substack{m \in I \\ (m, q) = 1}} 1 = \sum_{d\mid q}\mu(d)\sum_{\substack{m \in I \\ d \mid m}}1=\sum_{d\mid q}\mu(d)\left(\frac{|I|}{d}+O(1)\right)=|I|\frac{\varphi(q)}{q} + O(\tau(q)).
\end{equation}

We define the function $\sgn \colon \mathbb{R} \setminus \{0\} \to \{+, -\}$ by 
\[
\sgn(x) = 
\begin{cases}
+ & \text{if $x > 0$;} \\
- & \text{if $x < 0$.}
\end{cases}
\]
Throughout, we identify the set $\{+,-\}$ with the set $\{+1,-1\}$. Thus, for $\Delta_1, \Delta_2 \in \{+, -\}$, we define the product $\Delta_1 \Delta_2$ to be $+$ if $\Delta_1 = \Delta_2$ and $-$ otherwise. For a proposition $P$ and a set $A$, we define
\[
  \1_P = \begin{cases}
    1 &\text{if $P$ holds;} \\
    0 & \text{otherwise;}
  \end{cases}
  \quad \text{and} \quad
  \1[A](n) = \begin{cases}
    1 &\text{if $n \in A$;} \\
    0 & \text{otherwise.}
  \end{cases}
          \]

For $\Delta \in \{+, -\}$ and $x\geq 1$, we write 
\begin{align*}
  E_h^\Delta(x) &:= \{a \in \mathbb{Z}_q^\times \colon \text{there exists a square-free $n \leq x$ such that } \\
  & \qquad \qquad \qquad n \equiv a \pmod{q} \text{ and } \sgn(h(n)) = \Delta\}.
\end{align*}
Note that $E_h^{\Delta}(x)$ depends also on $q$, but the choice of $q$ will always be clear from context.

With this notation, our aim is to prove that, under the assumptions of Theorem~\ref{thm:general}, either
\begin{equation*}
E_h^+(q^{2}Q_1) = E_h^-(q^{2}Q_1) = \mathbb{Z}_q^\times,
\end{equation*}
or there exists a quadratic or principal character such that~\eqref{eq:thmCondition} holds.

By an abuse of notation, for an integer $n$ and a set $A\subseteq\mathbb{Z}_q^{\times}$, we write $n\in A$ to mean $n\pmod q\in A$. Also, when the choice of $q$ is clear from the context and $(n,q)=1$, we write $\overline{n}\in \mathbb{Z}_q^{\times}$ for the unique solution to $n\overline{n}\equiv 1\pmod q$. 

For functions $f,g\colon [I]_q\to \mathbb{C}$, we can extend them to all of $\mathbb{Z}$ by setting them equal to $0$ outside $[I]_q$. We then use $f*g$ to denote the Dirichlet convolution of these extensions of $f$ and $g$, i.e.
\[
  (f \ast g)(n) = \sum_{\substack{ab = n \\ a, b \in [I]_q}} f(a)g(b).
  \]
\section{Auxiliary results}\label{sec:auxil1}

\subsection{Signs of multiplicative functions} 
In this subsection, our goal is to prove the following lemma concerning positive and negative values of multiplicative functions at square-free integers.
\begin{lemma}\label{lemma:le:multfunctsigns}
Let $\varepsilon > 0,$ let $q \in \mathbb{N}$ be sufficiently large in terms of $\varepsilon$, and let $y \geq q^{\varepsilon/2}$. Let $h_0 \colon \mathbb{N} \to \mathbb{R}$ be a multiplicative function such that $h_0(p) \neq 0$ for every $p \nmid q$. Let $\chi_0$ be the principal character $\pmod q$. Then
  \begin{equation} \label{eq:h0poslow}
    \sum_{\substack{n \leq y \\ h_0(n) > 0}} \chi_0(n) |\mu(n)| \gg \frac{\varphi(q)}{q} y 
  \end{equation}
  and
  \begin{equation} \label{eq:h0neglow}
    \sum_{\substack{n \leq y \\ h_0(n) < 0}} \chi_0(n) |\mu(n)| \gg \frac{\varphi(q)}{q} y \cdot \min\left\{1, \sum_{\substack{p \leq y/q^{\varepsilon/8} \\ h_0(p) < 0}} \frac{\chi_0(p)}{p}\right\}. 
    \end{equation}
  \end{lemma}

For proving this, we need the following slight variant of~\cite[Theorem 1]{GKM} (we do not need the stronger form from~\cite{MSsieve}).

\begin{lemma} \label{le:sieveworks}
There exist positive constants $\lambda$ and $c$ such that if $x\geq 10$ and $\mathcal{P}$ is a subset of the primes $\le x$ for which there is some $v\in [1, c\sqrt{\log x}]$ with
\[
\sum_{\substack{ p\in \mathcal{P} \\ x^{1/(ev)}<p\leq x}} \frac 1p \geq 1+\lambda,
\]
then
\begin{equation}
\label{eq:SieveWorks}
\frac{1}{x} \sum_{\substack{n \leq x \\ p \mid n \implies p \in \mathcal{P}}} |\mu(n)| \gg \ \frac{1}{v^{v/c}}  \ \prod_{\substack{p \leq x \\ p \not \in \mathcal{P}}}  \left( 1 -\frac 1p \right).
\end{equation}
\end{lemma}

\begin{proof}
If we had $1$ in place of $|\mu(n)|$ on the left-hand side of~\eqref{eq:SieveWorks}, the claim would immediately follow from~\cite[Theorem 1]{GKM}. In order to prove our slightly stronger claim, we follow~\cite[Proof that Hypothesis P implies Theorem 1 in Section 3]{GKM} and write $\mathcal{A} = \mathcal{P} \cap [1, x^{1/(ev)}]$ and $\mathcal{B} = \mathcal{P} \cap (x^{1/(ev)}, x]$. Then
\begin{align} \label{eq:GKMfirst}
  \begin{aligned}
 \frac{1}{x} \sum_{\substack{n \leq x \\ p \mid n \implies p \in \mathcal{P}}} |\mu(n)| &\geq \frac{1}{x} \sum_{\substack{a \leq x^{1/4} \\ p \mid a \implies p \in \mathcal{A}}} |\mu(a)| \sum_{\substack{b \leq x/a \\ p \mid b \implies p \in \mathcal{B}}} |\mu(b)| \\
  & \geq \frac{1}{x} \sum_{\substack{a \leq x^{1/4} \\ p \mid a \implies p \in \mathcal{A}}} |\mu(a)| \sum_{\substack{b \leq x/a \\ p \mid b \implies p \in \mathcal{B}}} 1 - \frac{1}{x} \sum_{\substack{a \leq x^{1/4}}} \sum_{p \leq (x/a)^{1/2}} \sum_{\substack{b' \leq x/(ap^2)}} 1.
\end{aligned}
  \end{align}
Now the second term on the right-hand side is at most
\[
  \frac{1}{x} \sum_{\substack{a \leq x^{1/4}}} \sum_{p \leq (x/a)^{1/2}} \frac{x}{ap^2} \ll \frac{1}{x} \sum_{a \leq x^{1/4}} \left(\frac{x}{a}\right)^{1/2} \ll x^{-3/8}.
  \]
As in~\cite{GKM}, we apply~\cite[Hypothesis P]{GKM} (which holds with $\pi_v = v^{-O(v)}$ by~\cite[Proposition 4.1 and Section 6]{GKM}) to the first sum on the right-hand side of~\eqref{eq:GKMfirst}, obtaining
\[
  \frac{1}{x} \sum_{\substack{a \leq x^{1/4} \\ p \mid a \implies p \in \mathcal{A}}} |\mu(a)| \sum_{\substack{b \leq x/a \\ p \mid b \implies p \in \mathcal{B}}} 1 \gg \frac{1}{v^{O(v)} \log x} \sum_{\substack{a \leq x^{1/4} \\ p \mid a \implies p \in \mathcal{A}}} \frac{|\mu(a)|}{a}.
  \]
  We then argue as in~\cite[Proof of Lemma 2.1]{GKM}. Observe that
  \begin{align*}
  \sum_{\substack{\ell \leq x^{1/4} \\ p \mid \ell \implies p \not \in \mathcal{P}}} \frac{1}{\ell}\leq  \prod_{\substack{p \leq x^{1/4} \\ p \not \in \mathcal{P}}} \left(1-\frac{1}{p}\right)^{-1}.   
  \end{align*}
  Hence, we can estimate
  \begin{align*}
    \sum_{\substack{a \leq x^{1/4} \\ p \mid a \implies p \in \mathcal{A}}} \frac{|\mu(a)|}{a} &\geq \sum_{\substack{a \leq x^{1/4} \\ p \mid a \implies p \in \mathcal{A}}} \frac{|\mu(a)|}{a} \cdot \left(\sum_{\substack{\ell \leq x^{1/4} \\ p \mid \ell \implies p \not \in \mathcal{P}}} \frac{1}{\ell} \prod_{\substack{p \leq x^{1/4} \\ p \not \in \mathcal{P}}} \left(1-\frac{1}{p}\right) \right)  \\
    & \geq \sum_{\substack{n \leq x^{1/4} \\ p \mid n \implies p \leq x^{1/(ev)}}} \frac{|\mu(n)|}{n} \prod_{\substack{p \leq x^{1/4} \\ p \not \in \mathcal{P}}} \left(1-\frac{1}{p}\right) \gg (\log x) \prod_{\substack{p \leq x \\ p \not \in \mathcal{P}}} \left(1-\frac{1}{p}\right),
  \end{align*}
  and the claim follows.
\end{proof}
With this lemma in hand, we are ready to prove Lemma~\ref{lemma:le:multfunctsigns}.
  \begin{proof}[Proof of Lemma~\ref{lemma:le:multfunctsigns}]
    Let $C_0$ be a sufficiently large absolute constant. We may assume $q>C_0$. 
We split into two cases.
    
    \textbf{Case 1: We have}
    \begin{equation} \label{eq:h0sumsmall}
      \sum_{\substack{p \leq y \\ h_0(p) < 0}} \frac{\chi_0(p)}{p} < C_0. 
    \end{equation}
    Let
      \[
        \mathcal{P} := \{p \leq y \colon h_0(p) > 0 \text{ and } p \nmid q\}.
        \]
In this case, let $\lambda$ be the constant appearing in Lemma~\ref{le:sieveworks}. Then there exists $v = v(C_0) > 0$ such that, for any $y_0 \in [q^{\varepsilon/8}, y]$,
      \[
        \sum_{\substack{y_0^{1/(ev)} < p \leq y_0 \\ p \in \mathcal{P}}} \frac{1}{p} \geq 1 + \lambda.
        \]
Now, for any $y_0 \in [q^{\varepsilon/8}, y]$, we have, by Lemma~\ref{le:sieveworks} and~\eqref{eq:h0sumsmall},
      \begin{equation} \label{eq:sumnh0+low}
        \sum_{\substack{n \leq y_0 \\ p \mid n \implies h_0(p) > 0}} \chi_0(n) |\mu(n)| = \sum_{\substack{n \leq y_0 \\ p \mid n \implies p \in \mathcal{P}}} |\mu(n)| \gg y_0 \prod_{\substack{p \leq y_0 \\ p \not \in \mathcal{P}}} \left(1-\frac{1}{p}\right) \gg \frac{\varphi(q)}{q} y_0.
        \end{equation}
Taking $y_0 = y$, this immediately implies~\eqref{eq:h0poslow}. To show~\eqref{eq:h0neglow}, notice that
        \[
          \sum_{\substack{n \leq y \\ h_0(n) < 0}} \chi_0(n) |\mu(n)| \geq \sum_{\substack{p \leq y/q^{\varepsilon/8} \\ h_0(p) < 0}} \chi_0(p) \sum_{\substack{n \leq y/p \\ p' \mid n \implies h_0(p') > 0}} \chi_0(n) |\mu(n)|.
    \]
Now~\eqref{eq:h0neglow} follows from applying~\eqref{eq:sumnh0+low} with $y_0 = y/p$ to the inner sum.

\textbf{Case 2: We have}
\[
      \sum_{\substack{p \leq y \\ h_0(p) < 0}} \frac{\chi_0(p)}{p} \geq C_0. 
    \]
    Let $\Delta \in \{+, -\}$ and define the multiplicative function
    \[
      h_1(n) :=
      \begin{cases} 
	0 & \text{if $\mu(n) = 0$ or $(n, q) \neq 1$;} \\      
    \sgn(h_0(n)) & \text{otherwise.}
      \end{cases}
    \]
    Note that
    \begin{equation} \label{eq:splittingS1S2}
          \sum_{\substack{n \leq y \\ \sgn(h_0(n)) = \Delta}} \chi_0(n) |\mu(n)| = \frac{1}{2} \left( \sum_{\substack{n \leq y}} \chi_0(n) |\mu(n)| + \Delta \sum_{\substack{n \leq y}} h_1(n) \right)\eqqcolon \frac{1}{2}(S_1+S_2),
        \end{equation}
        say. By a standard M\"obius inversion calculation, we have
        \begin{equation} \label{eq:S1lower}
        S_1=(1+o(1)) \frac{\varphi(q)}{q} \prod_{\substack{p \leq y \\ p \nmid q}} \left(1-\frac{1}{p^2}\right)\cdot y \geq \frac{1}{10} \cdot \frac{\varphi(q)}{q} y.
      \end{equation}
      
For multiplicative functions $f,g\colon \mathbb{N}\to \mathbb{C}$ and for $r\in \mathbb{N}$ and $x\geq 2$, define the distance function
\begin{align*}
\mathbb{D}_r(f,g;x)=\left(\sum_{\substack{p\leq x\\ p\nmid r}}\frac{1-\textnormal{Re}(f(p)\overline{g(p)})}{p}\right)^{1/2}.    
\end{align*}
Then, by~\cite[Corollary 2.2]{bgs} (which is a quantitative version of Hal\'asz's theorem), for some large absolute constant $C$ independent of $C_0$ and some $t\in [-(\log y)^{1/2},(\log y)^{1/2}]$ we have
\begin{align*}
|S_2|\leq C\frac{\varphi(q)}{q}y(1+\mathbb{D}_q(h_1,n^{it};y)^2)\exp(-\mathbb{D}_q(h_1,n^{it};y)^2)+O\left(\frac{y}{(\log y)^{1/4}}\right).    
\end{align*}
Since $h_1$ is real-valued, by an argument of Granville and Soundararajan (see~\cite[Lemma C.1]{mrt}) we have, with an absolute implied constant,
\begin{align*}
\mathbb{D}_q(h_1,n^{it};y)\geq \frac{1}{100}\mathbb{D}_q(h_1,1;y)+O(1) \geq \frac{1}{100} \cdot \sqrt{2C_0}+O(1).   
\end{align*}
Hence, if $C_0$ is large enough,
\[
  |S_2| \leq \frac{1}{100} \frac{\varphi(q)}{q} y,
\]
and the claim follows from combining this with~\eqref{eq:splittingS1S2} and~\eqref{eq:S1lower}.
\end{proof}

\subsection{The dense model theorem}
The following proposition (which follows immediately from~\cite[Proposition 4.1]{MatoTeraEkq} where it was used in a related context) is a dense model theorem which gives, for an unbounded function $f: [I]_q \to \mathbb{R}_{\geq 0}$ which is majorized by a pseudorandom measure, a model function $g : \mathbb{Z}_q^\times \to \mathbb{R}_{\geq 0}$ that is bounded and such that character sums of $f$ and $g$ behave similarly. 

\begin{proposition}[A multiplicative dense model theorem] \label{prop:f=g+h}
Let $N \geq 2$ and $q \in \mathbb{N}$ and let $I = (N/e, N]$. Let $r > 1$ be fixed. Let $\eta, \varepsilon \in (0,1), C\geq 1$, and let 
\begin{equation*}
\delta \in \left(\left(\frac{10 C r  \log \log q}{\varepsilon \log q}\right)^{1/r}, \frac{1}{10}\right).
\end{equation*}
Let $f \colon [I]_q \to \mathbb{R}_{\geq 0}$ satisfy the following two assumptions.
\begin{enumerate}[({A}1)]
\item There exists a majorant function $\nu \colon [I]_q \to \mathbb{R}_{\geq 0}$ such that $f(n) \leq \nu(n)$ for every $n \in [I]_q$, 
\[
\left|\mathbb{E}_{n \in [I]_q} \nu(n) - 1\right| \leq \eta, \quad \text{and} \quad \max_{\chi \neq \chi_0\pmod q} \left|\mathbb{E}_{n \in [I]_q} \nu(n)\overline{\chi}(n) \right| \leq q^{-\varepsilon}.
\]
\item There exist at most $C\delta^{-r}$ characters $\chi \pmod{q}$ such that
\[
\left|\mathbb{E}_{n \in [I]_q} f(n)\overline{\chi}(n) \right| \geq \delta.
\]
\end{enumerate}
Then there exists a function $g \colon \mathbb{Z}_q^\times \to \mathbb{R}_{\geq 0}$ with the following properties.
\begin{enumerate}[(i)]
\item For every $a \in \mathbb{Z}_q^{\times}$, we have 
\[
0 \leq g(a) \leq 1+\eta + O(q^{-\varepsilon/2}).
\]
\item We have, for any $\chi \pmod{q}$,
\[
\left|\mathbb{E}_{n \in [I]_q} f(n) \overline{\chi}(n) - \mathbb{E}_{a \in \mathbb{Z}_q^\times} g(a) \overline{\chi}(a)\right| \leq \delta.
\]
\item We have, for any $\chi \pmod{q}$,
\[
\left|\mathbb{E}_{a \in \mathbb{Z}_q^\times} g(a) \overline{\chi}(a)\right| \leq |\mathbb{E}_{n \in [I]_q} f(n) \overline{\chi}(n)|
\]
and
\[
\left|\mathbb{E}_{n \in [I]_q} f(n) \overline{\chi}(n) - \mathbb{E}_{a \in \mathbb{Z}_q^\times} g(a) \overline{\chi}(a)\right| \leq |\mathbb{E}_{n \in [I]_q} f(n) \overline{\chi}(n)|.
\]
\item We have $\mathbb{E}_{a \in \mathbb{Z}_q^{\times}} g(a) = \mathbb{E}_{n \in [I]_q} f(n)$.
\item
Let $H \leq \mathbb{Z}_q^{\times}$ be a subgroup of index $2$. Then, for any $b \in \mathbb{Z}_q^{\times}$, we have
\[
\mathbb{E}_{\substack{n \in [I]_q}} f(n) \mathbf{1}_{n \in bH} = \mathbb{E}_{a \in \mathbb{Z}_q^\times} g(a) \mathbf{1}_{a \in bH} + O(\delta).
\]
\end{enumerate}
\end{proposition}

\begin{proof}
This is a slight variant of~\cite[Proposition 4.1]{MatoTeraEkq} --- there we had $I = [1, N]$, but exactly the same proof works here. Furthermore, (v) here corresponds to a special case of~\cite[Proposition 4.1(v)]{MatoTeraEkq}. 
\end{proof}

\subsection{Sieves, Burgess' bound and products of primes in cosets} \label{ssec:sievesetc}
In this subsection, our main goal is to prove Lemma~\ref{le:P2}, which gives a lower bound for the number of rough numbers in any index $2$ coset. For proving this, we need the Burgess bound.
\begin{lemma}
\label{le:Burgess}
Let $q \in \mathbb{N}$, let $\chi$ be a non-principal character $\pmod{q}$ and let $M, N \geq 1$. Then, for any $\varepsilon > 0$ and $r \in \{1, 2, 3\}$, we have
\[
\left|\sum_{M < n \leq M+N} \chi(n)\right| \ll_{\varepsilon} N^{1-\frac{1}{r}} q^{\frac{r+1}{4r^2} + \varepsilon}.
\]
If $q$ is cube-free or $\chi$ has bounded order, then this holds for any $r \in \mathbb{N}$.

In particular, for any $\varepsilon > 0$, there is $\delta = \delta(\varepsilon) > 0$ such that for any $q \in \mathbb{N}$ and any $N \geq q^{1/3+\varepsilon}$ we have
\[
\left|\sum_{n \leq N} \chi(n)\right| \ll_{\varepsilon} N^{1-\delta}.
\]
When $q$ is cube-free or $\chi$ has bounded order, this holds for $N \geq q^{1/4+\varepsilon}$.
\end{lemma}

\begin{proof}
For $r = 1$, the first part follows from the P{\'o}lya--Vinogradov inequality (see e.g.~\cite[Theorem 12.5]{iw-kow}). The case $r \geq 2$ of the first part is the Burgess bound; see e.g.~\cite[Theorem 12.5]{iw-kow} for the general and cube-free case and~\cite[Lemma 2.4]{Heath-Brown2} for the bounded order case. These are stated for primitive characters, but the case of non-principal $\chi \pmod q$ follows by writing $\chi(n)=\mathbf{1}_{(n,q)=1}\chi'(n)$ for some primitive character $\chi'$ and using M\"obius inversion on $\mathbf{1}_{(n,q)=1}$.

The second part of the claim follows from the first by taking $r=3$ in the case of arbitrary $q$, and by taking $r$ large in case $q$ is cube-free or $\chi$ has bounded order.
\end{proof}

We shall need in several places the fundamental lemma of the sieve, which we now state.
\begin{lemma}[Fundamental lemma of the sieve]
\label{le:fundsieve}
Let $\kappa \geq 1$ be fixed. Let $z\geq 2$ and let $D = z^s$ with $s\geq 9\kappa + 1$. There exist coefficients $\lambda_d^\pm \in \mathbb{R}$ such that the following hold.
\begin{enumerate}[(i)]
\item $|\lambda_d^\pm| \leq 1$ for every $d \in \mathbb{N}$ and $\lambda_d^\pm$ are supported on $\{d \leq D:\,\, d \mid P(z)\}$.
\item For every $n \in \mathbb{N}$,
\[
\sum_{d \mid n} \lambda^-_d \leq \mathbf{1}_{(n, P(z)) = 1} \leq \sum_{d \mid n} \lambda^+_d. 
\]
\item If $g \colon \mathbb{N} \to [0, 1)$ is a multiplicative function such that, for some $K \geq 1$, one has
\begin{align}\label{eq:Kbound}
\prod_{w_1 \leq p < z_1} (1-g(p))^{-1} \leq K \left(\frac{\log z_1}{\log w_1}\right)^\kappa
\end{align}
for any $z_1 \geq w_1 \geq 2$, then we have
\begin{align*}
\sum_{d \mid P(z)} \lambda_d^+ g(d) &\leq (1+e^{9\kappa-s} K^{10}) \prod_{p < z} (1-g(p)), \\
\sum_{d \mid P(z)} \lambda_d^- g(d) &\geq (1-e^{9\kappa-s} K^{10}) \prod_{p < z} (1-g(p)).
\end{align*}
\end{enumerate}
\end{lemma}

\begin{proof}
See e.g.~\cite[Lemma 6.8]{friedlander}.
\end{proof}

The following is a quick consequence of the fundamental lemma of the sieve and the Burgess bound.
\begin{lemma}[The number of rough numbers in cosets]
\label{le:P2}
Let $\varepsilon>0$ be sufficiently small. 
Let $q \in \mathbb{N}$ be large enough in terms of $\varepsilon$. Let $H \leq \mathbb{Z}_q^\times$ be a subgroup of index $2$, and let $b \in \mathbb{Z}_q^\times$. Then, for any $R\in [q^{0.26},q]$, we have
\[
\sum_{\substack{n \leq R \\ n \in bH \\ (n,P(q^{\sqrt{\varepsilon}}))=1}} 1 \geq \left(\frac{1}{2}-\varepsilon\right)\sum_{\substack{n \leq R \\ (n,P(q^{\sqrt{\varepsilon}}))=1}}1.
\]
\end{lemma}

\begin{proof} This is similar to the lower bound part of~\cite[Lemma 3.3]{MatoTeraEkq}. Let $\lambda_d^{-}$ be the lower bound sieve coefficients in Lemma~\ref{le:fundsieve} with $z=q^{\sqrt{\varepsilon}}$ and $D=q^{0.005}$. Then $s=0.005/\sqrt{\varepsilon}$ there.  Defining $g(d)=\frac{\mathbf{1}_{(d,q)=1}}{d}$, the condition~\eqref{eq:Kbound} holds by Mertens' theorem for some absolute constant $K\ll 1$. We may assume that $\varepsilon$ is sufficiently small in terms of $K$. 

By Lemma~\ref{le:fundsieve}(ii),
\begin{align} \label{eq:Sdef}
\sum_{\substack{n \leq R \\ n \in bH \\ (n,P(q^{\sqrt{\varepsilon}}))=1}} 1 \geq \sum_{\substack{d\leq D \\ (d, q) = 1}}\lambda_d^{-}\sum_{\substack{m \leq R/d \\ dm \in bH}} 1 =: S,    
\end{align}
say. Letting $\psi$ be the quadratic character that equals to $1$ on $H$, we have $\mathbf{1}_{n\in bH}=\mathbf{1}_{(n,q)=1}(1+\psi(\overline{b}n))/2$. Hence we have
\begin{align*}
S=\frac{1}{2}\sum_{d\leq D}\lambda_d^{-}\mathbf{1}_{(d,q)=1}\sum_{m \leq R/d} \mathbf{1}_{(m,q)=1}+\frac{\psi(\overline{b})}{2}\sum_{d\leq D}\lambda_d^{-}\psi(d)\sum_{m \leq R/d}\psi(m) =: S_1+S_2,     
\end{align*}
say. By the Burgess bound for quadratic characters (Lemma~\ref{le:Burgess}) and the fact that $R/d\geq q^{0.255}$, there exists some small absolute constant $\delta>0$ such that 
\begin{align*}
S_2\ll \sum_{d\leq D}\left(\frac{R}{d}\right)^{1-\delta} \ll Rq^{-\delta/5}.     
\end{align*}
Furthermore, by~\eqref{eq:Simplegcd} and the fundamental lemma of the sieve (Lemma~\ref{le:fundsieve}),
\begin{align*}
  \begin{aligned}
    S_1&= \frac{1}{2}\sum_{d \leq D} \lambda_d^- \mathbf{1}_{(d, q)=1} \left(\frac{R}{d} \frac{\varphi(q)}{q} + O(\tau(q))\right) =\frac{1}{2}R \frac{\varphi(q)}{q} \sum_{d\leq D}\lambda_d^{-}g(d)+O(\tau(q)D)\\
 &\geq R \frac{\varphi(q)}{q} \left(\frac{1}{2}-e^{9-s}K^{10}\right)\prod_{p<z}(1-g(p))+O(q^{0.01}).
  \end{aligned}
  \end{align*}
  Here
\[
\frac{\varphi(q)}{q}\prod_{p<z}(1-g(p)) = \prod_{p < z} \left(1-\frac{1}{p}\right) \cdot \prod_{\substack{p \mid q \\ p \geq z}} \left(1-\frac{1}{p}\right)  \geq  \left(1-\frac{1}{\sqrt{\varepsilon} q^{\sqrt{\varepsilon}}} \right) \prod_{p < z} \left(1-\frac{1}{p}\right).
\]  
Recall that $s = 0.005/\sqrt{\varepsilon}$. Hence, once $\varepsilon$ is sufficiently small in terms of $K$ and $q$ is sufficiently large in terms of $\varepsilon$ and $\delta$,
\begin{align}\label{eq:fundlem2}
  \begin{aligned}
S = S_1+S_2 &\geq \left(\frac{1}{2}-\frac{\varepsilon}{4}\right) R \prod_{p<z}\left(1-\frac{1}{p}\right).    
  \end{aligned}
  \end{align}

Using the upper bound sieve part of Lemma~\ref{le:fundsieve} similarly, we also obtain 
\begin{align}\label{eq:fundlem3}
\sum_{\substack{n \leq R \\ (n,P(q^{\sqrt{\varepsilon}}))=1}}1 \leq  \left(1+\frac{\varepsilon}{4}\right)R\prod_{p<z}\left(1-\frac{1}{p}\right),
\end{align}
and the claim follows from~\eqref{eq:Sdef},\eqref{eq:fundlem2}, and \eqref{eq:fundlem3}.
\end{proof}

\subsection{Mean and large value estimates for character sums}
Let us first state the basic mean value result for character sums.
\begin{lemma}[Mean value theorem] \label{le:MVT}
Let $q \in \mathbb{N}$ and $N \geq 2$. Then, for any complex numbers $a_n$, 
\[
\frac{1}{\varphi(q)} \sum_{\chi \pmod{q}} \left|\sum_{n \leq N} a_n \chi(n)\right|^2 \leq \left(1+\frac{N}{q}\right) \sum_{\substack{n \leq N \\ (n, q) = 1}} |a_n|^2.
\]
\end{lemma}
\begin{proof}
This is almost immediate from orthogonality, see e.g.~\cite[Theorem 6.2]{montgomery-topics}.
\end{proof}

The following lemma gives a variant of Hal{\'a}sz--Montgomery type mean value theorems that is tailored for character sums supported on numbers without small prime factors.

\begin{lemma}
\label{le:Hal-Mon}
 Let $\varepsilon>0$ and $C \geq 1$ be fixed. Let $q \in \mathbb{N}$ and let $\mathcal{X}$ be a set of Dirichlet characters of modulus $q$. 
Let $N \in [q^\varepsilon, q^C]$. Then, for any complex numbers $a_n$, we have
\[
\sum_{\chi \in \mathcal{X}} \left|\sum_{\substack{n \leq N \\ (n, P(q^\varepsilon)) = 1}} a_n \chi(n)\right|^2 \ll \left(\frac{N}{\log q} + N^{2/3} q^{\frac{1}{9}+2\varepsilon} |\mathcal{X}|\right) \sum_{\substack{n \leq N \\ (n, P(q^\varepsilon)) = 1}} |a_n|^2.
\]
\end{lemma} 
\begin{proof}
This follows similarly to~\cite[Lemma 3.8(ii)]{MatoTeraEkq}, but replacing the logarithmically weighted sums over the interval $(N_1, N_2]$ by unweighted sums over the interval $[1, N]$.
\end{proof}

The following lemma gives us an upper bound for the number of characters for which a character sum over primes is large.
\begin{lemma}\label{le:largevalue}
Let $C\geq 1$. Let $q\in \mathbb{N}$ be large, $P\in [(\log q)^C,q]$, and let $a_p$ be bounded complex numbers. For a character $\chi\pmod q$, define $P(\chi)\coloneqq \sum_{P/e<p\leq P}a_p\chi(p)$. Then, for any $\alpha\in [0,1/2]$, we have
\begin{align*}
|\{\chi\pmod q\colon |P(\chi)|\geq P^{1-\alpha}\}|\ll P^{2\alpha} q^{2\alpha+1/C+o(1)}.     
\end{align*}
\end{lemma}

\begin{proof} This follows from the proof of~\cite[Lemma 6.5]{KMT}. The only differences are that the sum is over $(P/e,P]$ rather than a dyadic interval, and that in the proof of the lemma we can use $(e^{20}k)^k\ll k^{(1+o(1))k}$ in place of $(e^{20}k)^k\ll k^{100k}$. 
\end{proof}

\subsection{Lower bounds on product sets}
The following simple lemma gives a lower bound for convolutions on a product set.

\begin{lemma}\label{le:convolution} 
Let $G$ be a finite abelian group.
\begin{enumerate}[(i)]
\item Let $A,B\subseteq G$ be nonempty subsets of $G$. Then we have
\begin{align*}
(\mathbf{1}_{A}*\mathbf{1}_{B})(c)\geq |A|+|B|-|G|  \end{align*}
for every $c\in G$.
\item Let $H \leq G$, let $a, b \in G$, and let $A \subseteq aH$ and $B \subseteq bH$. Then, for every $c \in abH$, we have
\begin{align*}
(\mathbf{1}_{A}*\mathbf{1}_{B})(c)\geq |A|+|B|-|H|.
\end{align*}
\end{enumerate}
\end{lemma}

\begin{proof}
For the quick proof, see~\cite[Lemma 3.4]{MatoTeraEkq}.    
\end{proof}

Kneser's theorem is a standard tool for studying product sets inside abelian groups. In what follows, for an abelian group $G$ and a subset $A \subset G$, the group $\{h \in G \colon hA = A\}$ is called the stabilizer of $A$.  

\begin{lemma}[Kneser's theorem]
\label{le:Kneser}
Let $G$ be a finite abelian group and let $A, B \subseteq G$. Let $H$ be the stabilizer of $A \cdot B$. Then
\[
|A \cdot B| \geq |A \cdot H| + |B \cdot H| - |H|\geq |A|+|B|-|H|.
\]
\end{lemma}

\begin{proof}
See for example~\cite[Theorem 5.5]{tao-vu}.    
\end{proof}

We shall use the following lemma, which is a quick consequence of work of Grynkiewicz~\cite{Grynkiewicz}, allowing us to reduce our need for ``popular Kneser'' to the usual Kneser theorem.

\begin{lemma}
\label{le:popularkneser}
Let $t\geq u\geq 1$ be integers. Let $A,B$ be  subsets of a finite abelian group $G$ with $|A|,|B|\geq t$. Then at least one of the following holds.
\begin{enumerate}[(a)]
\item We have \begin{align*}
(\mathbf{1}_{A}*\mathbf{1}_{B})(a)\geq u    
\end{align*}
for at least
\begin{align*}
|A|+|B| - 2 t -\frac{u|G|}{t}   \end{align*}
elements $a\in G$.
\item There exist subsets $A'\subseteq A, B'\subseteq B$ with
\begin{align*}
|A\setminus A'|+|B\setminus B'|\leq t-1    
\end{align*}
such that  
\[
(\mathbf{1}_{A}*\mathbf{1}_{B})(a)\geq t
\]
for every $a \in A' \cdot B'$.
\end{enumerate}
\end{lemma}

\begin{proof}
This is~\cite[Lemma 6.1]{MatoTeraEkq}.    
\end{proof}

The following consequence of Kneser's theorem tells us about the structure of $A$ and $B$ in case $A\cdot B$ is small.
\begin{lemma}
\label{le:structure}
Let $\alpha, \alpha', \beta \in (0, 1]$ be such that $\beta < 2\alpha\leq 2\alpha'$,  and let $A, B \subseteq \mathbb{Z}_q^\times$ with $|A|, |B| \geq \alpha \varphi(q)$. Assume that $A$ and $B$ each meet at least proportion $\alpha'$ of cosets of any subgroup $H_0 \leq \mathbb{Z}_q^\times$ of index $<1/(2\alpha-\beta)$. Then at least one of the following holds.
\begin{enumerate}[(a)]
\item We have 
\[
|A \cdot B| \geq \beta \varphi(q).
\]
\item Write $H \leq \mathbb{Z}_q^\times$ for the stabilizer of $A \cdot B$ and write $Y$ for its index. Then 
\[
1 < Y < \frac{1}{2\alpha' - \beta}.
\]
\end{enumerate}
\end{lemma}

\begin{proof} This follows immediately from~\cite[Lemma 6.3]{MatoTeraEkq}.
\end{proof}

Combining the previous three lemmas, we prove the following lemma.
\begin{lemma}[A lower bound on triple convolutions] \label{le:KneserAppl}
Let $\varepsilon > 0$, and let $q\in \mathbb{N}$ be large enough in terms of $\varepsilon$. Let $A_1,A_2,A_3$ be subsets of  $\mathbb{Z}_q^{\times}$ with $|A_1|, |A_2|, |A_3| > \left(2/5+\varepsilon\right) \varphi(q)$. Then at least one of the following holds.
\begin{enumerate}[(a)]
\item For every $a \in \mathbb{Z}_q^\times$ we have
\[
(\mathbf{1}_{A_1} \ast \mathbf{1}_{A_2}\ast \mathbf{1}_{A_3})(a) \geq \frac{1}{500}\varepsilon^2 \varphi(q)^2.
\]
\item There exists a subgroup $H \leq \mathbb{Z}_q^{\times}$ of index $2$ and elements $a_i \in A_i$ for $i\in \{1,2,3\}$ such that
\[
|A_i \cap a_iH| \geq |A_i| - \frac{\varepsilon}{2}\varphi(q)
\]
and
\[
(\mathbf{1}_{A_1} \ast \mathbf{1}_{A_2}\ast \mathbf{1}_{A_3})(a) \geq \frac{1}{25} \varphi(q)^2
\]
for every $a \in a_1 a_2a_3 H$. 
\end{enumerate}
\end{lemma}
\begin{proof} Let $t=\lceil \varepsilon \varphi(q)/10\rceil$, $u=\lceil \varepsilon^2\varphi(q)/100\rceil$; then $t\geq u\geq 1$. We apply Lemma~\ref{le:popularkneser} to the sets $A_1, A_2$. We split into cases, showing that in each case either claim (a) or (b) of Lemma~\ref{le:KneserAppl} always holds.

  \textbf{Case 1: Lemma~\ref{le:popularkneser}(a) holds.} Now there exists a set $T\subseteq\mathbb{Z}_q^{\times}$ such that
\begin{align*}
\mathbf{1}_{A_1}*\mathbf{1}_{A_2}\geq u\cdot \mathbf{1}_{T},\quad\textnormal{with}\quad  |T|\geq |A_1|+|A_2|-\frac{\varepsilon}{2}\varphi(q).   \end{align*}
Then by Lemma~\ref{le:convolution}(i) we have, for every $a \in \mathbb{Z}_q^\times$,
\begin{align*}
(\mathbf{1}_{A_1}*\mathbf{1}_{A_2}*\mathbf{1}_{A_3})(a)&\geq u\cdot (\mathbf{1}_{T}*\mathbf{1}_{A_3})(a)\\
                                                       &\geq u(|T|+|A_3|-\varphi(q)) \\
  &\geq u\left(|A_1|+|A_2|+|A_3|-\left(1+\frac{\varepsilon}{2}\right)\varphi(q)\right)\\
&\geq\frac{1}{5} u \varphi(q) \geq \frac{1}{500}\varepsilon^2\varphi(q)^2
\end{align*}
and thus claim (a) holds.

\textbf{Case 2: Lemma~\ref{le:popularkneser}(b) holds.} Now there exist sets $A_1'\subseteq A_1$, $A_2'\subseteq A_2$ such that 
\begin{align}\label{eq:Ait}
|A_i'|\geq |A_i|-t\geq |A_i|-\varepsilon \varphi(q)/2    
\end{align}
for $i\in \{1,2\}$ and 
\begin{align} \label{eq:tripleconvlow}
(\mathbf{1}_{A_1}*\mathbf{1}_{A_2}*\mathbf{1}_{A_3})(a)\geq t\cdot (\mathbf{1}_{A_1'A_2'}*\mathbf{1}_{A_3})(a).     
\end{align}
We split into two cases.

\textbf{Case 2.1: $|A_1'A_2'| \geq (\frac{3}{5}-\frac{\varepsilon}{2})\varphi(q)$.} 

By Lemma~\ref{le:convolution}, the right-hand side of~\eqref{eq:tripleconvlow} is, for every $a \in \mathbb{Z}_q^\times$,
\begin{align*}
\geq t(|A_1'A_2'|+|A_3|-\varphi(q))\geq t\cdot \frac{\varepsilon}{2}\varphi(q)\geq \frac{1}{20}\varepsilon^2\varphi(q)^2    
\end{align*}
and thus claim (a) holds.

\textbf{Case 2.2: $|A_1'A_2'| < (\frac{3}{5}-\frac{\varepsilon}{2})\varphi(q)$.} 
Let $S \leq \mathbb{Z}_q^\times$ be the stabilizer of $A_1'A_2'$, and let $Y \coloneqq [\mathbb{Z}_q^\times : S]$ be its index. We plan to apply Lemma~\ref{le:structure} to the sets $A_1', A_2'$ with $\beta=3/5-\varepsilon/2$, $\alpha=2/5+\varepsilon/2$, and $\alpha'=1/2$. To check its assumption, note that $\frac{1}{2\alpha - \beta} = \frac{1}{\frac{1}{5}+3\varepsilon/2} < 5$; now since $|A_1'|, |A_2'| \geq (2/5+\varepsilon/2)\varphi(q)$, trivially $A_1'$ and $A_2'$ must meet at least proportion $\alpha' = 1/2$ of cosets of any subgroup $H_0 \leq \mathbb{Z}_q^\times$ of index $< 5$. Hence Lemma~\ref{le:structure} is applicable. By the assumption of Case 2.2 we must be in case (b) and thus $1<Y<\frac{1}{2\cdot 1/2-3/5+\varepsilon/2}<3$. Hence we have $Y=2$, so $A_1'A_2'$ is a coset of some index $2$ subgroup $H'$. This implies that $A_1'$ and $A_2'$ are contained in some cosets $a_1'H'$ and $a_2'H'$ of $H'$.

By a symmetric argument with $A_1,A_3$ in place of $A_1,A_2$, we see that again either claim (a) of the lemma holds or there exist sets $A_1'' \subseteq A_1$ and $A_3'' \subseteq A_3$ such that 
\begin{align}\label{eq:Aippt}
|A_i''|\geq |A_i|-t\geq |A_i|-\varepsilon \varphi(q)/2    
\end{align}
for $i\in \{1,3\}$ and 
\begin{align*}
(\mathbf{1}_{A_1}*\mathbf{1}_{A_2}*\mathbf{1}_{A_3})(a)\geq t\cdot (\mathbf{1}_{A_1''A_3''}*\mathbf{1}_{A_2})(a),  
\end{align*}
and there exists an index $2$ subgroup $H''$ such that $A_1'', A_3''$ are contained in some cosets $a_1''H''$ and $a_3''H''$ of $H''$. But then $A_1' \cap A_1''$ is contained in $a_1'H'\cap a_1''H''$, so if $H'\neq H''$ we have $|A_1' \cap A_1''|\leq |a_1'H'\cap a_1''H''|=\frac{\varphi(q)}{4}$. But this contradicts the fact that $|A_1' \cap A_1''|\geq |A_1|-\varepsilon\varphi(q)\geq 2\varphi(q)/5$. Hence we must have $H'=H''$.

Hence $H'=H''$, and therefore there exist elements
$$
a_1',a_2',a_3'\in \mathbb{Z}_q^{\times}
$$
such that
$$
A_1' \subseteq a_1'H', \qquad A_2' \subseteq a_2'H', \qquad A_3'' \subseteq a_3'H'.
$$
Using~\eqref{eq:Ait} and~\eqref{eq:Aippt}, we obtain
\begin{align*}
|A_i \cap a_i'H'| &\ge |A_i'| \ge |A_i| - \frac{\varepsilon}{2}\varphi(q).
\end{align*}
 Since for any $b_i\in A_i'\cap a_i'H'$ we have $b_iH=a_i'H$, we may assume that $a_i'\in A_i$ for all $i\in \{1,2,3\}$. We can then use Lemma~\ref{le:convolution}(ii) to obtain for $a\in a_1'a_2'H'$ the bound
 \begin{align*}
 (\mathbf{1}_{A_1}*\mathbf{1}_{A_2})(a)\geq |A_1'|+|A_2'|-|H'|\geq \frac{3}{10}\varphi(q).
 \end{align*}
Hence we get, for $a\in a_1'a_2'a_3'H'$, the bound
\begin{align*}
  (\mathbf{1}_{A_1}*\mathbf{1}_{A_2}*\mathbf{1}_{A_3})(a) &= \sum_{b \in A_3} (\mathbf{1}_{A_1}*\mathbf{1}_{A_2})(a\overline{b}) \geq \sum_{b \in A_3 \cap a_3'H'} (\mathbf{1}_{A_1}*\mathbf{1}_{A_2})(a\overline{b}) \\
  &\geq |A_3 \cap a_3' H'| \cdot \frac{3}{10}\varphi(q) \geq \frac{1}{25}\varphi(q)^2.    
\end{align*}
Hence claim (b) holds.
\end{proof}

\section{Proof of Corollary~\ref{cor:Liouville} assuming Theorem~\ref{thm:general}} \label{sec:proofofCor}
In order to deduce Corollary~\ref{cor:Liouville} from Theorem~\ref{thm:general} we need the following lemma.
\begin{lemma}[Sums involving $1 \ast \psi$]\label{le:QRupper}
For each fixed but sufficiently small $\varepsilon > 0$, there exists a positive constant $c_\varepsilon$ such that the following holds. Let $q \in \mathbb{N}$ be sufficiently large in terms of $\varepsilon$ and let $\psi\pmod{q}$ be a real character. Then, for every $y \in [q^{1/3}, q]$, we have
\begin{align}\label{eq:yPq}
\sum_{\substack{n \leq y \\ (n, P(q^{\varepsilon})) = 1}} (1 \ast \psi)(n) \geq c_\varepsilon y L(1, \psi)\frac{\varphi(q)}{q} \prod_{\substack{2<p \leq q \\ \psi(p) = 1}} \left(1-\frac{2}{p}\right).
\end{align}
\end{lemma}

\begin{proof}
The proof is the same as the proof of~\cite[Lemma 9.4]{MatoTeraEkq}; for the sake of completeness we give some details. Let $\lambda_d^-$ be as in Lemma~\ref{le:fundsieve} with $\kappa = 2$, sifting parameter $z = q^{\varepsilon}$, and level $D = q^{\sqrt{\varepsilon}}$ (so that $s = 1/\sqrt{\varepsilon}$).

Now by Lemma~\ref{le:fundsieve}(ii) and \cite[Lemma 9.3]{MatoTeraEkq}, for some constant $\eta > 0$ we have
\begin{align*}
\sum_{\substack{n \leq y \\ (n, P(q^{\varepsilon})) = 1}} (1 \ast \psi)(n)&\geq  \sum_{n \leq y} (1 \ast \psi)(n) \sum_{\substack{e \mid n}} \lambda_e^- =\sum_{e\mid P(q^{\varepsilon})}\lambda_e^-\sum_{\substack{n \leq y\\ e \mid n}}(1 \ast \psi)(n) \\
&= yL(1, \psi)\sum_{e \mid P(q^{\varepsilon})}\lambda_e^- h(e) +O(y^{1-\eta}),
\end{align*}
where $h$ is a multiplicative function given on the primes by $h(p)=(1+\psi(p))/p-\psi(p)/p^2$. Now the claim follows from Lemma~\ref{le:fundsieve}(iii) and Siegel's bound $L(1,\psi)\gg_{\delta} q^{-\delta}$ since
\begin{align}
\label{eq:hprod}
\prod_{p<q^{\varepsilon}}(1-h(p))\asymp \prod_{p\leq q}(1- h(p)) \asymp\prod_{p \mid q}\left(1-\frac{1}{p}\right) \prod_{\substack{ 2<p\leq q\\ \psi(p) = 1}}\left(1-\frac{2}{p}\right)= \frac{\varphi(q)}{q}\prod_{\substack{2<p\leq q \\\psi(p) = 1}}\left(1-\frac{2}{p}\right),
\end{align}
matching the factor on the right-hand side of~\eqref{eq:yPq}.
\end{proof}

\begin{proof}[Proof of Corollary~\ref{cor:Liouville} assuming Theorem~\ref{thm:general}] Note that we can rewrite
\begin{align*}
L(q)=\max_{\substack{\chi \pmod{q}\\ \chi \textnormal{ real}}} \left\{L(1, \chi)^{-1} \prod_{p \leq q} \left(1-\frac{\chi(p)}{p}\right)^{-1}\right\}.
\end{align*}
Hence, by Siegel's theorem (see e.g.~\cite[Theorem 11.14]{MVBook}) and Mertens' theorem, we have $L(q) \ll_{\varepsilon_0} q^{\varepsilon_0/100}$ for every $\varepsilon_0 > 0$, and thus it suffices to establish that $R(\mu; q) \leq q^2 L_1(q)$, where
\begin{align*}
L_1(q)=\max\{C,L(q)^{100},B(q)\}    
\end{align*}
for some large absolute constant $C$.

Let $\varepsilon > 0$ be small but fixed and let $c_\varepsilon > 0$ be as in Lemma~\ref{le:QRupper}. Let 
\begin{equation} \label{eq:depsdef}
d_\varepsilon := \frac{c_\varepsilon \varepsilon}{500 \cdot 2^{1/\varepsilon}}.
\end{equation}
Let $\lambda$ be the Liouville function. Noting that $R(\mu;q)=R(\lambda;q)$, the claim follows from Theorem~\ref{thm:general} applied to $h=\lambda$, unless there exists a quadratic character $\chi \pmod{q}$ such that 
\begin{equation} \label{eq:corclaim}
\sum_{\substack{p \leq q^{1/2} \\ \chi(p) > 0}} \frac{1}{p} \leq \frac{d_\varepsilon}{3} \cdot  \frac{1}{L_1(q)^{1/100}}.
\end{equation}
If this holds, then
\begin{equation} \label{eq:counterassumpcorproof}
\sum_{q^\varepsilon \leq p \leq q^{1/2}} \frac{(1 \ast \chi)(p)}{p} \leq 2\frac{d_\varepsilon}{3} \frac{1}{L_1(q)^{1/100}} + \sum_{\substack{q^\varepsilon \leq p \leq q^{1/2} \\ p \mid q}} \frac{1}{p} \leq \frac{d_\varepsilon}{L_1(q)^{1/100}}.
\end{equation}
Observe that by multiplicativity $(1*\chi)(n)\leq 2^{\Omega(n)}$. Hence, for $y \in [q^{2/5}, q^{1/2}]$, we have
\[
  \sum_{n \leq y} (1 \ast \chi)(n) \mathbf{1}_{n \in \mathbb{P}} \geq \sum_{\substack{n \leq y \\ (n, P(q^{\varepsilon})) = 1}} (1 \ast \chi)(n)  - \sum_{\substack{q^{\varepsilon} \leq p \leq y^{1/2}}} (1 \ast \chi)(p) \sum_{\substack{m \leq y/p \\ (m, P(q^\varepsilon)) = 1}} 2^{1/\varepsilon}. 
\]
Notice that
\[
y L(1, \chi)\frac{\varphi(q)}{q} \prod_{\substack{2<p \leq q \\ \chi(p) = 1}} \left(1-\frac{2}{p}\right) \geq \frac{1}{10} \cdot \frac{y}{L(q) \log q} \geq \frac{1}{10} \cdot \frac{y}{L_1(q)^{1/100} \log q}.
\]
Thus, applying Lemma~\ref{le:QRupper} and a variant of~\eqref{eq:fundlem3}, we obtain
\[
  \sum_{n \leq y} (1 \ast \chi)(n) \mathbf{1}_{n \in \mathbb{P}} \geq \frac{c_\varepsilon}{20} \frac{y}{L_1(q)^{1/100} \log q}  - 2^{1/\varepsilon+1} \frac{y}{\log q^{\varepsilon}}\sum_{\substack{q^{\varepsilon} \leq p \leq q^{1/2}}} \frac{(1 \ast \chi)(p)}{p}.
\]
If now \eqref{eq:counterassumpcorproof} holds with $d_\varepsilon$ as in~\eqref{eq:depsdef}, we obtain that
\[
   \sum_{n \leq y} (1 \ast \chi)(n) \mathbf{1}_{n \in \mathbb{P}} \geq \frac{c_\varepsilon}{40} \frac{y}{L_1(q)^{1/100} \log y}.
  \]
But noting that this holds for all $y \in [q^{2/5}, q^{1/2}]$ (and that $n \mid q$ make a negligible contribution), we obtain
  \[
    \sum_{\substack{q^{2/5} < p \leq q^{1/2} \\ \chi(p) = 1}} \frac{1}{p} \geq \frac{c_\varepsilon}{400 L_1(q)^{1/100}} 
  \]
which contradicts~\eqref{eq:corclaim}.
  \end{proof}

\section{Proof of Theorem~\ref{thm:generaleasy}}\label{sec:thmgeneraleasy}

\subsection{The set-up} \label{ssec:easyset-up}

We may restrict to multiplicative functions not taking the value $0$ thanks to the assumption in Theorem~\ref{thm:generaleasy} that $h(p)\neq 0$ for $p\nmid q$. Then let $h\colon \mathbb{N}\to \mathbb{R}\setminus\{0\}$ be multiplicative. We shall look for numbers $n \equiv a \pmod{q}$ with $\sgn(h(n)) = \Delta$, where $n$ has a very specific, but convenient, shape. To formulate this, we need several definitions that will hold for this whole section.

Let $\varepsilon > 0$ be sufficiently small, let $q \in \mathbb{N}$ be sufficiently large in terms of $\varepsilon$, and let
\begin{equation} \label{eq:UQ1Mdefeasy}
  z \coloneqq q^{\sqrt{\varepsilon}}, \quad Q_1 := q^\varepsilon, \quad R \coloneqq q^{1/2}, \quad I \coloneqq (R/e, R].
\end{equation}
For $\Delta \in \{+, -\}$ and $B \subseteq \mathbb{Z}_q^\times$, define the sets
\begin{align*} 
  \mathcal{Q}_B^\Delta&:= \{p \in (Q_1/e,Q_1], \, p \in B, \, \sgn(h(p)) = \Delta\}, \\
  \mathcal{U}_B^\Delta & := \{u \leq R \colon |\mu(u)| = 1, \, u \in B, \, \sgn(h(u)) = \Delta\},
\end{align*}
and the function $f^{\Delta} \colon \mathbb{Z} \to \mathbb{R}_{\geq 0}$ by
\begin{equation}
\label{eq:fDef}
f^{\Delta}(n) :=  \prod_{\substack{p < z \\ p \nmid q}} \left(1-\frac{1}{p}\right)^{-1} \mathbf{1}_{\sgn(h(n)) = \Delta} \mathbf{1}_{(n, P(z)) = 1}\mathbf{1}_{[I]_q}(n).
\end{equation}
For $B_2, B_3 \subseteq \mathbb{Z}_q^\times $ and $\overline{\Delta} = (\Delta_1, \Delta_2, \Delta_3) \in \{+, -\}^3$, we consider the function $S_{B_2, B_3}^{\overline{\Delta}} \colon \mathbb{Z}_q^\times \to \mathbb{R}_{\geq 0}$ defined by
\begin{align} \label{eq:SB2B3def}
  S^{\overline{\Delta}}_{B_2, B_3}(a)&\coloneqq \frac{1}{S}\sum_{\substack{n \equiv a \pmod{q}}} \left(f^{\Delta_1} \ast f^{\Delta_1} \ast f^{\Delta_1} *\1[\mathcal{Q}_{B_2}^{\Delta_2}]*\1[\mathcal{U}_{B_3}^{\Delta_3}]\right)(n),
\end{align}
with
\[
  S := |[I]_q|^3\cdot Q_1R \asymp \left(\frac{\varphi(q)}{q}\right)^3 q^2 Q_1.
\]
\begin{remark}\label{rem:ShapeS}
Notice that if a square-free natural number $n$ is counted by $S_{B_2, B_3}^{\overline{\Delta}}(a)$, then $n\equiv a\pmod q, \, \sgn(h(n)) = \Delta_1^3 \Delta_2 \Delta_3 = \Delta_1 \Delta_2 \Delta_3,$ and 
$$n = r_1r_2r_3 \cdot p \cdot u \leq q^{2+\varepsilon},$$
where, for $j = 1, 2,3$, 
\begin{itemize}
    \item $r_j\in [I]_q$ and  $(r_j,P(z))=1$;
        \item $p \in (Q_1/e, Q_1]$ is a prime;
      \item $u$ is a square-free integer with $u \leq R$.
      \end{itemize}
The different convolution factors in~\eqref{eq:SB2B3def} serve different purposes. We will be able to apply the dense model theorem (Lemma~\ref{prop:f=g+h}) to the functions $f^{\Delta_1}$ and replace them by dense functions $g^{\Delta_1}$ which makes lower bounds for the convolution much easier. In order to rigorously do this replacement in Lemma~\ref{lemma:le:STcomparisonEasy} via estimates for character sums and their means, we will take advantage of the short prime factor $p \in \mathcal{Q}_{B_2}^{\Delta_2}$ (see Remark~\ref{re:Smallprole}). The factor $u \in \mathcal{U}_{B_3}^{\Delta_3}$ is used to guarantee that we find numbers with both signs of $h(n)$, utilizing Lemma~\ref{lemma:le:multfunctsigns}. 
\end{remark}

Let us first quickly show that it suffices to find a sufficiently good lower bound for $S_{B_2, B_3}^{\overline{\Delta}}$.
\begin{lemma} \label{le:reductiontoBoundingSEasy}
Let $\varepsilon > 0$ be sufficiently small and let $q \in \mathbb{N}$ be sufficiently large in terms of $\varepsilon$.
Let $\Delta \in \{+, -\}$ and $a \in \mathbb{Z}_q^\times$. Assume that there exist $B_2, B_3 \subseteq \mathbb{Z}_q^\times$ and $\overline{\Delta} = (\Delta_1, \Delta_2, \Delta_3) \in \{+, -\}^3$ such that $\Delta_1 \Delta_2 \Delta_3 = \Delta$ and $S_{B_2, B_3}^{\overline{\Delta}}(a) \gg \frac{1}{q^{1+\varepsilon/100}\log q}$. Then $a \in E_h^{\Delta}(q^{2+\varepsilon})$.
\end{lemma}
\begin{proof}
  Let $B_2, B_3$ and $\overline{\Delta}$ be as in the statement. By Remark~\ref{rem:ShapeS} it suffices to show that the contribution of non-square-free numbers to $S_{B_2, B_3}^{\overline{\Delta}}(a)$ is $o(\frac{1}{q^{1+\varepsilon/2}})$. If an integer $n$ counted by $S_{B_2, B_3}^{\overline{\Delta}}(a)$ is not square-free, it must be divisible by a prime square $p^2$ with $p \in (Q_1/e, q^{1/2}]$. Such $n$ contribute to $S_{B_2, B_3}^{\overline{\Delta}}(a)$ at most
      \[
        \ll \frac{1}{S} \sum_{Q_1/e \leq p \leq q^{1/2}} \sum_{\substack{m \leq q^{2+\varepsilon}/p^2 \\ mp^2 \equiv a \pmod{q}}} \tau_5(m) \ll \frac{1}{S} \sum_{Q_1/e \leq p \leq q^{1/2}} \frac{1}{q} \cdot \frac{q^{2+\varepsilon}}{p^2} \log^{O(1)}{q} = o\left(\frac{1}{q^{1+\varepsilon/2}}\right),
      \]
      and the claim follows.
\end{proof}

\subsection{Applying the dense model theorem}
We shall apply Proposition~\ref{prop:f=g+h} with $r = 2$,
\begin{equation}
\label{eq:deltadef}
\delta := \frac{1}{\log^{1/4} q},
\end{equation}
$N = q^{1/2}$, $C=O(1)$, and the functions $f^{\pm}$ defined in~\eqref{eq:fDef}. Let $D = q^{1/100},$ recall that $z = q^{\sqrt{\varepsilon}}$, and let $\lambda_d^+$ be the upper bound sieve coefficients from Lemma~\ref{le:fundsieve} with these parameters and $\kappa = 1$ (and $s = 1/(100\sqrt{\varepsilon})$). For $\Delta \in \{+, -\}$, the function $f^\Delta$ has a majorant $\nu \colon \mathbb{Z} \to \mathbb{R}_{\geq 0}$ with
\begin{equation}
\label{eq:nuDef}
\nu(n) := \prod_{\substack{p < z \\ p \nmid q}} \left(1-\frac{1}{p}\right)^{-1} \sum_{\substack{d \mid n \\ d \leq D}} \lambda_d^+\cdot \mathbf{1}_{[I]_q}(n).
\end{equation}

Let us show that our choices satisfy Proposition~\ref{prop:f=g+h}(A1, A2) as in~\cite[Proof of Proposition 5.2]{MatoTeraEkq}.

\textbf{Verification of Proposition~\ref{prop:f=g+h}(A2):}
Write
\[
  \mathcal{X}^{\Delta} := \left\{\chi \pmod{q} \colon \left|\mathbb{E}_{n \in [I]_q} f^\Delta(n) \overline{\chi}(n)\right| \geq \delta\right\}.
  \]
  By Lemma~\ref{le:Hal-Mon} and an upper bound sieve,
  \begin{align*}
    |\mathcal{X}^\Delta| \cdot \delta^2 &\leq \sum_{\chi \in \mathcal{X}} \left|\mathbb{E}_{n \in [I]_q} f^\Delta(n) \overline{\chi}(n)\right|^2 \leq \left(\frac{R}{\log q} + R^{2/3}q^{1/9+2\varepsilon}|\mathcal{X}^\Delta|\right)\frac{1}{\left(R \frac{\varphi(q)}{q}\right)^2} \sum_{n \in [I]_q} |f^\Delta(n)|^2 \\
&\ll \left(\frac{1}{\log q} + q^{-1/18+2\varepsilon}|\mathcal{X}^\Delta|\right) \log q.
  \end{align*}
Now the second term on the right-hand side cannot dominate by~\eqref{eq:deltadef} and thus $|\mathcal{X}^\Delta| \ll \delta^{-2}$ as claimed.

\textbf{Verification of Proposition~\ref{prop:f=g+h}(A1):} 
For any character $\chi \pmod{q}$, we have
\begin{align}
\label{eq:nuFT}
\mathbb{E}_{n \in [I]_q} \nu(n) \overline{\chi}(n) 
&=   \prod_{\substack{p < z \\ p \nmid q}} \left(1-\frac{1}{p}\right)^{-1} \cdot \mathbb{E}_{\substack{n \in [I]_q}} \overline{\chi}(n) \sum_{\substack{d \mid n \\ d \leq D}} \lambda_d^+ \\
\nonumber
&=   \prod_{\substack{p < z \\ p \nmid q}} \left(1-\frac{1}{p}\right)^{-1} \sum_{d\leq D} \lambda_d^+ \overline{\chi}(d) \frac{1}{|[I]_q|} \sum_{\substack{R/(ed) < m \leq R/d}} \overline{\chi}(m).
\end{align}
Here $R/d \geq q^{2/5}$, say. Hence when $\chi \neq \chi_0$, the Burgess bound (Lemma~\ref{le:Burgess}) gives that the innermost sum is $O(q^{-2\delta_0}R/d)$ for some absolute constant $\delta_0> 0$ and thus we have that, for any $\chi \neq \chi_0$,
\[
\left|\mathbb{E}_{n \in [I]_q} \nu(n) \overline{\chi}(n)\right| \ll q^{-\delta_0}.
\]

On the other hand, by~\eqref{eq:nuFT},~\eqref{eq:Simplegcd} and the fact that $\tau(q)\ll q^{1/2000}$, we have
\begin{equation*}
\begin{aligned}
\mathbb{E}_{n \in [I]_q} \nu(n) &= \mathbb{E}_{n \in [I]_q} \nu(n)\overline{\chi_0}(n) =   \prod_{\substack{p < z \\ p \nmid q}} \left(1-\frac{1}{p}\right)^{-1} \sum_{\substack{d \leq D \\ (d,q) = 1}} \lambda_d^+ \frac{1}{|[I]_q|} \sum_{\substack{R/(ed) < n \leq R/d \\ (n, q) = 1}} 1 \\
&= \prod_{\substack{p < z \\ p \nmid q}} \left(1-\frac{1}{p}\right)^{-1} \sum_{\substack{d \leq D \\ (d,q) = 1}} \frac{\lambda_d^+}{d} + O\left(\frac{Dq^{1/1000}}{R}\right).
\end{aligned}
\end{equation*}
By Lemma~\ref{le:fundsieve}(iii), we see that
\begin{equation*}
\sum_{\substack{d \leq D \\ (d,q) = 1}} \frac{\lambda_d^+}{d} =  (1+O(\exp(-1/(100\sqrt{\varepsilon}))) \prod_{\substack{p < z \\ p \nmid q}}  \left(1-\frac{1}{p}\right) 
\end{equation*}
and thus, once $\varepsilon$ is sufficiently small,
\[
\left|\mathbb{E}_{n \in [I]_q} \nu(n)  - 1\right| \leq \varepsilon^2/2.
\]
We deduce that Proposition~\ref{prop:f=g+h}(A1) holds for $\eta = \varepsilon^2/2$ and $\varepsilon = \delta_0$.

Having established Proposition~\ref{prop:f=g+h}(A1, A2), we may apply Proposition~\ref{prop:f=g+h} and make the following definition.
\begin{definition} \label{def:gDeltadef}
For $\Delta \in \{+, -\}$, let $g^{\Delta} \colon \mathbb{Z}_q^\times \to [0, 1+\varepsilon^2]$ be the function obtained from Proposition~\ref{prop:f=g+h} with $r = 2$ and $\delta, f^\Delta,$ and $\nu$ as in~\eqref{eq:deltadef},~\eqref{eq:fDef} and~\eqref{eq:nuDef}. 
\end{definition}

In order to lower bound $S_{B_2, B_3}^{\overline{\Delta}}(a)$, we now compare the function $S^{\overline{\Delta}}_{B_2, B_3}$ with the function  $T_{B_2, B_3}^{\overline{\Delta}} \colon \mathbb{Z}_q^\times \to \mathbb{R}_{\geq 0}$ defined by
\begin{align} \label{eq:TB2B3def}
  T_{B_2, B_3}^{\overline{\Delta}}(a)&\coloneqq \frac{1}{T}\left(g^{\Delta_1} \ast g^{\Delta_1} \ast g^{\Delta_1} * \1[\mathcal{Q}_{B_2}^{\Delta_2}] * \1[\mathcal{U}_{B_3}^{\Delta_3}]\right)(a),
\end{align}
with
\[
  T := |\mathbb{Z}_q^\times|^3 \cdot Q_1 R,
  \] 
where we abuse notation by identifying a function $f\colon \mathbb{Z}\to \mathbb{C}$ supported on numbers coprime to $q$ with the function $a\mapsto \sum_{n\equiv a\pmod q}f(n)$ on $\mathbb{Z}_q^{\times}$.  

\begin{lemma}\label{lemma:le:STcomparisonEasy}
Let $\varepsilon > 0$ be sufficiently small and let $q \in \mathbb{N}$ be sufficiently large in terms of $\varepsilon$. Let $B_2, B_3 \subseteq \mathbb{Z}_q^\times $, let $\overline{\Delta} = (\Delta_1, \Delta_2, \Delta_3) \in \{+, -\}^3$, and let $S_{B_2, B_3}^{\overline{\Delta}}$ and $T_{B_2, B_3}^{\overline{\Delta}}$ be as in~\eqref{eq:SB2B3def} and~\eqref{eq:TB2B3def}. Then, for all $a \in \mathbb{Z}_q^\times$,
 \begin{align*}
|S_{B_2, B_3}^{\overline{\Delta}}(a) - T_{B_2, B_3}^{\overline{\Delta}}(a)| \ll \frac{1}{q^{1+\varepsilon/50}} + \frac{1}{\varphi(q)\log^{5/4} q} \cdot \frac{|\mathcal{U}_{B_3}^{\Delta_3}|}{R}.
  \end{align*}
\end{lemma}

\begin{proof}
For $\Delta \in \{+, -\}$, define
\begin{align*}
  F^{\Delta}(\chi)&:=\mathbb{E}_{n\in [I]_q} f^{\Delta}(n)\overline{\chi}(n) \quad \text{and} \quad G^\Delta(\chi):=\mathbb{E}_{b \in \mathbb{Z}_q^\times} g^{\Delta}(b)\overline{\chi}(b),
\end{align*}
and, for $\Delta \in \{+, -\}$ and $B \subseteq \mathbb{Z}_q^\times$, define
\begin{align*}
Q_B^\Delta(\chi):= \frac{1}{Q_1} \sum_{p\in \mathcal{Q}_B^\Delta} \overline{\chi}(p), \quad \text{and} \quad U_B^\Delta(\chi):=\frac{1}{R}\sum_{u \in \mathcal{U}_B^\Delta} \overline{\chi}(u).
\end{align*}

By orthogonality of characters and Proposition~\ref{prop:f=g+h}(iii), we have 
\begin{align*}
&\phantom{=} S_{B_2, B_3}^{\overline{\Delta}}(a) - T_{B_2, B_3}^{\overline{\Delta}}(a)\\
&=\frac{1}{\varphi(q)} \sum_{\chi \pmod{q}}\chi(a) F^{\Delta_1}(\chi)^3\cdot   Q_{B_2}^{\Delta_2}(\chi) U_{B_3}^{\Delta_3}(\chi)\\
& \quad - \frac{1}{\varphi(q)} \sum_{\chi \pmod{q}}\chi(a) G^{\Delta_1}(\chi)^3\cdot  Q_{B_2}^{\Delta_2}(\chi) U_{B_3}^{\Delta_3}(\chi)\\
&= O\Biggl(\frac{1}{\varphi(q)} \sum_{\chi \pmod{q}} \left|F^{\Delta_1}(\chi)-G^{\Delta_1}(\chi)\right| \left| F^{\Delta_1}(\chi)\right|^2\cdot  |Q_{B_2}^{\Delta_2}(\chi)| |U_{B_3}^{\Delta_3}(\chi)| \Biggr). 
\end{align*}
To bound the right-hand side, we split the characters modulo $q$ into two sets:
\begin{align*}
\begin{aligned}
  \mathcal{X} &:= \{\chi \pmod{q} \colon |Q_{B_2}^{\Delta_2}(\chi)|\leq Q_1^{-1/40} \},\\
  \mathcal{Y} &:= \{\chi \pmod{q}\} \setminus \mathcal{X}.
\end{aligned}
\end{align*}

\textbf{Contribution of $\mathcal{X}$.}
Recall~\eqref{eq:UQ1Mdefeasy}. By the definition of $\mathcal{X}$, Proposition~\ref{prop:f=g+h}(iii), the Cauchy--Schwarz inequality, and the mean value theorem (Lemma~\ref{le:MVT}), we have
\begin{align*}
&\frac{1}{\varphi(q)}\sum_{\chi \in \mathcal{X}} \left|F^{\Delta_1}(\chi)-G^{\Delta_1}(\chi)\right| \left| F^{\Delta_1}(\chi)\right|^2 \cdot  |Q_{B_2}^{\Delta_2}(\chi)| |U_{B_3}^{\Delta_3}(\chi)|  \\
\ll& \frac{Q_1^{-1/40}}{\varphi(q)}\sum_{\chi \in \mathcal{X}} \left|F^{\Delta_1}(\chi)\right|^3   |U_{B_3}^{\Delta_3}(\chi)|  \\
\ll& \frac{Q_1^{-1/40}}{\varphi(q)}\left(\sum_{\chi\pmod q}|F^{\Delta_1}(\chi)|^2|U_{B_3}^{\Delta_3}(\chi)|^2\right)^{1/2}\left(\sum_{\chi\pmod q}|F^{\Delta_1}(\chi)|^4 \right)^{1/2}\\
\ll& \frac{Q_1^{-1/40}}{\varphi(q)}\left(\frac{\frac{\varphi(q)}{q}R^2+\varphi(q)}{R^2}\right)^{1/2}\left(\frac{\frac{\varphi(q)}{q}R^2+\varphi(q)}{R^2}\right)^{1/2} \log^{O(1)} q \ll \frac{1}{q^{1+\varepsilon/50}}.
\end{align*}
\textbf{Contribution of $\mathcal{Y}$.}
Recall the definition of $\delta$ from~\eqref{eq:deltadef}. Then, by Proposition~\ref{prop:f=g+h}(ii), we have
\begin{align*}
&\phantom{=}\frac{1}{\varphi(q)}\sum_{\chi \in \mathcal{Y}} |F^{\Delta_1}(\chi)-G^{\Delta_1}(\chi)| |F^{\Delta_1}(\chi)|^2 |Q_{B_2}^{\Delta_2}(\chi)| |U_{B_3}^{\Delta_3}(\chi)|\\
&\ll \frac{\log^{-1/4} q}{\varphi(q)}\sum_{\chi \in \mathcal{Y}} |F^{\Delta_1}(\chi)|^2 |Q_{B_2}^{\Delta_2}(\chi)| |U_{B_3}^{\Delta_3}(\chi)| =: \Sigma,
\end{align*}
say.

Using the trivial upper bounds
\[
|Q_{B_2}^{\Delta_2}(\chi)| \ll \frac{1}{\log q}, \quad |U_{B_3}^{\Delta_3}(\chi)| \leq U_{B_3}^{\Delta_3}(\chi_0) = \frac{|\mathcal{U}_{B_3}^{\Delta_3}|}{R},
\]
we see that
\begin{align*}
  \Sigma&\ll \frac{1}{\varphi(q)\log^{5/4} q} \frac{|\mathcal{U}_{B_3}^{\Delta_3}|}{R} \sum_{\chi \in \mathcal{Y}}|F^{\Delta_1}(\chi)|^2.
\end{align*}

By Lemma~\ref{le:largevalue}, we have
\begin{align}\label{eq:X2bound}
|\mathcal{Y}|\ll Q_1^2 q^{1/20+\varepsilon'}     
\end{align}
for every $\varepsilon' > 0$, and by~\eqref{eq:X2bound} and Lemma~\ref{le:Hal-Mon}, for every $\varepsilon' > 0$,
\begin{align} \label{eq:F2boundEasy}
\sum_{\chi \in \mathcal{Y}}|F^{\Delta_1}(\chi)|^2 \ll \frac{1}{R}\left(\frac{R}{\log q} + R^{2/3} q^{1/9+\varepsilon'} |\mathcal{Y}|\right) \log q \ll 1.
\end{align}
Hence
\begin{align*}
  \Sigma&\ll \frac{1}{\varphi(q)\log^{5/4} q} \frac{|\mathcal{U}_{B_3}^{\Delta_3}|}{R},
\end{align*}
and the claim follows by combining the contributions of the sums over $\mathcal{X}$ and $\mathcal{Y}$.
\end{proof}

\begin{remark} \label{re:Smallprole}
For this step it was crucial to have the prime factor $p \in (Q_1/e, Q_1]$. When dealing with $\mathcal{X}$ using an $L^{\infty}$ bound for the corresponding character sum $Q_{B_2}^{\Delta_2}(\chi)$, we were still left with two character sums $F^{\Delta_1}(\chi)^2$ and $F^{\Delta_1}(\chi) U_{B_3}^{\Delta_3}(\chi)$ of length $R^2 = q$ for which the mean value theorem worked excellently. On the other hand here $Q_{B_2}^{\Delta_2}(\chi)$ is sufficiently long for concluding that the set $\mathcal{Y}$ is small (see~\eqref{eq:X2bound}) which was crucial in~\eqref{eq:F2boundEasy}. In the proof of Theorem~\ref{thm:general}, we need to use the Matom\"aki--Radziwi{\l\l} method~\cite{matomaki-radziwill} to make a ladder from a sufficiently small prime to a sufficiently large prime.
\end{remark}

\subsection{Working with the dense model}\label{sec:DenseModelWorkEasy}
It will be convenient to work with a subset of $\mathbb{Z}_q^\times$ rather than the dense model function $g^\Delta$. To facilitate this, we make the following definition.
\begin{definition} \label{def:ADelta}
For $\Delta \in \{+, -\}$, let $g^\Delta$ be as in Definition~\ref{def:gDeltadef}. Define
\begin{align*}
A^{\Delta} \coloneqq \{a \in \mathbb{Z}_q^{\times} \colon |g^{\Delta}(a)| \geq \varepsilon^2\}.
\end{align*}
\end{definition}
The following lemma gives us fundamental information about the sets $A^\pm$.
\begin{lemma} \label{le:Aprop} Let $\varepsilon > 0$ be sufficiently small and let $q \in \mathbb{N}$ be sufficiently large in terms of $\varepsilon$. Let $\Delta \in \{+, -\}$ and let $A^\Delta$ be as in Definition~\ref{def:ADelta}. 
\begin{enumerate}[(i)]
\item We have
\begin{align*}
|A^{+}| + |A^{-}|\geq \left(1-\varepsilon\right)\varphi(q).
\end{align*}
\item For any subgroup $H \leq \mathbb{Z}_q^{\times}$ of index at most $2$ and any $b\in \mathbb{Z}_q^{\times}$, we have
\begin{equation}
\label{eq:A'binprimes1}
|A^{\Delta} \cap bH|\geq \left(\frac{|\{n \in [I]_q, (n, P(z)) = 1, \sgn(h(n)) = \Delta\} \cap bH|}{|\{n \in [I]_q, (n, P(z)) = 1\}|}-\varepsilon\right)\varphi(q).
\end{equation}
\end{enumerate}
\end{lemma}

\begin{proof}
\textbf{Claim (i):} By the definition of $A^\Delta$ and the range of $g^\Delta$, for $\Delta \in \{+, -\}$,
\begin{equation} \label{eq:gRAupper}
\mathbb{E}_{a \in \mathbb{Z}_q^\times} g^{\Delta}(a) = \frac{1}{\varphi(q)}\left(\sum_{a \in \mathbb{Z}_q^\times \setminus A^{\Delta}} g^{\Delta}(a) + \sum_{a \in A^{\Delta}} g^{\Delta}(a)\right) \leq \varepsilon^2 + \frac{|A^{\Delta}|}{\varphi(q)}(1+\varepsilon^2).
\end{equation}
Furthermore, the fundamental lemma of the sieve (Lemma~\ref{le:fundsieve}) gives
\begin{equation}
\label{eq:f++f-}
\mathbb{E}_{n \in [I]_q} (f^+(n) + f^{-}(n)) = \prod_{\substack{p < z \\ p \nmid q}} \left(1-\frac{1}{p}\right)^{-1} \cdot \mathbb{E}_{n \in [I]_q} \mathbf{1}_{(n, P(z)) = 1} = 1 + O(\varepsilon^2).
\end{equation}
Now by Proposition~\ref{prop:f=g+h}(iv) and~\eqref{eq:f++f-},
\begin{equation*}
\mathbb{E}_{a \in \mathbb{Z}_q^\times} (g^{+}(a) + g^{-}(a))= \mathbb{E}_{n \in [I]_q} (f^{+}(n) + f^{-}(n) ) =  1 + O(\varepsilon^2),
\end{equation*}
and the claim follows by combining this with~\eqref{eq:gRAupper} and using the assumption that $\varepsilon>0$ is small.

\textbf{Claim (ii):} The definition of $f^\Delta$, the fundamental lemma of the sieve (Lemma~\ref{le:fundsieve}), Proposition~\ref{prop:f=g+h}(v), the range of $g^\Delta$, and the definition of $A^\Delta$ imply that
\begin{align*}
&\quad \frac{|\{n \in [I]_q, (n, P(z)) = 1, \sgn(h(n)) = \Delta\} \cap bH|}{|\{n \in [I]_q, (n, P(z)) = 1\}|}\\
&= \mathbb{E}_{\substack{n\in [I]_q}}\mathbf{1}_{n\in bH} f^{\Delta}(n)+O(\varepsilon^2) = \mathbb{E}_{a \in \mathbb{Z}_q^\times} \mathbf{1}_{a\in bH} g^{\Delta}(a) + O(\varepsilon^2)  \\
&\leq \left(1+\varepsilon^2\right)\frac{|A^{\Delta} \cap bH|}{\varphi(q)} + O(\varepsilon^2) \leq \frac{|A^{\Delta} \cap bH|}{\varphi(q)} + 
\varepsilon
\end{align*}
if $\varepsilon>0$ is small enough. Now~\eqref{eq:A'binprimes1} follows.
\end{proof}

The following two propositions show that Theorem~\ref{thm:generaleasy} holds assuming that triple convolutions of $\mathbf{1}[A^\Delta]$ satisfy certain conditions.
\begin{proposition}\label{prop:AAA->E1}
Let $\varepsilon > 0$ be sufficiently small, let $c > 0$, and let $q \in \mathbb{N}$ be sufficiently large in terms of $\varepsilon$ and $c$. For $\Delta \in \{+, -\}$, let $A^\Delta$ be as in Definition~\ref{def:ADelta}. Assume that the following two conditions hold.
\begin{itemize}
\item[(A1)] There exists a sign $\Delta_1 \in \{+, -\}$ such that
\[
\left(\mathbf{1}[A^{\Delta_1}] \ast \mathbf{1}[A^{\Delta_1}] \ast \mathbf{1}[A^{\Delta_1}]\right)(b) \geq c \varphi(q)^2
\]
for every $b \in \mathbb{Z}_q^\times.$
\item[(A2)] We have
\[
\sum_{\substack{p \leq q^{1/3} \\ h(p) < 0}} \frac{1}{p} \geq \frac{c}{q^{\varepsilon/100}}.
\] 
\end{itemize}
Then $E_h^+(q^{2+\varepsilon}) = E_h^-(q^{2+\varepsilon}) = \mathbb{Z}_q^\times$.
\end{proposition}

\begin{proof} Let $G=\mathbb{Z}_q^{\times}$. Let $\Delta_1 \in \{+, -\}$ be as in (A1). By Lemmas~\ref{le:reductiontoBoundingSEasy} and~\ref{lemma:le:STcomparisonEasy}, it suffices to show that, for every $\Delta \in \{+, -\}$ and $a \in \mathbb{Z}_q^\times$, there exist $\Delta_2, \Delta_3 \in \{+, -\}$ such that $\Delta_1 \Delta_2 \Delta_3 = \Delta$ and, for  $\overline{\Delta} = (\Delta_1, \Delta_2, \Delta_3) \in \{+, -\}^3$,
  \begin{equation} \label{eq:TGGclaim}
T_{G, G}^{\overline{\Delta}}(a) \gg \frac{1}{q^{1+\varepsilon/100} \log q} + \frac{1}{\varphi(q)\log q} \frac{|\mathcal{U}_{G}^{\Delta_3}|}{R}.
\end{equation}
Recall that $g^{\Delta_1}(b)\geq \varepsilon^2\mathbf{1}[A^{\Delta_1}](b)$ for every $b \in \mathbb{Z}_q^\times$. By the prime number theorem and the pigeonhole principle we can choose $\Delta_2 \in \{+, -\}$ such that
\begin{equation} \label{eq:fDelta2avEasy}
\frac{|\mathcal{Q}_{G}^{\Delta_2}|}{Q_1}  \geq \frac{1}{10 \log q}.
\end{equation}

Let $\Delta \in \{+, -\}$ and $a \in \mathbb{Z}_q^\times$ be arbitrary. Choose $\Delta_3 = \Delta \cdot \Delta_1 \Delta_2$, so that $\Delta_1 \Delta_2 \Delta_3 = \Delta$. Now, for $\overline{\Delta} = (\Delta_1, \Delta_2, \Delta_3)$, we have
\begin{align*}
T^{\overline{\Delta}}_{G, G}(a) &= \frac{1}{T}\left(g^{\Delta_1} \ast g^{\Delta_1} \ast g^{\Delta_1}  * \1[\mathcal{Q}_{G}^{\Delta_2}]*\1[\mathcal{U}_{G}^{\Delta_3}]\right)(a)\\
&\geq \frac{\varepsilon^{6}}{T} \left(\mathbf{1}[A^{\Delta_1}] \ast \mathbf{1}[A^{\Delta_1}] \ast \mathbf{1}[A^{\Delta_1}]  * \1[\mathcal{Q}_{G}^{\Delta_2}]*\1[\mathcal{U}_{G}^{\Delta_3}]\right)(a) \\
& \gg \frac{1}{\varphi(q)^3Q_1R} \sum_{p \in \mathcal{Q}_{G}^{\Delta_2}} \sum_{u \in \mathcal{U}_{G}^{\Delta_3}} \left(\mathbf{1}[A^{\Delta_1}] \ast \mathbf{1}[A^{\Delta_1}] \ast \mathbf{1}[A^{\Delta_1}]\right)(a\overline{pu}).
\end{align*}
Recalling (A1) and~\eqref{eq:fDelta2avEasy}, we see that
\begin{align*}
T^{\overline{\Delta}}_{G, G}(a) \gg \frac{1}{\varphi(q)} \cdot\frac{|\mathcal{Q}_{G}^{\Delta_2}|}{Q_1} \cdot \frac{|\mathcal{U}_{G}^{\Delta_3}|}{R} \gg \frac{1}{\varphi(q) \log q} \frac{|\mathcal{U}_{G}^{\Delta_3}|}{R}.
\end{align*}
Now~\eqref{eq:TGGclaim} follows by combining this with Lemma~\ref{lemma:le:multfunctsigns}, using (A2).
\end{proof}

\begin{proposition}\label{prop:AAA->E2}
Let $\varepsilon > 0$ be sufficiently small, let $c > 0$, and let $q \in \mathbb{N}$ be sufficiently large in terms of $\varepsilon$ and $c$. For $\Delta \in \{+, -\}$, let $A^\Delta$ be as in Definition~\ref{def:ADelta}. Assume that there exists a subgroup $H \leq \mathbb{Z}_q^\times$ of index two such that the following two conditions hold.
  \begin{itemize}
  \item[(A1)] There exist elements $b^+, b^- \in \mathbb{Z}_q^\times$ with $b^+H \neq b^-H$ such that
\[
\left(\mathbf{1}[A^{+}] \ast \mathbf{1}[A^+] \ast \mathbf{1}[A^+]\right)(b) \gg \varphi(q)^2
\]
for every $b \in b^+ H$ and
\[
\left(\mathbf{1}[A^{-}] \ast \mathbf{1}[A^{-}] \ast \mathbf{1}[A^{-}]\right)(b) \gg \varphi(q)^2
\]
for every $b \in b^- H$.
\item[(A2)] Let $\chi$ be the quadratic character for which $\chi(b) = 1$ iff $b \in H$. We have
\begin{equation*}
\sum_{\substack{p \leq q^{1/3} \\ h(p)\chi(p) < 0}} \frac{1}{p} \geq \frac{c}{q^{\varepsilon/100}}.
\end{equation*}
\end{itemize}
Then $E_h^+(q^{2+\varepsilon}) = E_h^-(q^{2+\varepsilon}) = \mathbb{Z}_q^\times.$
\end{proposition}

\begin{proof}
By Lemmas~\ref{le:reductiontoBoundingSEasy} and~\ref{lemma:le:STcomparisonEasy} it suffices to show that, for every $\Delta \in \{+, -\}$ and $a \in \mathbb{Z}_q^\times$, there exists $B_2, B_3 \subseteq \mathbb{Z}_q^\times$ and $\overline{\Delta} = (\Delta_1, \Delta_2, \Delta_3) \in \{+, -\}^3$ such that $\Delta_1 \Delta_2 \Delta_3 = \Delta$ and
  \[
T_{B_2, B_3}^{\overline{\Delta}}(a) \gg \frac{1}{q^{1+\varepsilon/100}\log q} + \frac{1}{\varphi(q)\log q} \frac{|\mathcal{U}_{B_3}^{\Delta_3}|}{R}.
    \]

Let $H \leq \mathbb{Z}_q^\times$ and $b^+, b^- \in \mathbb{Z}_q^\times$ be as in the assumptions of the proposition. By the prime number theorem and the pigeonhole principle we can choose $\Delta_2 \in \{+, -\}$ and $b_2 \in \mathbb{Z}_q^\times$ such that
\begin{equation} \label{eq:QlowEasy2}
|\mathcal{Q}_{b_2 H}^{\Delta_2}| \geq \frac{Q_1}{10 \log q}.
\end{equation}

Let $\Delta \in \{+, -\}$ and $a \in \mathbb{Z}_q^\times$ be arbitrary. Let $\Delta_1 = +, \, \Delta_3 = \Delta \Delta_2, \Delta_1' = -,$ and $\Delta_3' = -\Delta \Delta_2$. Let further $\overline{\Delta} = (\Delta_1, \Delta_2, \Delta_3)$ and $\overline{\Delta'} = (\Delta_1', \Delta_2, \Delta_3')$. Now $\Delta_1 \Delta_2 \Delta_3 = \Delta_1' \Delta_2 \Delta_3' = \Delta$. Choose $b_3 = a\overline{b^+ b_2}$ and $b_3' = a\overline{b^- b_2}$. Now
\begin{align*}
  &T^{\overline{\Delta}}_{b_2 H, b_3 H}(a) + T^{\overline{\Delta'}}_{b_2 H, b'_3 H}(a) \\
  &\geq \frac{\varepsilon^{6}}{T} \left(\mathbf{1}[A^+] \ast \mathbf{1}[A^+] \ast \mathbf{1}[A^+] \ast \1[\mathcal{Q}_{b_2 H}^{\Delta_2}]*\1[\mathcal{U}^{\Delta_3}_{b_3 H}] \right)(a) \\
  & \quad + \frac{\varepsilon^{6}}{T} \left(\mathbf{1}[A^-] \ast \mathbf{1}[A^-] \ast \mathbf{1}[A^-] \ast \1[\mathcal{Q}_{b_2 H}^{\Delta_2}]*\1[\mathcal{U}^{\Delta'_3}_{b_3' H}] \right)(a) \\
  & \gg \frac{1}{\varphi(q)^3Q_1R}  \sum_{p \in \mathcal{Q}_{b_2 H}^{\Delta_2}} \sum_{u \in \mathcal{U}_{b_3 H}^{\Delta_3}} \left(\mathbf{1}[A^+] \ast \mathbf{1}[A^+] \ast \mathbf{1}[A^+]\right)(a\overline{p u}) \\
  & \quad + \frac{1}{\varphi(q)^3Q_1R} \sum_{p \in \mathcal{Q}_{b_2 H}^{\Delta_2}} \sum_{u \in \mathcal{U}_{b_3' H}^{\Delta_3'}} \left(\mathbf{1}[A^-] \ast \mathbf{1}[A^-] \ast \mathbf{1}[A^-]\right)(a\overline{p u}).
\end{align*}
On the first line on the right-hand side the sums over $p$ and $u$ are supported on $a\overline{pu} \in a\overline{b_2H}\overline{b_3H} = b^+ H$ and on the second line on the right-hand side on $a\overline{pu} \in a\overline{b_2H b_3'H} = b^- H$. Thus, by (A1) and~\eqref{eq:QlowEasy2},
\begin{align*}
  &T^{\overline{\Delta}}_{b_2 H, b_3 H}(a) + T^{\overline{\Delta'}}_{b_2 H, b_3' H}(a) \gg \frac{1}{\varphi(q) \log q} \cdot \frac{|\mathcal{U}^{\Delta_3}_{b_3 H}| + |\mathcal{U}^{\Delta_3'}_{b_3' H}|}{R}.
  \end{align*}
By Lemma~\ref{lemma:le:multfunctsigns} and (A2), we have 
\begin{align*}
|\mathcal{U}^{\Delta_3}_{b_3 H}| + |\mathcal{U}^{\Delta_3'}_{b_3' H}| = \sum_{\substack{u \leq R \\ \sgn(h(u)\chi(u)) = \Delta_3 \sgn(\chi(b_3))}} |\mu(u)| \gg \frac{\varphi(q)}{q} \cdot \frac{R}{q^{\varepsilon/100}}.
\end{align*}
Hence
\begin{align*}
  &T^{\overline{\Delta}}_{b_2 H, b_3 H}(a) + T^{\overline{\Delta'}}_{b_2 H, b_3' H}(a) \gg \frac{1}{\varphi(q) \log q} \cdot \frac{|\mathcal{U}^{\Delta_3}_{b_3 H}| + |\mathcal{U}^{\Delta_3'}_{b_3' H}|}{R} + \frac{1}{q^{1+\varepsilon/100} \log q}
  \end{align*}
and consequently either
\begin{align*}
T^{\overline{\Delta}}_{b_2 H, b_3 H}(a) \gg \frac{1}{\varphi(q)\log q} \frac{|\mathcal{U}^{\Delta_3}_{b_3 H}|}{R} + \frac{1}{q^{1+\varepsilon/100} \log^2 q}
  \end{align*}
  or
  \begin{align*}
T^{\overline{\Delta'}}_{b_2 H, b_3' H}(a) \gg \frac{1}{\varphi(q)\log q}  \frac{|\mathcal{U}^{\Delta'_3}_{b_3' H}|}{R} + \frac{1}{q^{1+\varepsilon/100} \log^2 q}
  \end{align*}
and the claim follows.
\end{proof}

\subsection{Finishing the proof of Theorem~\ref{thm:generaleasy}}
\begin{proof}[Proof of Theorem~\ref{thm:generaleasy}]
  If~\eqref{eq:thmConditionEasyInThm} holds for some character $\chi$ of order at most two, there is nothing to prove. If it does not, then Proposition~\ref{prop:AAA->E1}(A2) and Proposition~\ref{prop:AAA->E2}(A2) hold. Hence Theorem~\ref{thm:generaleasy} follows if we can show that always either Proposition~\ref{prop:AAA->E1}(A1) or Proposition~\ref{prop:AAA->E2}(A1) holds.

  For $\Delta \in \{+, -\}$, let $A^\Delta$ be as in Definition~\ref{def:ADelta}. We split into three cases.
  
\textbf{Case 1:} There exist $\Delta \in \{+, -\}$ such that we have
\begin{align} \label{eq:EasyCase1}
(\mathbf{1}[A^{\Delta}] \ast \mathbf{1}[A^{\Delta}] \ast \mathbf{1}[A^{\Delta}])(b) \gg \varphi(q)^2 \quad \text{for every $b \in \mathbb{Z}_q^\times$.}
\end{align}
This means that Proposition~\ref{prop:AAA->E1}(A1) holds.

\textbf{Case 2:} There exist $\Delta \in \{+, -\}$ such that
\begin{align*}
|A^{\Delta}| \geq \left(\frac{1}{2}+\frac{1}{100}\right) \varphi(q).
\end{align*}
By Lemma~\ref{le:convolution}(i),
\[
(\mathbf{1}[A^\Delta] \ast   \mathbf{1}[A^\Delta])(c) \geq \frac{1}{50} \varphi(q)
\]
for every $c \in \mathbb{Z}_q^\times$. Thus~\eqref{eq:EasyCase1} holds, and we are actually in Case 1.

\textbf{Case 3}: We are not in Cases 1 or 2.
By Lemma~\ref{le:Aprop}(i) and the assumption that we are not in Case 2, we have, for $\Delta \in \{+, -\}$,
\begin{align*}
|A^\Delta| \geq \left(\frac{1}{2}-\frac{1}{50}\right)\varphi(q). 
\end{align*}
For $\Delta \in \{+, -\}$, apply Lemma~\ref{le:KneserAppl} with $A_1 = A_2 = A_3 = A^\Delta$. Using that we are not in Case 1, we obtain a subgroup $H^\Delta \leq \mathbb{Z}_q^\times$ of index $2$ and an element $b^\Delta \in \mathbb{Z}_q^\times$ such that
\begin{align}\label{eq:AR0EasyPlus}
|A^\Delta \cap b^\Delta H^\Delta| \geq |A^\Delta|-\frac{\varepsilon}{2}\varphi(q) 
\end{align}
and
\begin{equation*}
\left(\mathbf{1}[A^{\Delta}] \ast \mathbf{1}[A^{\Delta}] \ast \mathbf{1}[A^{\Delta}]\right)(a) \gg \varphi(q)^2 \quad \text{for every $a \in b^\Delta H^\Delta$.}
\end{equation*}

Now Proposition~\ref{prop:AAA->E2}(A1) follows once we have shown that $H^+=H^-$ and $b^+ H^+ \neq b^- H^-$. 
If either of these fails, then for $b_0 \not \in b^+ H^+,$ we have
\[
|b_0 H^+ \cap b^+ H^+| = 0 \quad \text{and} \quad |b_0 H^+ \cap b^- H^-| \in \left\{0,  \frac{\varphi(q)}{4}\right\}.
\]
Thus by~\eqref{eq:AR0EasyPlus}
\[
|A^+ \cap b_0 H^+| + |A^- \cap b_0 H^+| \leq \frac{\varepsilon}{2} \varphi(q) + \left(\frac{1}{4}+\frac{\varepsilon}{2}\right)\varphi(q) = \left(\frac{1}{4}+\varepsilon\right) \varphi(q),
\]
which by Lemma~\ref{le:Aprop}(ii) contradicts Lemma~\ref{le:P2}. Thus Proposition~\ref{prop:AAA->E2}(A1) holds.
\end{proof}

\section{More auxiliary results}\label{sec:auxil2}
The proof of Theorem~\ref{thm:general} is more involved and we need some more auxiliary results in addition to those in Section~\ref{sec:auxil1}.
\subsection{Character sums}
We need some more lemmas concerning mean values of character sums. The first two lemmas allow us to handle error terms coming from an application of a Ramar\'e-type identity.

  \begin{lemma} \label{le:squarecontr} 
Let $Q \geq P \geq 1$ and $N \geq Q^4$. For any complex numbers $\alpha_{p, m} \ll 1$, we have
    \[
     \frac{1}{\varphi(q)} \sum_{\chi \pmod{q}} \left| \sum_{\substack{p^2m \leq N \\ P < p \leq Q}} \alpha_{p, m} \overline{\chi}(p^2m)\right|^2 \ll \frac{\varphi(q)}{q}\left(N + \frac{N^2}{q} \right) \frac{1}{P}.
      \]
    \end{lemma}
    \begin{proof}
      By the mean value theorem (Lemma~\ref{le:MVT}), we have
      \begin{align*}
        &\frac{1}{\varphi(q)} \sum_{\chi \pmod{q}} \left| \sum_{\substack{p^2m \leq N\\ P<p\leq Q}} \alpha_{p, m} \overline{\chi}(p^2m)\right|^2 \ll \left(1 + \frac{N}{q}\right) \sum_{\substack{n \leq N \\ (n, q) = 1}} \left(\sum_{\substack{n = p^2 m \\ P < p \leq Q}} 1 \right)^2.
      \end{align*}
      Here
      \begin{align*}
  &\sum_{\substack{n \leq N \\ (n, q) = 1}} \left(\sum_{\substack{n = p^2 m \\ P < p \leq Q}} 1 \right)^2 \ll \sum_{\substack{P < p_1, p_2 \leq Q}} \sum_{\substack{n \leq N \\ [p_1^2, p_2^2] \mid n \\ (n, q) = 1}} 1  \ll N \frac{\varphi(q)}{q} \sum_{P < p_1, p_2 \leq Q} \frac{1}{[p_1, p_2]^2} \ll \frac{\varphi(q)}{q} \cdot \frac{N}{P},
      \end{align*}
      and the claim follows.
    \end{proof}

  \begin{lemma} \label{le:shortscontr}
Let $\varepsilon > 0$ be small and let $q \in \mathbb{N}$ be sufficiently large. Let $K\geq 0$ be an integer with $K \leq \min\{\log q, \frac{1}{4} \log M\}$. Let $H, M, N \geq 1$ with $N \geq M^4 \geq q^\varepsilon$ and $H \leq q^{\varepsilon/20}$. Let $\mathcal{I} \coloneqq \bigcup_{|k| \leq 2K+1} (Me^{k-1/H}, Me^{k}]$. For any complex numbers $\alpha_{\ell, m} \ll 1$, we have
    \[
      \frac{1}{\varphi(q)} \sum_{\chi \pmod{q}} \left|\sum_{\substack{\ell m \leq N \\ m \in \mathcal{I}, \, (\ell, P(q^\varepsilon)) = 1}} \alpha_{\ell, m} \overline{\chi}(\ell m)\right|^2 \ll_{\varepsilon} \frac{\varphi(q)}{q} \cdot \left(N + \frac{N^2}{q}\right) \frac{1}{H}.
      \]
    \end{lemma}

\begin{proof}
By the mean value theorem (Lemma~\ref{le:MVT}), we have
\begin{align*}
&\frac{1}{\varphi(q)} \sum_{\chi \pmod{q}} \left| \sum_{\substack{\ell m \leq N \\ m \in \mathcal{I}, \, (\ell, P(q^\varepsilon)) = 1}} \alpha_{\ell, m} \overline{\chi}(\ell m)\right|^2 \ll \left(1 + \frac{N}{q} \right) \sum_{\substack{n \leq N \\ (n, q) = 1}} \left(\sum_{\substack{n = \ell m \\ m \in \mathcal{I}, \, (\ell, P(q^\varepsilon)) = 1}} 1 \right)^2.
\end{align*}
Thus it suffices to show that
\begin{equation}\label{eq:shorts-second-moment}
S := \sum_{\substack{n \leq N \\ (n, q) = 1}} \left(\sum_{\substack{n = \ell m \\ m \in \mathcal{I}, \, (\ell, P(q^\varepsilon)) = 1}} 1 \right)^2
\ll \frac{\varphi(q)}{q} \cdot \frac{N}{H}.
\end{equation}

Expanding the square,
\begin{align*}
S = \sum_{\substack{m_1, m_2 \in \mathcal{I} \\ (m_1m_2, q) = 1}} \sum_{\substack{\ell_1 m_1 = \ell_2 m_2 \leq N \\ (\ell_1 \ell_2, qP(q^\varepsilon)) = 1}} 1.
\end{align*}
Write
$$
r \coloneqq (m_1, m_2), \qquad m_1=ra, \qquad m_2=rb,
$$
so that $(a,b)=1$.
Then the relation $\ell_1 m_1=\ell_2 m_2$ becomes
$
\ell_1 a=\ell_2 b.$
Since $(a,b)=1,$ this implies that $b \mid \ell_1$ and $a \mid \ell_2$, so there exists an integer $\ell \geq 1$ such that
$$
\ell_1=b\ell, \qquad \ell_2=a\ell.
$$
Conversely, every such $\ell$ gives a solution. Hence
$$
\ell_1m_1=\ell_2m_2=\ell [m_1,m_2] = \ell a b r.
$$
Moreover, since $(\ell_1\ell_2, P(q^\varepsilon))=1,$ we must have $ (ab, P(q^\varepsilon)) = 1$,
that is,
$$
\left(\frac{m_1m_2}{(m_1,m_2)^2}, P(q^\varepsilon)\right)=1.
$$
Therefore
\begin{align*}
S \ll \sum_{\substack{m_1, m_2 \in \mathcal{I} \\ (m_1m_2, q) = 1 \\ \left(\frac{m_1m_2}{(m_1,m_2)^2}, P(q^\varepsilon)\right)=1}} \sum_{\substack{\ell \leq N/[m_1, m_2] \\ (\ell, qP(q^\varepsilon)) = 1}} 1.
\end{align*}

Applying the fundamental lemma of the sieve to the innermost sum, we obtain 
\begin{align*}
S &\ll_\varepsilon  \frac{N}{\log q}
\sum_{\substack{m_1, m_2 \in \mathcal{I} \\ (m_1m_2, q) = 1 \\ \left(\frac{m_1m_2}{(m_1,m_2)^2}, P(q^\varepsilon)\right)=1}} \frac{1}{[m_1,m_2]}.
\end{align*}
Writing again $r=(m_1,m_2),$ we obtain
\begin{align*}
S &\ll_\varepsilon \frac{N}{\log q}
\sum_{\substack{r \leq Me^{2K+1+1/H} \\ (r, q) = 1}} \frac{1}{r}
\sum_{\substack{n_1, n_2 \\ n_1r, n_2r \in \mathcal{I} \\ (n_1n_2, qP(q^\varepsilon)) = 1}} \frac{1}{m_1m_2}.
\end{align*}
Now each sum over $n_i$ has size
$$\ll_\varepsilon \frac{K}{H\log q} \ll \frac{1}{H},$$
unless $1$ is counted in the sum, in which case the sum is $O(1)$. We may have $n_i=1$ only if $r \in \mathcal{I}$. Hence
\begin{align*}
S &\ll_\varepsilon \frac{N}{\log q}
\left(
\sum_{\substack{r \in \mathcal{I} \\ (r, q) = 1}} \frac{1}{r}
+\sum_{\substack{r \leq Me^{2K+1+1/H} \\ (r, q) = 1}} \frac{1}{r} \cdot \frac{1}{H^2}
  \right).
\end{align*}

By~\eqref{eq:Simplegcd} and the assumptions of the lemma,
$$
\sum_{\substack{r \in \mathcal{I} \\ (r, q) = 1}} \frac{1}{r} = \sum_{|k| \leq 2K+1} \sum_{\substack{r \in (Me^{k-1/H}, Me^{k}] \\ (r, q) = 1}} \frac{1}{r} \ll \frac{\varphi(q)}{q} \frac{K}{H} + \sum_{|k| \leq 2K+1} \frac{\tau(q)}{Me^k} \ll \frac{\varphi(q)}{q} \cdot \frac{\log q}{H}.
$$
Furthermore
$$
\sum_{\substack{r \leq Me^{2K+1+1/H} \\ (r, q) = 1}} \frac{1}{r} \cdot\frac{1}{H^2} \ll \frac{\varphi(q)}{q} \cdot \frac{\log q}{H^2}.
$$
Thus~\eqref{eq:shorts-second-moment}, and hence the lemma, follows.
\end{proof}

    The third lemma on character sums upper bounds a moment involving a large power of a prime character sum.
\begin{lemma}\label{le:amplify} Let $q\in \mathbb{N}$, $X \geq Y_2 \geq Y_1\geq 2$, and $\ell=\lceil (\log Y_2)/(\log Y_1)\rceil$. For any $1$-bounded complex numbers $a_n,c_p$, let
$$Q(\chi)=\sum_{Y_1\leq p\leq 2Y_1}c_p\chi(p),\quad A(\chi)=\sum_{X/Y_2\leq n\leq 2X/Y_2}a_n\chi(n).$$
Then
$$\frac{1}{\varphi(q)} \sum_{\chi\pmod q}|Q(\chi)|^{2\ell}|A(\chi)|^2\ll \frac{\varphi(q)}{q}\left(1+\frac{XY_12^{\ell}}{q}\right)XY_12^{\ell}(\ell+1)!^2.$$
\end{lemma}

\begin{proof}
This is~\cite[Lemma 6.6]{KMT}.
\end{proof}

\subsection{Products in arithmetic progressions}
In order to ensure that many of the solutions we find with $n_1 \equiv n_2 \equiv a \pmod{q}$ and $h(n_1) h(n_2) < 0$ are square-free, we shall use the following lemma.

\begin{lemma} \label{le:KL=amodq}
Let $\varepsilon > 0$ be fixed. Let $a, q \in \mathbb{N}$ be such that $(a, q) = 1$. Let $K, L \geq 2$. Then
\[
\sum_{\substack{k \ell \equiv a \pmod{q} \\ k \leq K, \, \ell \leq L}} 1 \ll \frac{\varphi(q)}{q} \cdot \frac{KL}{q} + q^{1/2+\varepsilon}.
\]
\end{lemma} 
\begin{proof}
Both sides are symmetric in $K$ and $L$ and thus we can assume $L \geq K$.

Let $H \colon \mathbb{R} \to \mathbb{R}_{\geq 0}$ be a fixed smooth function for which $H(x) = 1$ for $|x| \leq 1$ and $H(x) = 0$ for $|x| \geq 2$. By the Poisson summation formula (see e.g.~\cite[equation (4.24)]{iw-kow}) and the superpolynomial decay of $\widehat{H}$, we have 
\begin{align} \label{eq:klbound} \begin{aligned}
\sum_{\substack{k \ell \equiv a \pmod{q} \\ k \leq K, \, \ell \leq L}} 1 &\leq \sum_{(k, q) = 1} H\left(\frac{k}{K}\right) \sum_{\ell \equiv a \overline{k} \pmod{q}} H\left(\frac{\ell}{L}\right) \\
&= \sum_{(k, q) = 1} H\left(\frac{k}{K}\right) \frac{L}{q} \sum_{h \in \mathbb{Z}} \widehat{H}\left(\frac{Lh}{q}\right) e\left(\frac{a\overline{k}}{q}h\right) \\
& \ll  \frac{L}{q} \sum_{|h| \leq \frac{q^{1+\varepsilon}}{L}} \left|\sum_{(k, q) = 1} H\left(\frac{k}{K}\right) e\left(\frac{a\overline{k}}{q}h\right)\right| + O(q^{-100}).
\end{aligned}
\end{align}
The term with $h=0$ contributes $\ll \frac{\varphi(q)}{q} \cdot \frac{KL}{q}$. For the remaining terms, the inner sum is an incomplete Kloosterman sum, and using partial summation and a bound for incomplete Kloosterman sums (see for example~\cite[Equation (2)]{kor}),  we see that the contribution of the terms with $h \neq 0$ is
\[
\ll \frac{L}{q} \sum_{\substack{0 < |h| \leq \frac{q^{1+\varepsilon}}{L}}} \left(\frac{K}{q}+1\right)\left(q^{1/2+\varepsilon}(h, q)^{1/2}\right) = \left(\frac{K}{q}q^{1/2+\varepsilon}+q^{1/2+\varepsilon}\right) \frac{L}{q}\sum_{\substack{0 < |h| \leq \frac{q^{1+\varepsilon}}{L}}} (h, q)^{1/2}.
\]
Here 
\begin{align*}
\sum_{0 < |h|\leq \frac{q^{1+\varepsilon}}{L}}(h,q)^{1/2} \leq \sum_{r \mid q} r^{1/2} \sum_{|h| \leq q^{1+\varepsilon}/(Lr)} 1 \ll \frac{q^{1+2\varepsilon}}{L}.  
\end{align*}
Combining with~\eqref{eq:klbound} and the contribution of the case $h= 0$, we obtain 
\[
\sum_{\substack{k \ell \equiv a \pmod{q} \\ k \leq K, \, \ell \leq L}} 1 \ll \frac{\varphi(q)}{q} \cdot \frac{KL}{q} + \frac{Kq^{1/2+3\varepsilon}}{q} + q^{1/2+3\varepsilon}.
\]
The second term dominates only if $K \geq q$ and $L \leq q^{1/2+2\varepsilon} \frac{q}{\varphi(q)}$ which contradicts the assumption $L \geq K$. Hence the claim follows by adjusting $\varepsilon$.
\end{proof}

\section{Proof of Theorem~\ref{thm:general}: The set-up}\label{sec:general1}
Now we are ready to turn to the proof of Theorem~\ref{thm:general}. The overall strategy is the same as in the proof of Theorem~\ref{thm:generaleasy}, but we need to invoke the Matom\"aki--Radziwi{\l\l} method~\cite{matomaki-radziwill} to be able to prove a counterpart of Lemma~\ref{lemma:le:STcomparisonEasy}, and while doing so, we need to be very careful not to lose density.

We start by fixing some notation for the rest of this paper. Let $\varepsilon > 0$ be sufficiently small. Assume that $q$ is sufficiently large in terms of $\varepsilon$. Furthermore, let $h\colon \mathbb{N}\to \mathbb{R}\setminus\{0\}$ be multiplicative.

For $k \in \mathbb{Z}$ and $y \geq 1$, let
\begin{equation*}
  I_y(k) := (e^{k-1} y, e^k y].
\end{equation*}
Let $Q_1$ be as in the statement of Theorem~\ref{thm:general} and let $P_1 = Q_1/e$. By adjusting $\varepsilon$, we can assume that $Q_1 \leq q^{\varepsilon^2}.$
For $j = 2, \dotsc, J$, let
  \begin{equation} \label{eq:PjQjdef}
     P_j \coloneqq \exp\left(j^{4j}(\log Q_1)^j\right) \quad \text{and} \quad Q_j \coloneqq \exp\left(100 j^{4j+2} (\log Q_1)^j\right),
   \end{equation}
   with $J$ being the largest index such that $Q_J \leq \exp((\log q)^{1/2})$. We let $\mathcal{S}$ be the set of all integers that have at least one prime factor from each interval $(P_j, Q_j]$ with $j \in \{2, 3, \dotsc, J\}$.

Let also
\[
  K \coloneqq \lfloor \varepsilon^2 \log q \rfloor, \,  R:=q^{1/2-\varepsilon/4}, \, M := \frac{q}{R^2} = q^{\varepsilon/2}, \, U:=\frac{q}{R} = q^{1/2+\varepsilon/4}, \, z:= q^{\sqrt{\varepsilon}},
\]
and, for every integer $v \in [-3K, 3K]$, every $\Delta \in \{+, -\}$, and every $B \subseteq \mathbb{Z}_q^\times$, define the sets
  \begin{align*} 
    \mathcal{Q}_B^\Delta&:= \{p \in (P_1,Q_1], \, p \in B, \, \sgn(h(p)) = \Delta\}, \\
    \mathcal{U}_{B, v}^\Delta&:=\{u \in I_U(v) \colon |\mu(u)| = 1, \, u \in B, \, \sgn(h(u)) = \Delta\}, \\
      \mathcal{M}_{B, v}^\Delta&:=\{m \in I_M(v) \colon |\mu(m)| = 1, \, m \in \mathcal{S}, \, m \in B, \, \sgn(h(m)) = \Delta, (m, P(Q_1)) = 1\}.
\end{align*}
For $|k|\leq K$, define
\begin{equation}
\label{eq:fkDef}
f^{\Delta}_k(n) =  \prod_{\substack{p < z \\ p \nmid q}} \left(1-\frac{1}{p}\right)^{-1} \mathbf{1}_{\sgn(h(n)) = \Delta}\mathbf{1}_{(n, P(z)) = 1}\mathbf{1}_{n \in [I_R(k)]_q}.
\end{equation}

For $\overline{\Delta} = (\Delta_1, \dotsc, \Delta_6) \in \{+, -\}^6, \, \mathcal{K} \subseteq (\mathbb{Z} \cap [-K, K])^3$ and $B_4, B_5, B_6 \subseteq \mathbb{Z}_q^\times$, define the function $S_{B_4, B_5, B_6}^{\overline{\Delta}, \mathcal{K}} \colon \mathbb{Z}_q^\times \to \mathbb{R}_{\geq 0}$ by
\begin{align} \label{eq:SgenDef} \begin{aligned} 
&  S^{\overline{\Delta}, \mathcal{K}}_{B_4, B_5, B_6}(a) \\ &:= \sum_{\substack{\overline{k} \in \mathcal{K} \\ \overline{k} = (k_1, k_2, k_3)}} \frac{1}{S_{\overline{k}}} \sum_{\substack{n \equiv a \pmod{q}}} \left(f_{k_1}^{\Delta_1} \ast f_{k_2}^{\Delta_2} \ast f_{k_3}^{\Delta_3}  \ast \1[\mathcal{Q}^{\Delta_4}_{B_4}] \ast \1[\mathcal{U}^{\Delta_5}_{B_5, -k_1}] \ast \1[\mathcal{M}^{\Delta_6}_{B_6, -k_2-k_3}] \right)(n),
\end{aligned}
\end{align}
where
\[
 S_{\overline{k}} := |I_R(k_1)| \cdot |I_R(k_2)| \cdot |I_R(k_3)| \cdot Q_1 \cdot Ue^{-k_1} \cdot Me^{-k_2-k_3} \asymp q^2 Q_1.
\]
\begin{remark} \label{rem:ShapeSgen}
Notice that if a natural number $n$ is counted by $S_{B_4, B_5, B_6}^{\overline{\Delta}, \mathcal{K}}(a)$, then for some $\overline{k} = (k_1, k_2, k_3) \in \mathcal{K}$, we have $n\equiv a\pmod q, \,\sgn(h(n)) = \Delta_1\dotsm \Delta_6 = \Delta$ and 
\begin{equation} \label{eq:ngendec}
  n = r_1r_2r_3 \cdot p_1 \cdot u \cdot m \leq q^2 Q_1,
\end{equation}
where, for $i \in \{1, 2,3\}$, 
\begin{itemize}
\item $p_1 \in (P_1, Q_1]$ is a prime;
    \item $r_i\in [I_R(k_i)]_q$ and $(r_i,P(z))=1$;
      \item $u$ is a square-free integer with $u \in I_U(-k_1)$;
      \item $m$ is a square-free integer such that $m \in \mathcal{S}, \, m \in I_M(-k_2-k_3),$ and $(m, P(Q_1)) = 1.$
      \end{itemize}
The set-up is similar as in Remark~\ref{rem:ShapeS}, but now we have the new factor $m \in \mathcal{S}$ that we shall utilize in our Matom\"aki--Radziwi{\l\l} type argument and the sum over $\overline{k}$ that ensures that the density of the set we are working with is sufficient for successful applications of the mean value theorem. For applying the dense model theorem, it is convenient to have primes from $e$-adic intervals $I_R(k_i)$, but if we did not sum over $k_i$, we would only work with numbers that have prime factors from three fixed $e$-adic intervals and would lose $\asymp 1/\log^3 R$ in density. 
\end{remark}

\begin{lemma} \label{le:reductiontoBoundingS}
Let $Q_1 \geq 3$ be sufficiently large. Let $\Delta \in \{+, -\}$ and $a \in \mathbb{Z}_q^\times$. Assume that there exist $\overline{\Delta} = (\Delta_1, \dotsc, \Delta_6) \in \{+, -\}^6, \, \mathcal{K} \subseteq (\mathbb{Z} \cap [-K, K])^3$ and $B_4, B_5, B_6 \subseteq \mathbb{Z}_q^\times$ such that $\Delta_1 \dotsm \Delta_6 = \Delta$ and $S_{B_4, B_5, B_6}^{\overline{\Delta}, \mathcal{K}}(a) \gg \frac{\varphi(q)}{q} \cdot \frac{\log^3 q}{q Q_1^{1/100} \log^2 Q_1}$. Then $a \in E_h^{\Delta}(q^{2}Q_1)$.
\end{lemma}

    \begin{proof} 
      Let $B_4, B_5, B_6,$ and $\overline{\Delta}$ be as in the statement. By Remark~\ref{rem:ShapeSgen} it suffices to show that the contribution of non-square-free integers to $S_{B_4, B_5, B_6}^{\overline{\Delta}, \mathcal{K}}(a)$ is $o(\frac{\varphi(q)}{q} \cdot \frac{\log^3 q}{q Q_1^{1/100} \log^2 Q_1}).$ If an integer counted by $S_{B_4, B_5, B_6}^{\overline{\Delta}, \mathcal{K}}(a)$ is not square-free, it must be divisible by a prime square $p^2$ with $p \in [P_1, Re^K]$. We write $n$ as in~\eqref{eq:ngendec} as $n = k_1 \ell_1$ with $k_1= r_1 u \leq q$ and $\ell_1 = r_2r_3 m p_1 \leq qQ_1$.

      Consider the contribution of $n = k_1 \ell_1$ such that $p^2 \mid k_1 \ell_1$ for some $p \in [P_1, Re^K]$. The contribution of $n$ with $p^2 \mid \ell_1$ is
      \[
\ll \frac{\log^3 q}{q^2 Q_1} \sum_{P_1 < p \leq Re^K} \sum_{\substack{\ell \leq qQ_1/p^2 \\ (\ell, q) = 1}} \sum_{k_1 \leq q}  \mathbf{1}_{k_1 \ell p^2 \equiv a \pmod{q}} \ll \frac{\varphi(q)}{q} \cdot \frac{\log^3 q}{q P_1},
        \]
the contribution of $n$ with $p^2 \mid k_1$ is
\[
 \ll \frac{\log^3 q}{q^2 Q_1} \sum_{P_1 < p \leq Re^K} \sum_{\substack{k \leq q/p^2 \\ (k, q) = 1}} \sum_{\ell_1 \leq qQ_1} \mathbf{1}_{k \ell_1 p^2 \equiv a \pmod{q}} \ll \frac{\varphi(q)}{q} \frac{\log^3 q}{qP_1}
\]
and the contribution of $n$ with $p \mid (k_1, \ell_1)$ is
        \begin{align*}
&\ll \frac{\log^3 q}{q^2 Q_1} \sum_{P_1 < p \leq Re^K} \sum_{k \leq q/p} \sum_{\substack{\ell \leq qQ_1 /p}} \mathbf{1}_{k \ell p^2 \equiv a \pmod{q}}.
\end{align*}
Applying Lemma~\ref{le:KL=amodq} (with $\varepsilon^2$ in place of $\varepsilon$), we see that this is also $\ll \frac{\varphi(q)}{q} \frac{\log^3 q}{q P_1}$, and thus the total contribution of non-square-free integers to $S_{B_4, B_5, B_6}^{\overline{\Delta}, \mathcal{K}}(a)$ is
\[
  \ll \frac{\varphi(q)}{q} \cdot \frac{\log^3 q}{q P_1} = o\left(\frac{\varphi(q)}{q} \cdot \frac{\log^3 q}{q Q_1^{1/100} \log^2 Q_1}\right).
  \]
\end{proof}

\section{Proof of Theorem~\ref{thm:general}: Applying the dense model theorem}
Let $D = q^{1/100},$ recall that $z = q^{\sqrt{\varepsilon}}$, and let $\lambda_d^+$ be the upper bound sieve coefficients from Lemma~\ref{le:fundsieve} with these parameters and $\kappa = 1$ (and $s = 1/(100\sqrt{\varepsilon})$). Now, for $\Delta \in \{+, -\}$ and $|k| \leq K$, the function $f_k^\Delta$ has a majorant $\nu_k \colon \mathbb{Z} \to \mathbb{R}_{\geq 0}$ given by
\begin{equation}
\label{eq:nukDef}
\nu_k(n) := \prod_{\substack{p < z \\ p \nmid q}} \left(1-\frac{1}{p}\right)^{-1} \sum_{\substack{d \mid n \\ d \leq D}} \lambda_d^+\cdot \mathbf{1}_{n \in [I_R(k)]_q}.
\end{equation}
We can establish Proposition~\ref{prop:f=g+h}(A1, A2) as in Section~\ref{ssec:easyset-up} (the only difference is the slightly different supports of the functions) and thus we may apply Proposition~\ref{prop:f=g+h} with the same parameters as there and make the following definition.

\begin{definition} \label{def:gkDef}
For each $|k| \leq K$ and $\Delta \in \{+, -\}$, let $g_k^{\Delta} \colon \mathbb{Z}_q^\times \to [0, 1+\varepsilon^2]$ be the function obtained from Proposition~\ref{prop:f=g+h} with $r = 2$ and $\delta, f^\Delta_k,$ and $\nu_k$ as in~\eqref{eq:deltadef},~\eqref{eq:fkDef} and~\eqref{eq:nukDef}. 
\end{definition}

In order to lower bound $S_{B_4, B_5, B_6}^{\overline{\Delta}, \mathcal{K}}(a)$, we compare the function $S^{\overline{\Delta}, \mathcal{K}}_{B_4, B_5, B_6}$ with the function  $T_{B_4, B_5, B_6}^{\overline{\Delta}, \mathcal{K}} \colon \mathbb{Z}_q^\times \to \mathbb{R}_{\geq 0}$ defined by
\begin{align} \label{eq:Tgendef}
  T^{\overline{\Delta}, \mathcal{K}}_{B_4, B_5, B_6}(a) &:= \sum_{\substack{\overline{k} \in \mathcal{K}}} \frac{1}{T_{\overline{k}}} \left(g_{k_1}^{\Delta_1} \ast g_{k_2}^{\Delta_2} \ast g_{k_3}^{\Delta_3} \ast \1[\mathcal{Q}^{\Delta_4}_{B_4}] \ast \1[\mathcal{U}^{\Delta_5}_{B_5, -k_1}] \ast \1[\mathcal{M}^{\Delta_6}_{B_6, -k_2-k_3}]\right)(a),
\end{align}
where
\begin{equation}\label{eq:Tkbardef}
  T_{\overline{k}} := |\mathbb{Z}_q^\times|^3 \cdot Q_1 \cdot Ue^{-k_1} \cdot Me^{-k_2-k_3}.
\end{equation}
We will prove the following lemma.
\begin{lemma}\label{lemma:le:STcomparison}
Let $\mathcal{K} = \mathcal{K}_1 \times \mathcal{K}_2$ with $\mathcal{K}_1 \subseteq \mathbb{Z} \cap [-K, K]$ and $\mathcal{K}_2 \subseteq (\mathbb{Z} \cap [-K, K])^2$. Let $H \leq \mathbb{Z}_q^\times$ be of bounded index and let, for $i \in \{4, 5, 6\}$, $b_i \in \mathbb{Z}_q^\times$ and $B_i = b_i H$. Let $\overline{\Delta} = (\Delta_1, \dotsc, \Delta_6) \in \{+, -\}^6$, and let $S_{B_4, B_5, B_6}^{\overline{\Delta}, \mathcal{K}}$ and $T_{B_4, B_5, B_6}^{\overline{\Delta}, \mathcal{K}}$ be as in~\eqref{eq:SgenDef} and~\eqref{eq:Tgendef}. Then, for all $a \in \mathbb{Z}_q^\times$,
 \begin{align*}
&|S_{B_4, B_5, B_6}^{\overline{\Delta}, \mathcal{K}}(a) - T_{B_4, B_5, B_6}^{\overline{\Delta}, \mathcal{K}}(a)|\\
\ll& \frac{Q_1^{-1/90}}{q} \frac{\varphi(q)}{q} \log^3q + \frac{\log^{7/4} q}{q \log^2 Q_1} \frac{(q, P(Q_1))}{\varphi((q, P(Q_1)))} \sum_{k_1 \in \mathcal{K}_1} \frac{|\mathcal{U}_{B_5, -k_1}^{\Delta_5}|}{e^{-k_1} U}.
  \end{align*}
\end{lemma}
We do some preparations before turning to the proof of Lemma~\ref{lemma:le:STcomparison}. For $\Delta \in \{+, -\}$ and $B \subseteq \mathbb{Z}_q^\times$, write
  \begin{align*}
    F_k^{\Delta}(\chi)&:=\mathbb{E}_{n\in [I_R(k)]_q} f_k^{\Delta}(n)\overline{\chi(n)}, \qquad G_k^\Delta(\chi):=\mathbb{E}_{b \in \mathbb{Z}_q^\times} g_k^{\Delta}(b)\overline{\chi}(b), \\
    Q_B^\Delta(\chi)&:= \frac{1}{Q_1} \sum_{p\in \mathcal{Q}_B^\Delta} \overline{\chi}(p), \qquad U_{B,v}^\Delta(\chi) := \frac{1}{Ue^v} \sum_{\substack{u \in \mathcal{U}_{B, v}^\Delta}}  \overline{\chi}(u), \\
    M_{B,v}^\Delta(\chi) &:= \frac{1}{Me^v} \sum_{\substack{m \in \mathcal{M}_{B, v}^\Delta}} \overline{\chi}(m).
  \end{align*}
By orthogonality of characters,
\begin{align*} 
\begin{aligned}
  &  S_{B_4, B_5, B_6}^{\overline{\Delta}, \mathcal{K}}(a) - T_{B_4, B_5, B_6}^{\overline{\Delta}, \mathcal{K}}(a)\\
  =&\frac{1}{\varphi(q)}  \sum_{\chi \pmod{q}}\chi(a) \sum_{\substack{\overline{k} \in \mathcal{K}}} \Biggl( \prod_{i=1}^3 F_{k_i}^{\Delta_i}(\chi) -  \prod_{i=1}^3 G_{k_i}^{\Delta_i}(\chi)\Biggr) Q_{B_4}^{\Delta_4}(\chi) U_{B_5, -k_1}^{\Delta_5}(\chi) M_{B_6, -k_2-k_3}^{\Delta_6}(\chi).
\end{aligned}
\end{align*}
We will need to extract several prime factors from the character sum $M_{B,v}^\Delta(\chi)$.

In order to do this, we define, for $j \in \{2, \dotsc, J\}$, 
\begin{align*} 
  \begin{aligned}
H_j &:= j^4 Q_1^{1/40}, \quad \mathcal{I}_j := \{w \in \mathbb{N} \colon H_j \log P_j \leq w \leq \lceil H_j \log Q_j \rceil\}, \\
  Q_{j, B, w}^\Delta(\chi) &:= \frac{1}{e^{w/H_j}} \sum_{\substack{P_j < p \leq Q_j \\ e^{(w-1)/H_j} < p \leq e^{w/H_j}}} \1_{\sgn(h(p)) = \Delta} \1_{p \in B} \overline{\chi}(p), \\
    R_{j, B, v, w}^\Delta(\chi) &:= \frac{1}{Me^{v-w/H_j}} \sum_{\substack{Me^{v-w/H_j-1} < m \leq Me^{v-w/H_j} \\ m \in \mathcal{S}_j \cap B}} |\mu(m)| \1_{\sgn(h(m)) = \Delta} \mathbf{1}_{(m, P(P_1)) = 1} \overline{\chi}(m),
  \end{aligned}
\end{align*}
where $\mathcal{S}_j$ consists of integers that have at least one prime factor from each $(P_r, Q_r]$ with $r \in \{2, \dotsc, J\} \setminus\{j\}$.
Note that $Q_{j, B, w}^\Delta(\chi)$ can be non-zero only when $w \in \mathcal{I}_j$ and
\begin{equation} \label{eq:|Ij|bound}
|\mathcal{I}_j| \ll H_j \log Q_j.
\end{equation}

The following lemma allows us to replace $M_{bH, v}^\Delta(\chi)$ in character sums arising from~\eqref{eq:S-Tchar} by
\[
    \widetilde{M}_{j, bH, v}^{\Delta}(\chi) \coloneqq  \sum_{\Delta_1 \Delta_2 = \Delta} \sum_{\substack{b_1, b_2 \in \mathbb{Z}_q^\times /H \\ b_1 b_2 H = bH}}\sum_{w \in \mathcal{I}_j} Q_{j, b_1 H, w}^{\Delta_1}(\chi) R_{j, b_2 H, v, w}^{\Delta_2}(\chi). 
  \]
  \begin{lemma} \label{le:ChiSumMsplit}
Let $\mathcal{K} = \mathcal{K}_1 \times \mathcal{K}_2$ with $\mathcal{K}_1 \subseteq \mathbb{Z} \times [-K, K]$ and $\mathcal{K}_2 \subseteq (\mathbb{Z} \cap [-K, K])^2$. Let $H \leq \mathbb{Z}_q^\times$ have bounded index and let, for $i \in \{4, 5, 6\}$, $b_i \in \mathbb{Z}_q^\times$ and $B_i = b_i H$. Let $\overline{\Delta} = (\Delta_1, \dotsc, \Delta_6) \in \{+, -\}^6$. Then
\begin{align*}
&\frac{1}{\varphi(q)}  \sum_{\chi \pmod{q}}\Biggl| \sum_{\substack{\overline{k} \in \mathcal{K}}} \Biggl( \prod_{i=1}^3 F_{k_i}^{\Delta_i}(\chi) -  \prod_{i=1}^3 G_{k_i}^{\Delta_i}(\chi)\Biggr) Q_{B_4}^{\Delta_4}(\chi) U_{B_5, -k_1}^{\Delta_5}(\chi) \\
& \qquad \qquad \cdot \left(M_{B_6, -k_2-k_3}^{\Delta_6}(\chi)-\widetilde{M}_{j, B_6, -k_2-k_3}^{\Delta_6}(\chi)\right)\Biggr| \ll \frac{\varphi(q)}{q^2} \left( \frac{1}{P_j^{1/2}} + \frac{1}{H_j^{1/2}}\right) \log^3 q.
\end{align*}
\end{lemma}
We will prove this using the following lemma.
\begin{lemma}\label{lemma:le:decomp}
Let $\Delta \in \{+, -\},$ $|v| \leq 3K,$ and $j \in\{2, \dotsc, J\}$. Let $b \in \mathbb{Z}_q^\times$ and let $H \leq \mathbb{Z}_q^\times$ be of bounded index. There exist bounded coefficients $d_{j, bH, v}^\Delta(m)$ and $d_{j, bH, v}^\Delta(p, m)$ such that
  \begin{align} \label{eq:Mdec}
    \begin{aligned}
    M_{bH, v}^{\Delta}(\chi) &= \widetilde{M}_{j, bH, v}^{\Delta}(\chi) + E_{j, bH, v}^{1, \Delta}(\chi)  + E_{j, bH, v}^{2, \Delta}(\chi),
    \end{aligned}
    \end{align}
    where
    \begin{align*} 
     \begin{aligned}  
  E_{j, bH, v}^{1, \Delta}(\chi) &\coloneqq \frac{1}{Me^v} \sum_{\substack{p^2 m \in I_M(v) \\ P_j < p \leq Q_j}} d_{j, bH, v}^{\Delta}(p, m) \overline{\chi}(p^2m), \\
 E_{j, bH, v}^{2, \Delta}(\chi) &\coloneqq \frac{1}{Me^v} \sum_{Me^{v-1-1/H_j} < m \leq Me^{v-1}} d_{j, bH, v}^\Delta(m) \overline{\chi}(m)  \\
 &\qquad + \frac{1}{Me^v} \sum_{Me^{v-1/H_j} < m \leq Me^{v}} d_{j, bH, v}^\Delta(m) \overline{\chi}(m).
    \end{aligned}
  \end{align*}
  \end{lemma}

  \subsection{Proofs of the decomposition lemmas (Lemmas \ref{lemma:le:decomp} and \ref{le:ChiSumMsplit})}
  \begin{proof}[Proof of Lemma~\ref{lemma:le:decomp}]
We start from a Ramar\'e-type decomposition according to one distinguished prime factor from the interval $(P_j,Q_j]$. Writing
\[
\omega(n;P_j,Q_j)\coloneqq |\{p\in (P_j,Q_j]\colon p\mid n\}|,
\]
we have
\begin{align*}
M_{bH,v}^{\Delta}(\chi)
&=
\frac{1}{Me^v}
\sum_{\substack{n\in I_M(v)\\ |\mu(n)|=1,\ n\in \mathcal S\cap bH\\ \sgn(h(n))=\Delta\\ (n,P(P_1))=1}}
\overline{\chi}(n) =
\frac{1}{Me^v}
\sum_{\substack{n\in I_M(v)\\ |\mu(n)|=1,\ n\in \mathcal S\cap bH\\ \sgn(h(n))=\Delta\\ (n,P(P_1))=1}}
\frac{\overline{\chi}(n)}{\omega(n;P_j,Q_j)}
\sum_{\substack{P_j<p\leq Q_j\\ p\mid n}} 1.
\end{align*}
Writing $n=pm$, this becomes
\begin{align*}
M_{bH,v}^{\Delta}(\chi)
&=
\sum_{\Delta_1\Delta_2=\Delta}
\sum_{\substack{b_1,b_2\in \mathbb Z_q^\times/H\\ b_1b_2H=bH}}
\frac{1}{Me^v}
\sum_{\substack{P_j<p\leq Q_j\\ p\in b_1H\\ \sgn(h(p))=\Delta_1}}
\overline{\chi}(p)
\sum_{\substack{m\in I_{M/p}(v)\\ |\mu(mp)|=1,\ m\in \mathcal S_j\cap b_2H\\ \sgn(h(m))=\Delta_2\\ (m,P(P_1))=1}}
\frac{\overline{\chi}(m)}{\omega(m;P_j,Q_j)+1}.
\end{align*}

We now decompose this into a main term and two error terms. First, replacing $|\mu(mp)|$ by $|\mu(m)|$ creates an error supported on integers of the form $p^2m$, with $P_j<p\leq Q_j$ and $p^2m\in I_M(v)$. Thus this contribution is of the shape $E_{j,bH,v}^{1,\Delta}(\chi)$.

After making this replacement, split the prime variable into the intervals
\[
e^{(w-1)/H_j}<p\le e^{w/H_j}, \qquad w\in \mathcal I_j.
\]
For such $p$, replace the condition $m\in I_{M/p}(v)$ by the $p$-independent condition
\[
m\in \bigl(Me^{v-w/H_j-1},\,Me^{v-w/H_j}\bigr].
\]
This produces exactly the main term
\[
\widetilde M_{j,bH,v}^{\Delta}(\chi)
=
\sum_{\Delta_1\Delta_2=\Delta}
\sum_{\substack{b_1,b_2\in \mathbb Z_q^\times/H\\ b_1b_2H=bH}}
\sum_{w\in \mathcal I_j}
Q_{j,b_1H,w}^{\Delta_1}(\chi)\,
R_{j,b_2H,v,w}^{\Delta_2}(\chi).
\]

It remains to identify the error coming from replacing $I_{M/p}(v)$ by the fixed interval
$I_{Me^{-w/H_j}}(v)$. If $p \in (e^{(w-1)/H_j}, e^{w/H_j}]$, then the symmetric difference between
\[
I_{M/p}(v)
\quad\text{and}\quad
\bigl(Me^{v-w/H_j-1},\,Me^{v-w/H_j}\bigr]
\]
is
\[
\bigl(Me^{v-w/H_j-1}, Me^{v-1}/p\bigr]
\cup
\bigl(Me^{v-w/H_j},\,Me^{v}/p\bigr].
\]
After multiplying by $p$, these correspond to values of $n=pm$ lying in the two edge pieces
encoded by $E_{j,bH,v}^{2,\Delta}(\chi)$. Hence this second contribution is of the shape
$E_{j,bH,v}^{2,\Delta}(\chi)$.

Collecting the main term and the two error terms, we obtain~\eqref{eq:Mdec}.
\end{proof}

\begin{proof}[Proof of Lemma~\ref{le:ChiSumMsplit}]
By Lemma~\ref{lemma:le:decomp} and the triangle inequality, it is enough to show that, for
$\ell \in \{1,2\}$,
\begin{align}
\label{eq:decompFcaseClaim-new} \begin{aligned}
&\frac{1}{\varphi(q)}  \sum_{\chi \pmod{q}}\Biggl| \sum_{\substack{\overline{k} \in \mathcal{K}}} 
\Biggl( \prod_{i=1}^3 F_{k_i}^{\Delta_i}(\chi)\Biggr) U_{B_5, -k_1}^{\Delta_5}(\chi) 
E_{j, B_6, -k_2-k_3}^{\ell, \Delta_6}(\chi)\Biggr| \\
&\hspace{5cm}
\ll  \frac{\varphi(q)}{q^2} \left( \frac{1}{P_j^{1/2}} + \frac{1}{H_j^{1/2}}\right) \log^3 q
\end{aligned}
\end{align}
and the same bound with $F_{k_i}^{\Delta_i}$ replaced by $G_{k_i}^{\Delta_i}$. We only prove~\eqref{eq:decompFcaseClaim-new}, since the proof of the variant with $G_{k_i}^{\Delta_i}$ is similar, using only the pointwise bound $0\leq g_k^\Delta \leq 1+\varepsilon^2$ in place of the
corresponding bound for $f_k^\Delta$.

Recall $\mathcal{K} = \mathcal{K}_1 \times \mathcal{K}_2$. Set
\begin{equation} \label{eq:A(chi)def}
A(\chi):=\sum_{k_1 \in \mathcal{K}_1} F_{k_1}^{\Delta_1}(\chi)U_{B_5,-k_1}^{\Delta_5}(\chi)
\end{equation}
and, for $\ell \in \{1, 2\}$,
\[
B_\ell(\chi):=\sum_{(k_2, k_3) \in \mathcal{K}_2} F_{k_2}^{\Delta_2}(\chi)F_{k_3}^{\Delta_3}(\chi)
E_{j,B_6,-k_2-k_3}^{\ell,\Delta_6}(\chi).
\]
Then by the Cauchy--Schwarz inequality, the left-hand side of~\eqref{eq:decompFcaseClaim-new} is bounded by
\[
\frac{1}{\varphi(q)}\sum_{\chi\pmod q}|A(\chi)B_\ell(\chi)|
\leq
\left(\frac{1}{\varphi(q)}\sum_{\chi\pmod q}|A(\chi)|^2\right)^{1/2}
\left(\frac{1}{\varphi(q)}\sum_{\chi\pmod q}|B_\ell(\chi)|^2\right)^{1/2}.
\]

We first bound the $A(\chi)$ factor. Unwrapping the definitions of $F_{k_1}^{\Delta_1}$ and
$U_{B_5,-k_1}^{\Delta_5}$, we may write
\[
A(\chi)=\frac{1}{q}\sum_{n\leq q} a_n \overline{\chi}(n),
\]
where, by Mertens' theorem, the coefficients satisfy 
\begin{align*}
|a_n|\ll \frac{q}{\varphi(q)} \sum_{k_1 \in \mathcal{K}_1}\sum_{\substack{n=ru\\ r\in I_R(k_1)\\ (r,P(z))=1\\ u\in I_U(-k_1)}}\prod_{\substack{p < z \\ p \nmid q}} \left(1-\frac{1}{p}\right)^{-1}\ll (\log q)\sum_{|k_1|\leq K}\sum_{\substack{n=ru\\ r\in I_R(k_1)\\ (r,P(z))=1\\ u\in I_U(-k_1)}} 1.    
\end{align*}
For each fixed $n\leq q$, the number of choices for $r$ is $\ll 1$, since $r$ has all its prime factors $\geq z=q^{\sqrt{\varepsilon}}$. Moreover, each divisor $r$ can only belong to one of the intervals $I_R(k_1)$. Thus $|a_n| \ll \log q$, so the mean value theorem (Lemma~\ref{le:MVT}) gives
\begin{equation}
\label{eq:Achi2bound}
\frac{1}{\varphi(q)}\sum_{\chi\pmod q}|A(\chi)|^2 \ll \frac{1}{q^2} \sum_{\substack{n \leq q \\ (n, q)=1}} |a_n|^2 \ll \frac{\varphi(q)}{q^2}\log^2 q.
\end{equation}
Hence~\eqref{eq:decompFcaseClaim-new} follows once we have shown that, for $\ell \in \{1, 2\}$,
\begin{equation} \label{eq:Bellneed}
\frac{1}{\varphi(q)}\sum_{\chi\pmod q}|B_\ell(\chi)|^2 \ll  \frac{\varphi(q)}{q^2} \left( \frac{1}{P_j} + \frac{1}{H_j}\right) \log^4 q.
\end{equation}
\medskip
\noindent\textbf{The case $\ell=1$.}
By the definition of $E_{j,B_6,v}^{1,\Delta_6}$, we may write 
\[
B_1(\chi)=\frac{1}{q} \left(\frac{q}{\varphi(q)}\right)^2 \prod_{\substack{p < z \\ p \nmid q}} \left(1-\frac{1}{p}\right)^{-2} \sum_{\substack{p^2 h\leq q \\ P_j < p \leq Q_j}} \alpha_{p, h} \overline{\chi}(p^2 h),
\]
where
\[
\alpha_{p, h} = \sum_{\substack{(k_2, k_3) \in \mathcal{K}_2}}
\sum_{\substack{h=r_2r_3m\\
r_2\in I_R(k_2),\, r_3\in I_R(k_3)\\
(r_2r_3,P(z))=1\\
p^2m\in I_M(-k_2-k_3)}} d_{j, B_6, -k_2-k_3}^{\Delta_6}(p, m).
\]
For fixed $h\leq q$, the number of choices of $r_2,r_3$ with $(r_2r_3,P(z))=1$ is
$\ll 1$. Once $r_2,r_3$ are fixed, the variables $m, k_2, k_3$ are determined. Hence $\alpha_{p, h} \ll 1$.
Applying Mertens' theorem and Lemma~\ref{le:squarecontr} with $N= q$ and $P=P_j$, we obtain
\begin{equation*}
\frac{1}{\varphi(q)}\sum_{\chi\pmod q}|B_1(\chi)|^2 \ll \frac{\log^4 q}{q^2} \cdot \frac{1}{\varphi(q)} \sum_{\chi\pmod q}\left|\sum_{\substack{p^2 h\leq q \\ P_j < p \leq Q_j}} \alpha_{p, h} \overline{\chi}(p^2 h)\right|^2 \ll \frac{\varphi(q)}{q^2} \cdot \frac{\log^4 q}{P_j},
\end{equation*}
and thus~\eqref{eq:Bellneed} holds for $\ell = 1$.
\medskip

\noindent\textbf{The case $\ell=2$.}
For $k\in [-2K,2K]$, write
\[
\mathcal{J}_k\coloneqq
(Me^{-k-1-1/H_j},\,Me^{-k-1}]
\cup
(Me^{-k-1/H_j},\,Me^{-k}],
\]
so that the sums in the definitions of $E_{j,B_6,-k_2-k_3}^{2,\Delta_6}$ are supported on $\mathcal{J}_{k_2+k_3}$. Let also
\[
\mathcal{J}:=\bigcup_{|k|\le 2K}\mathcal{J}_k \subseteq \bigcup_{|k| \leq 2K+1} (M e^{k-1/H_j}, Me^k].
\]
Unwrapping the definitions, we may write
\[
B_2(\chi)= \frac{1}{q} \left(\frac{q}{\varphi(q)}\right)^2 \prod_{\substack{p < z \\ p \nmid q}} \left(1-\frac{1}{p}\right)^{-2} \sum_{\substack{m n \leq q \\ m \in \mathcal{J}}} \alpha_{m, n} \overline{\chi}(m n),
\]
where 
\[
\alpha_{m, n}
=
\sum_{\substack{(k_2, k_3)\in \mathcal{K}_2}} \mathbf{1}_{m\in \mathcal{J}_{k_2+k_3}} d_{j, B_6, -k_2-k_3}(m)
\sum_{\substack{n=r_2r_3\\
r_2\in I_R(k_2),\,r_3\in I_R(k_3)\\
(r_2r_3,P(z))=1}} 1.
\]
Since $r_2,r_3$ are $z$-rough and $n \le q$, the number of possibilities for the pair $(r_2,r_3)$ with product $n$ is $\ll 1$. Moreover, once $r_2,r_3$ are fixed, the conditions $r_2\in I_R(k_2)$ and $r_3\in I_R(k_3)$ determine $k_2$ and $k_3$ uniquely. Thus $\alpha_{m, n} \ll 1$.

Hence Mertens' theorem and Lemma~\ref{le:shortscontr} imply that
\begin{equation*}
\frac{1}{\varphi(q)}\sum_{\chi\pmod q}|B_2(\chi)|^2 \ll \frac{\log^4 q}{q^2} \cdot \frac{1}{\varphi(q)} \sum_{\chi\pmod q}\left| \sum_{\substack{m n \leq q \\ m \in \mathcal{J}}} \alpha_{m, n} \overline{\chi}(m n) \right|^2 \ll \frac{\varphi(q)}{q^2}\cdot \frac{\log^4 q}{H_j},
\end{equation*}
and thus~\eqref{eq:Bellneed} holds for $\ell = 2$.
\end{proof}

\subsection{Proof of Lemma~\ref{lemma:le:STcomparison}}
\begin{proof}[Proof of Lemma~\ref{lemma:le:STcomparison}]
  Recall we want to estimate
\begin{align} \label{eq:S-Tchar}
\begin{aligned}
  & \phantom{=}  S_{B_4, B_5, B_6}^{\overline{\Delta}, \mathcal{K}}(a) - T_{B_4, B_5, B_6}^{\overline{\Delta}, \mathcal{K}}(a)\\
  &=\frac{1}{\varphi(q)}  \sum_{\chi \pmod{q}}\chi(a) \sum_{\substack{\overline{k} \in \mathcal{K}}} \Biggl( \prod_{i=1}^3 F_{k_i}^{\Delta_i}(\chi) -  \prod_{i=1}^3 G_{k_i}^{\Delta_i}(\chi)\Biggr) Q_{B_4}^{\Delta_4}(\chi) U_{B_5, -k_1}^{\Delta_5}(\chi) M_{B_6, -k_2-k_3}^{\Delta_6}(\chi).
\end{aligned}
\end{align}
Following the method introduced by the first author and Radziwi{\l\l} in~\cite{matomaki-radziwill}, we will split the characters $\pmod{q}$ into several sets. For this, we need a bit more notation.

  For $j = 1, \dotsc, J$, let
\[
  \alpha_j=\frac{1}{40} - \eta\left(1+\frac{1}{2j}\right),
\]
where $\eta > 0$ is a small constant. Let $\mathcal{I}_1 \coloneqq \{1\}$ and $Q_{1, B, v}^{\Delta}(\chi) \coloneqq Q_{B}^\Delta(\chi)$.

Let $H$ be as in the statement of Lemma~\ref{lemma:le:STcomparison}. We write
\begin{equation*} 
\{\chi \pmod{q}\} = \bigcup_{j=1}^J \mathcal{X}_j \cup \mathcal{Y}
\end{equation*}
as a disjoint union, where $\chi \in \mathcal{X}_j$, when $j$ is the smallest index such that
\[
\text{for all $v \in \mathcal{I}_j$, all $b \in \mathbb{Z}_q^\times$ and all $\Delta \in \{+, -\}$: } |Q_{j, bH, v}^{\Delta}(\chi)|\leq e^{-\alpha_j v /H_j}.
\]
Finally $\chi \in \mathcal{Y}$ if this does not hold for any $j \in \{1, \ldots, J\}$. When considering the right-hand side of~\eqref{eq:S-Tchar}, we consider separately the contribution of $\chi \in \mathcal{X}_1$, the contribution of $\mathcal{X}_j$ with $j \in \{2, \dotsc, J\}$, and the contribution of $\chi \in \mathcal{Y}$.

\textbf{Contribution of $\mathcal{X}_1$.}
We write the right-hand side of~\eqref{eq:S-Tchar} as a difference of two terms involving the product over $F_{k_i}^{\Delta_i}(\chi)$ and the product over $G_{k_i}^{\Delta_i}(\chi)$. Consider first the contribution of the product over $F_{k_i}^{\Delta_i}(\chi)$ to the right-hand side of~\eqref{eq:S-Tchar}. By the definition of $\mathcal{X}_1$ and the Cauchy--Schwarz inequality, the characters from $\mathcal{X}_1$ contribute
\begin{align} \label{eq:X1contr}
  \begin{aligned}
  &\ll \frac{Q_1^{-\alpha_1}}{\varphi(q)}  \sum_{\chi \in \mathcal{X}_1} \left| \sum_{\substack{\overline{k} \in \mathcal{K}}} \left( \prod_{\substack{i=1}}^3  F_{k_i}^{\Delta_i}(\chi)\right) U_{B_5, -k_1}^{\Delta_5}(\chi) M_{B_6, -k_2-k_3}^{\Delta_6}(\chi)\right| \\
&\ll Q_1^{-\alpha_1}\left(\frac{1}{\varphi(q)}\sum_{\chi\pmod q}|A(\chi)|^2\right)^{1/2}
\left(\frac{1}{\varphi(q)}\sum_{\chi\pmod q}|B(\chi)|^2\right)^{1/2},
\end{aligned}
\end{align}
where $A(\chi)$ is as in~\eqref{eq:A(chi)def} and 
\[
B(\chi):=\sum_{(k_2, k_3) \in \mathcal{K}_2} F_{k_2}^{\Delta_2}(\chi)F_{k_3}^{\Delta_3}(\chi)
M_{B_6,-k_2-k_3}^{\Delta_6}(\chi).
\]
By~\eqref{eq:Achi2bound} we have
\[
\frac{1}{\varphi(q)}\sum_{\chi\pmod q}|A(\chi)|^2 \ll \frac{\varphi(q)}{q^2}\log^2 q 
\]
and a similar argument gives
\[
\frac{1}{\varphi(q)}\sum_{\chi\pmod q}|B(\chi)|^2 \ll \frac{\varphi(q)}{q^2}\log^4 q. 
\]
Thus we can bound~\eqref{eq:X1contr} by
\[
  \ll Q_1^{-1/90} \frac{\varphi(q)}{q^2} \log^3 q.
  \]
The contribution of the product of $G_{k_i}^{\Delta_i}(\chi)$ can be bounded completely similarly.

\textbf{Contribution of $\mathcal{X}_j$ for $2 \leq j \leq J$.} 
Since
\[
  \sum_{j = 2}^J\left(\frac{1}{P_j^{1/2}} + \frac{1}{H_j^{1/2}}\right) \ll \frac{1}{Q_1^{1/80}}, 
\]
by Lemma~\ref{le:ChiSumMsplit}, instead of the right-hand side of~\eqref{eq:S-Tchar} with $\chi \in \cup_{j=2}^J \mathcal{X}_j$, it suffices to consider
\[
  \sum_{j = 2}^J  \frac{1}{\varphi(q)}  \sum_{\chi \in \mathcal{X}_j} \left|\sum_{\substack{\overline{k} \in \mathcal{K}}} \Biggl( \prod_{i=1}^3 F_{k_i}^{\Delta_i}(\chi) -  \prod_{i=1}^3 G_{k_i}^{\Delta_i}(\chi)\Biggr) Q_{B_4}^{\Delta_4}(\chi) U_{B_5, -k_1}^{\Delta_5}(\chi) \widetilde{M}_{j, B_6, -k_2-k_3}^{\Delta_6}(\chi)\right|.
  \]
We write this as a difference of two terms involving the product over
$F_{k_i}^{\Delta_i}(\chi)$ and the product over $G_{k_i}^{\Delta_i}(\chi)$. We only treat the former, since the
latter is handled in the same way.

Consider, for $j \in \{2, \dotsc, J\}$,
\[
E_j := \frac{1}{\varphi(q)}  \sum_{\chi \in \mathcal{X}_j} \left|\sum_{\substack{\overline{k} \in \mathcal{K}}} \left(\prod_{i=1}^3 F_{k_i}^{\Delta_i}(\chi)\right)  Q_{B_4}^{\Delta_4}(\chi) U_{B_5, -k_1}^{\Delta_5}(\chi) \widetilde{M}_{j, B_6, -k_2-k_3}^{\Delta_6}(\chi)\right|.
  \]
For every $\chi \in \mathcal{X}_j$, since $\chi \notin \mathcal{X}_{j-1}$, there exist
$b_8 \in \mathbb{Z}_q^\times$, $\Delta_8 \in \{+,-\}$ and $u \in \mathcal{I}_{j-1}$ such that
\begin{equation}\label{eq:Xjlargeprev-fixed}
|Q_{j-1,b_8H,u}^{\Delta_8}(\chi)| > e^{-\alpha_{j-1}u/H_{j-1}}.
\end{equation}
We partition $\mathcal{X}_j$ into subsets $\mathcal{X}_j(u,b_8,\Delta_8)$ according to one such choice of
$(u,b_8,\Delta_8)$. The number of such subsets is $\ll |\mathcal{I}_{j-1}|$.

Thus
\begin{align}
\label{eq:Xjstart-fixed}
\begin{aligned}
E_j &\ll \frac{1}{\varphi(q)} |\mathcal{I}_j| |\mathcal{I}_{j-1}|
\max_{\substack{\Delta_6,\Delta_7, \Delta_8 \in\{+,-\}\\ b_6,b_7, b_8 \in \mathbb{Z}_q^\times/H\\ w\in \mathcal{I}_j, u \in \mathcal{I}_{j-1}}}
\sum_{\chi\in \mathcal{X}_j(u,b_8,\Delta_8)}
\Biggl|
\sum_{\substack{\overline{k} \in \mathcal{K}}}
\Biggl(\prod_{i=1}^3 F_{k_i}^{\Delta_i}(\chi)\Biggr)
\\
&\qquad\qquad\qquad\qquad\qquad\qquad
  \cdot Q_{B_4}^{\Delta_4}(\chi)U_{B_5,-k_1}^{\Delta_5}(\chi)  Q_{j,b_6H,w}^{\Delta_6}(\chi)
R_{j,b_7H,-k_2-k_3,w}^{\Delta_7}(\chi) \Biggr|.
\end{aligned}
\end{align}

Assume $u \in \mathcal{I}_{j-1}$ and $w \in \mathcal{I}_j$ give the maximum here and put
\begin{equation} \label{eq:Y1Y2ldef}
Y_1 := e^{u/H_{j-1}}, \qquad Y_2 := e^{w/H_j}, \qquad 
\ell := \left\lceil \frac{\log Y_2}{\log Y_1}\right\rceil
= \left\lceil \frac{wH_{j-1}}{uH_j}\right\rceil.
\end{equation}
Then for $\chi \in \mathcal{X}_j(u,b_8,\Delta_8)$, by the definition of $\mathcal{X}_j$ and~\eqref{eq:Xjlargeprev-fixed} and,
\[
|Q_{j,b_6H,w}^{\Delta_6}(\chi)| \le e^{-\alpha_j w/H_j} \quad \text{and} \quad 1 \le e^{\ell \alpha_{j-1}u/H_{j-1}}|Q_{j-1,b_8H,u}^{\Delta_8}(\chi)|^\ell.
\]
Hence
\begin{align}
\label{eq:Xjsmalllarge-fixed}
|Q_{j,b_6H,w}^{\Delta_6}(\chi)| 
&\le e^{-\alpha_j w/H_j + \ell \alpha_{j-1}u/H_{j-1}}
|Q_{j-1,b_8H,u}^{\Delta_8}(\chi)|^\ell =: e^{d_j} |Q_{j-1,b_8H,u}^{\Delta_8}(\chi)|^\ell,
\end{align}
say.

Now
\[
\ell \frac{u}{H_{j-1}} \le \frac{w}{H_j} + \frac{u}{H_{j-1}},
\]
so
\begin{equation} \label{eq:aljwuexp}
d_j = -\alpha_j \frac{w}{H_j} + \ell \alpha_{j-1}\frac{u}{H_{j-1}}
\le
-(\alpha_j-\alpha_{j-1})\frac{w}{H_j} + \alpha_{j-1}\frac{u}{H_{j-1}}.
\end{equation}
Since $\alpha_{j-1} \leq 1/40$,
\[
\alpha_j-\alpha_{j-1}=\frac{\eta}{2j(j-1)}, \quad 
\frac{w}{H_j}\ge \log P_j, \quad \frac{u}{H_{j-1}}\le \log Q_{j-1},
\]
we obtain from~\eqref{eq:aljwuexp} that
\[
d_j \leq -\frac{\eta}{2j(j-1)} \log P_j + \frac{1}{40} \log Q_{j-1}
\]
From this and the definitions of $P_j$ and $Q_j$ in~\eqref{eq:PjQjdef} we see that once $Q_1$ sufficiently large in terms of $\eta$,
\[
 d_j \leq -\frac{\eta}{4j^2} \log P_j.
\]
Combining this with~\eqref{eq:Xjstart-fixed} and~\eqref{eq:Xjsmalllarge-fixed}, we obtain
\begin{align*}
\begin{aligned}
  E_j &\ll \frac{P_j^{-\eta/(4j^2)}}{\varphi(q)} |\mathcal{I}_j| |\mathcal{I}_{j-1}|
\max_{\substack{\Delta_6,\Delta_7, \Delta_8 \in\{+,-\}\\ b_6,b_7, b_8 \in \mathbb{Z}_q^\times/H\\ w\in \mathcal{I}_j, u \in \mathcal{I}_{j-1}}}
\sum_{\chi\pmod{q}}
 |Q_{j-1,b_8H,u}^{\Delta_8}(\chi)|^\ell \Biggl|
\sum_{\substack{\overline{k} \in \mathcal{K}}}
\Biggl(\prod_{i=1}^3 F_{k_i}^{\Delta_i}(\chi)\Biggr)
\\
&\qquad\qquad\qquad \qquad\qquad\qquad\qquad\qquad
Q_{B_4}^{\Delta_4}(\chi) U_{B_5,-k_1}^{\Delta_5}(\chi) R_{j,b_7H,-k_2-k_3,w}^{\Delta_7}(\chi) \Biggr|.
\end{aligned}
\end{align*}
Next we apply the Cauchy--Schwarz inequality. Let again
\[
A(\chi):=\sum_{k_1 \in \mathcal{K}_1} F_{k_1}^{\Delta_1}(\chi)U_{B_5,-k_1}^{\Delta_5}(\chi)
\]
and this time, let
\[
B(\chi):=\frac{1}{\log^2 q}\sum_{(k_2, k_3) \in \mathcal{K}_2}
F_{k_2}^{\Delta_2}(\chi)F_{k_3}^{\Delta_3}(\chi)
Q_{B_4}^{\Delta_4}(\chi)R_{j,b_7H,-k_2-k_3,w}^{\Delta_7}(\chi).
\]
By the Cauchy--Schwarz inequality,
\begin{align} 
\label{eq:Xj-CS} \begin{aligned}
E_j &\ll
P_j^{-\eta/(4j^2)} \log^2 q |\mathcal{I}_j| |\mathcal{I}_{j-1}|
\max_{\substack{\Delta_6,\Delta_7, \Delta_8 \in\{+,-\}\\ b_6,b_7, b_8 \in \mathbb{Z}_q^\times/H\\ w\in \mathcal{I}_j, u \in \mathcal{I}_{j-1}}}
\left(
\frac{1}{\varphi(q)}\sum_{\chi\pmod q}|A(\chi)|^2
\right)^{1/2} \\
& \qquad \cdot \left(
\frac{1}{\varphi(q)}\sum_{\chi\pmod q}|Q_{j-1,b_8H,u}^{\Delta_8}(\chi)|^{2\ell}|B(\chi)|^2
\right)^{1/2}.
\end{aligned}
\end{align}
The first factor can be bounded by~\eqref{eq:Achi2bound}. For the second factor, notice that the character sum in $Q_{j-1,b_8H,u}^{\Delta_8}(\chi)$ is supported on primes in $[Y_1, 2Y_1]$ whereas the coefficients of $B(\chi)$ are bounded and supported on $n \asymp qQ_1/Y_2$. Thus Lemma~\ref{le:amplify} is applicable with $X \asymp qQ_1$ and yields
\begin{align*}
\frac{1}{\varphi(q)}\sum_{\chi\pmod q}
|Q_{j-1,b_8H,u}^{\Delta_8}(\chi)|^{2\ell}|B(\chi)|^2
&\ll
  \frac{1}{Y_1^{2\ell} (q Q_1/Y_2)^2} \cdot \frac{\varphi(q)}{q} \cdot q (Q_1 Y_1)^2 \cdot 4^\ell (\ell+1)!^2\\
  & \ll \frac{\varphi(q)}{q^2} Y_1^4 \ell^{3\ell}
\cdot\end{align*}
Here, using~\eqref{eq:Y1Y2ldef} and~\eqref{eq:PjQjdef},
\[
  Y_1^4 \ell^{3\ell} \ll Q_{j-1} \exp\left(\left(\frac{\log Q_j}{\log P_{j-1}} +1\right) \log \log Q_j\right) \ll P_j^{\eta/(100j^2)}.
  \]

Combining this with~\eqref{eq:Xj-CS},~\eqref{eq:Achi2bound},~\eqref{eq:|Ij|bound}, and~\eqref{eq:PjQjdef}, we conclude that 
\[
E_j \ll P_j^{-\eta/(5j^2)} \log^3 q \cdot |\mathcal{I}_j| |\mathcal{I}_{j-1}| \frac{\varphi(q)}{q^2} \ll P_j^{-\eta/(10j^2)} \log^3 q \frac{\varphi(q)}{q^2} 
\]
and thus, using again~\eqref{eq:PjQjdef},
\[
  \sum_{j = 2}^J E_j \ll \frac{Q_1^{-1/90}}{q} \cdot \frac{\varphi(q)}{q} \log^3 q.
  \]

\textbf{Contribution of $\mathcal{Y}$.}
By Proposition~\ref{prop:f=g+h}(iii), the characters $\chi \in \mathcal{Y}$ contribute to the right-hand side of~\eqref{eq:S-Tchar} at most of order
\begin{align*}
\begin{aligned} 
\frac{1}{\varphi(q)}  \sum_{\substack{\overline{k} \in \mathcal{K}}} \max_{i_0 \in \{1, 2, 3\}}\sum_{\chi \in \mathcal{Y}} \left| F_{k_{i_0}}^{\Delta_{i_0}}(\chi) -  G_{k_{i_0}}^{\Delta_{i_0}}(\chi)\right| \prod_{\substack{i=1 \\ i \neq i_0}}^3 \left|F_{k_{i}}^{\Delta_{i}}(\chi) \right|  \left|Q_{B_4}^{\Delta_4}(\chi) U_{B_5, -k_1}^{\Delta_5}(\chi) M_{B_6, -k_2-k_3}^{\Delta_6}(\chi)\right|.
\end{aligned}
\end{align*}
For simplicity, we consider the case that the maximum is attained for $i_0 = 1$, the other cases are handled similarly. Recall the definition of $\delta$ from~\eqref{eq:deltadef}. By Proposition~\ref{prop:f=g+h}(ii), the contribution of $i_0 = 1$ is
\begin{align*}
&\ll \frac{\log^{-1/4} q}{\varphi(q)} \sum_{\overline{k} \in \mathcal{K}} \sum_{\chi \in \mathcal{Y}}  \left|F_{k_2}^{\Delta_2}(\chi) F_{k_3}^{\Delta_3}(\chi) Q_{B_4}^{\Delta_4}(\chi) U_{B_5, -k_1}^{\Delta_5}(\chi) M_{B_6, -k_2-k_3}^{\Delta_6}(\chi)\right| =: \Sigma,
\end{align*}
say. Using the trivial estimates 
\begin{align*}
|Q_{B_4}^{\Delta_4}(\chi)| &\ll \frac{1}{\log Q_1}, \qquad |U_{B_5, -k_1}^{\Delta_5}(\chi)|\ll |U_{B_5, -k_1}^{\Delta_5}(\chi_0)|, \\
  |M_{B_6, -k_2-k_3}^{\Delta_6}(\chi)| &\ll \frac{1}{Me^{-k_2-k_3}} \sum_{\substack{n \in I_{M}(-k_2-k_3) \\ (n, qP(Q_1)) = 1}} 1 \ll \frac{\varphi(q)}{q} \cdot \frac{(q, P(Q_1))}{\varphi((q, P(Q_1)))} \cdot \frac{1}{\log Q_1},
  \end{align*}
  and applying H\"older's inequality, we see that
\begin{align*}
  \Sigma&\ll \frac{\log^{-1/4} q}{q \log^2 Q_1} \frac{(q, P(Q_1))}{\varphi((q, P(Q_1)))} \sum_{\overline{k} \in \mathcal{K}} |U_{B_5, -k_1}^{\Delta_5}(\chi_0)| \left(\sum_{\chi \in \mathcal{Y}}|F_{k_2}^{\Delta_2}(\chi)|^2\right)^{1/2}\cdot \left(\sum_{\chi \in \mathcal{Y}}|F_{k_3}^{\Delta_3}(\chi)|^2\right)^{1/2}. 
 \end{align*}
By Lemma~\ref{le:largevalue}, we have
\begin{align*}
|\mathcal{Y}|\ll q^{1/20+1/200}.
\end{align*}
Hence, by Lemma~\ref{le:Hal-Mon} we have, for $i \in \{2, 3\}$,
\begin{align*}
\sum_{\chi\in \mathcal{Y}}\left|F_{k_i}^{\Delta_i}(\chi)\right|^2 \ll 1,
\end{align*}
and so, recalling $\mathcal{K} = \mathcal{K}_1 \times \mathcal{K}_2$ with $|\mathcal{K}_2| \ll \log^2 q$,
\[
\Sigma \ll \frac{\log^{7/4} q}{q \log^2 Q_1} \frac{(q, P(Q_1))}{\varphi((q, P(Q_1)))} \sum_{k_1 \in \mathcal{K}_1} |U_{B_5, -k_1}^{\Delta_5}(\chi_0)|.
  \]
The claim follows by combining the contributions of the sums over $\mathcal{X}_j$ and $\mathcal{Y}$.
\end{proof}

\section{Proof of Theorem~\ref{thm:general}: Working with the dense model}
As in Section~\ref{sec:DenseModelWorkEasy}, it will be convenient to work with a subset of $\mathbb{Z}_q^\times$ rather than the dense model functions $g_k^\Delta$. To facilitate this, we make the following definition.
\begin{definition} \label{def:AkDelta}
For $|k| \leq K$ and $\Delta \in \{+, -\}$, let $g_k^\Delta$ be as in Definition~\ref{def:gkDef}. Define
\begin{align*}
A_k^{\Delta} \coloneqq \{a \in \mathbb{Z}_q^{\times} \colon |g_k^{\Delta}(a)| \geq \varepsilon^2\}.
\end{align*}
\end{definition}
The arguments giving Lemma~\ref{le:Aprop} also give also the following lemma --- the only difference is the slightly different support of $f^\Delta_k$. 
\begin{lemma} \label{le:Akdeltabasics} Let $\Delta \in \{+, -\}$ and $|k| \leq K$, and let $A_k^\Delta$ be as in Definition~\ref{def:AkDelta}. 
  \begin{enumerate}
\item We have
\begin{align*}
|A_k^{+}| + |A_k^{-}|\geq \left(1-\varepsilon\right)\varphi(q).
\end{align*}
\item For any subgroup $H \leq \mathbb{Z}_q^{\times}$ of index at most $2$ and any $b\in \mathbb{Z}_q^{\times}$, we have
\begin{equation*}
|A_k^{\Delta} \cap bH|\geq \left(\frac{|\{n \in [I_R(k)]_q, (n, P(z)) = 1, \sgn(h(n)) = \Delta\} \cap bH|}{|\{n \in [I_R(k)]_q, (n, P(z)) = 1\}|}-\varepsilon\right)\varphi(q).
\end{equation*}
\end{enumerate}
\end{lemma}
Our next task is to lower bound $T_{B_4, B_5, B_6}^{\overline{\Delta}, \mathcal{K}}$. For this, we define, for $\overline{\Delta} = (\Delta_1, \dotsc, \Delta_6) \in \{+, -\}^6$ and $B_4, B_5, B_6 \subseteq \mathbb{Z}_q^\times$, and $\overline{k} = (k_1, k_2, k_3) \in (\mathbb{Z} \cap [-K, K])^3$, the function $T_{B_4, B_5, B_6, \overline{k}}^{\overline{\Delta}} \colon \mathbb{Z}_q^\times \to \mathbb{R}_{\geq 0}$,
\[
T_{B_4, B_5, B_6, \overline{k}}^{\overline{\Delta}}(a) :=\frac{1}{T_{\overline{k}}} \left(g_{k_1}^{\Delta_1} \ast g_{k_2}^{\Delta_2} \ast g_{k_3}^{\Delta_3} \ast \1[\mathcal{Q}^{\Delta_4}_{B_4}] \ast \1[\mathcal{U}^{\Delta_5}_{B_5, -k_1}]\ast  \1[\mathcal{M}^{\Delta_6}_{B_6, -k_2-k_3}]\right)(a),
\]
where $T_{\overline{k}}$ is as in~\eqref{eq:Tkbardef}.

\begin{lemma} \label{le:Tlower(iii)}
  Let $G = \mathbb{Z}_q^\times$ and $\overline{k} = (k_1, k_2, k_3) \in (\mathbb{Z} \cap [-K, K])^3$, and let $A_k^\Delta$ be as in Definition~\ref{def:AkDelta}. Assume that the following two conditions hold.
  \begin{itemize}
    \item[(A1)] There exist $\Delta_1, \Delta_2, \Delta_3 \in \{+, -\}$ such that
\[
\left(\mathbf{1}[A_{k_1}^{\Delta_1}] \ast \mathbf{1}[A_{k_2}^{\Delta_2}] \ast \mathbf{1}[A_{k_3}^{\Delta_3}]\right)(b) \gg \varphi(q)^2 
\]
for every $b \in \mathbb{Z}_q^\times.$
\item[(A2)] We have
\[
\sum_{\substack{p \leq q^{1/2} \\ h(p) < 0}} \frac{1}{p} \gg \frac{1}{Q_1^{1/100}}.
\] 
\end{itemize}
Then, for every $\Delta \in \{+, -\}$ and $a \in \mathbb{Z}_q^\times$, there exist $\Delta_4, \Delta_5, \Delta_6 \in \{+, -\}$ such that $\Delta_1 \dotsm \Delta_6 = \Delta$ and, for  $\overline{\Delta} = (\Delta_1, \dotsc, \Delta_6),$
\[
T_{G, G, G, \overline{k}}^{\overline{\Delta}}(a) \gg \frac{1}{q \log^2 Q_1} \cdot  \frac{(q, P(Q_1))}{\varphi((q, P(Q_1)))} \left(\frac{|\mathcal{U}_{G, -k_1}^{\Delta_5}|}{Ue^{-k_1}} + \frac{\varphi(q)}{q} Q_1^{-1/100}\right).
    \]
\end{lemma}

\begin{proof}
  Recall that, for $j = 1, 2, 3$, $g^{\Delta_j}_{k_j}(b)\geq \varepsilon^2\mathbf{1}_{A_{k_j}^{\Delta_j}}(b)$ for every $b \in \mathbb{Z}_q^\times$. Notice that, for any $|v| \leq 2K$,
  \begin{align*}
         &|\mathcal{M}_{G, v}^{+}| + |\mathcal{M}_{G, v}^{-}|\\
        &= |\{m \in I_M(v) \colon |\mu(m)| = 1, \, m \in \mathcal{S}, \, (m, qP(Q_1)) = 1\}| \\
      &\geq |\{m \in I_{M}(v) \colon |\mu(m)| = 1, \, (m, qP(Q_1)) = 1\}| \\
      & \quad - \sum_{j = 2}^J |\{m \in I_{M}(v) \colon |\mu(m)| = 1, \, p' \mid m \implies p' \not \in (P_j, Q_j], \, (m, qP(Q_1)) = 1\}|.
     \end{align*}
  Recalling $Q_1 \geq B(q)$, where $B(q)$ is as in~\eqref{eq:defB(q)}, we see from the fundamental lemma of the sieve and Mertens' theorem that
  \begin{align}  \label{eq:Mtotalsize}
    \begin{aligned}
      |\mathcal{M}_{G, v}^{+}| + |\mathcal{M}_{G, v}^{-}| &\geq  \frac{1}{10} \cdot \frac{\varphi(q)}{q} \frac{(q, P(Q_1))}{\varphi((q, P(Q_1)))} \cdot \frac{Me^v}{\log Q_1}\\
      &\quad - \sum_{j = 2}^J \frac{\varphi(q)}{q} \frac{(q, P(Q_1))}{\varphi((q, P(Q_1)))} \frac{Me^v}{\log Q_1} \left(\frac{\log P_j}{\log Q_j}\right)^{4/5} \\
      &\gg \frac{\varphi(q)}{q} \cdot \frac{(q, P(Q_1))}{\varphi((q, P(Q_1)))} \cdot \frac{Me^v}{\log Q_1}.
    \end{aligned}
  \end{align}
Thus, by the pigeonhole principle, we can choose $\Delta_6$ in such a way that
\begin{equation} \label{eq:fDelta2av}
|\mathcal{M}_{G, -k_2-k_3}^{\Delta_6}| \gg \frac{\varphi(q)}{q} \cdot \frac{(q, P(Q_1))}{\varphi((q, P(Q_1)))} \cdot \frac{Me^{-k_2-k_3}}{\log Q_1}.
\end{equation}
Furthermore, by the prime number theorem and pigeonhole principle (recalling again $Q_1 \geq B(q)$ with $B(q)$ as in~\eqref{eq:defB(q)}), we can choose $\Delta_4 \in \{+, -\}$ in such that
\begin{equation} \label{eq:QlowG}
|\mathcal{Q}_G^{\Delta_4}| \gg \frac{Q_1}{\log Q_1}.
\end{equation}

Let $\Delta \in \{+, -\}$ and $a \in \mathbb{Z}_q^\times$ be arbitrary. Choose $\Delta_5 = \Delta \cdot \Delta_1 \Delta_2 \Delta_3 \Delta_4 \Delta_6$, so that $\Delta_1 \dotsm \Delta_6 = \Delta$. Now, for $\overline{\Delta} = (\Delta_1, \dotsc, \Delta_6)$, we have
\begin{align*}
&T_{G, G, G, \overline{k}}^{\overline{\Delta}}(a) \geq \frac{\varepsilon^6}{T_{\overline{k}}} \left(\mathbf{1}[A_{k_1}^{\Delta_1}] \ast \mathbf{1}[A_{k_2}^{\Delta_2}] \ast \mathbf{1}[A_{k_3}^{\Delta_3}] \ast \1[\mathcal{Q}^{\Delta_4}_G] \ast \1[\mathcal{U}^{\Delta_5}_{G, -k_1}] \ast \1[\mathcal{M}_{G, -k_2-k_3}^{\Delta_6}]\right)(a) \\
& \gg \frac{1}{\varphi(q)^3 Q_1 \cdot Ue^{-k_1} \cdot Me^{-k_2-k_3}} \sum_{p \in \mathcal{Q}^{\Delta_4}_G} \sum_{m \in \mathcal{M}_{G, -k_2-k_3}^{\Delta_6}} \sum_{u \in \mathcal{U}_{G, -k_1}^{\Delta_5}} \left(\mathbf{1}[A_{k_1}^{\Delta_1}] \ast \mathbf{1}[A_{k_2}^{\Delta_2}] \ast \mathbf{1}[A_{k_3}^{\Delta_3}]\right)(a\overline{pmu}).
\end{align*}
Recalling (A1) and~\eqref{eq:fDelta2av}--\eqref{eq:QlowG}, we see that
\begin{align*}
T_{G, G, G, \overline{k}}^{\overline{\Delta}}(a) \gg \frac{1}{q \log^2 Q_1} \cdot \frac{(q, P(Q_1))}{\varphi((q, P(Q_1)))} \cdot \frac{|\mathcal{U}_{G, -k_1}^{\Delta_5}|}{U e^{-k_1}}.
\end{align*}
The claim follows by combining this with Lemma~\ref{lemma:le:multfunctsigns}, using (A2).
\end{proof}

\begin{lemma} \label{le:Tlower(iv)}
  Let $\overline{k} = (k_1, k_2, k_3) \in (\mathbb{Z} \cap [-K, K])^3$, and let $A_k^\Delta$ be as in Definition~\ref{def:AkDelta}. Let $H \leq \mathbb{Z}_q^\times$ be a subgroup of index two such that the following two conditions hold.
  \begin{itemize}
    \item[(A1)] There exist elements $b^+, b^- \in \mathbb{Z}_q^\times$ with $b^+H \neq b^-H$ such that
\[
\left(\mathbf{1}[A_{k_1}^{+}] \ast \mathbf{1}[A_{k_2}^{+}] \ast \mathbf{1}[A_{k_3}^{+}]\right)(b) \gg \varphi(q)^2
\]
for every $b \in b^+H$ and
\[
\left(\mathbf{1}[A_{k_1}^{-}] \ast \mathbf{1}[A_{k_2}^{-}] \ast \mathbf{1}[A_{k_3}^{-}]\right)(b) \gg \varphi(q)^2
\]
for every $b \in b^- H$.
\item[(A2)] Let $\chi$ be the quadratic character for which $\chi(b) = 1$ iff $b \in H$. We have
\begin{equation*} 
\sum_{\substack{p \leq q^{1/2} \\ h(p)\chi(p) < 0}} \frac{1}{p} \gg \frac{1}{Q_1^{1/100}}.
\end{equation*}
\end{itemize}
Then, for every $\Delta \in \{+, -\}$ and $a \in \mathbb{Z}_q^\times$, there exists $b_4, b_5, b_6 \in \mathbb{Z}_q^\times$ and $\overline{\Delta} = (\Delta_1, \dotsc, \Delta_6)$ such that $\Delta_1 \dotsm \Delta_6 = \Delta$ and
\[
T_{b_4H, b_5H, b_6 H, \overline{k}}^{\overline{\Delta}}(a) \gg \frac{1}{q \log^2 Q_1} \cdot \frac{(q, P(Q_1))}{\varphi((q, P(Q_1)))} \cdot \left(\frac{|\mathcal{U}_{b_5 H, -k_1}^{\Delta_5}|}{U e^{-k_1}} + \frac{\varphi(q)}{q} Q_1^{-1/100} \right).
    \]
\end{lemma}

\begin{proof}
Recall that, for $i \in \{1, 2, 3\}$, $g^{\Delta_i}_{k_i}(b)\geq \varepsilon^2\mathbf{1}[A_{k_i}^{\Delta_i}](b)$ for every $b \in \mathbb{Z}_q^\times$. By the pigeonhole principle and~\eqref{eq:Mtotalsize} we can choose $b_6 \in \mathbb{Z}_q^\times$ and $\Delta_6 \in \{+, -\}$ in such a way that
\begin{equation} \label{eq:fDelta56av}
|\mathcal{M}_{b_6 H, -k_2-k_3}^{\Delta_6}| \gg \frac{\varphi(q)}{q} \cdot \frac{(q, P(Q_1))}{\varphi((q, P(Q_1)))} \cdot \frac{Me^{-k_2-k_3}}{\log Q_1}.
\end{equation}
Furthermore, as in~\eqref{eq:QlowG}, we can choose $b_4 \in \mathbb{Z}_q^\times$ and $\Delta_4 \in \{+, -\}$ in such a way that
\begin{equation}\label{eq:QlowGb4}
|\mathcal{Q}_{b_4 H}^{\Delta_4}| \gg \frac{Q_1}{\log Q_1}.
\end{equation}

Let $\Delta \in \{+, -\}$ and $a \in \mathbb{Z}_q^\times$ be arbitrary. Take $\Delta_1 = \Delta_2 = \Delta_3 = +, \, \Delta_5 = \Delta \Delta_4 \Delta_6$ and $\Delta_1' = \Delta_2' = \Delta_3' = -, \, \Delta_5'=-\Delta \Delta_4 \Delta_6 $. Let further $\overline{\Delta} = (\Delta_1,\dotsc, \Delta_6)$ and $\overline{\Delta'} = (\Delta_1', \Delta_2', \Delta_3', \Delta_4, \Delta_5', \Delta_6)$. Now $\Delta_1 \dotsm \Delta_6 = \Delta_1' \Delta_2' \Delta_3' \Delta_4 \Delta_5' \Delta_6  = \Delta$. Choose $b_5 = a \overline{b^+ b_4 b_6}$ and $b'_5 = a \overline{b^- b_4 b_6}$. Let $B_4 = b_4 H, B_5 = b_5 H, B_5' = b_5' H, B_6 = b_6 H$.

Now
\begin{align*}
  &T_{B_4, B_5, B_6, \overline{k}}^{\overline{\Delta}}(a) + T_{B_4, B_5', B_6, \overline{k}}^{\overline{\Delta'}}(a)\\
  \geq& \frac{\varepsilon^6}{T_{\overline{k}}} \left(\mathbf{1}[A_{k_1}^{+}] \ast \mathbf{1}[A_{k_2}^{+}] \ast \mathbf{1}[A_{k_3}^{+}] \ast \1[\mathcal{Q}^{\Delta_4}_{B_4}] \ast \1[\mathcal{U}^{\Delta_5}_{B_5, -k_1}] \ast \1[\mathcal{M}^{\Delta_6}_{B_6, -k_2-k_3}]\right)(a) \\
  & \quad + \frac{\varepsilon^6}{T_{\overline{k}}} \left(\mathbf{1}[A_{k_1}^{-}] \ast \mathbf{1}[A_{k_2}^{-}] \ast \mathbf{1}[A_{k_3}^{-}] \ast \1[\mathcal{Q}^{\Delta_4}_{B_4}] \ast \1[\mathcal{U}^{\Delta_5'}_{B_5', -k_1}] \ast \1[\mathcal{M}^{\Delta_6}_{B_6, -k_2-k_3}]\right)(a) \\
   & \gg \frac{1}{T_{\overline{k}}} \sum_{p \in \mathcal{Q}^{\Delta_4}_{B_4}} \sum_{u \in \mathcal{U}^{\Delta_5}_{B_5, -k_1}}  \sum_{m \in \mathcal{M}^{\Delta_6}_{B_6, -k_2-k_3}} \left(\mathbf{1}[A_{k_1}^{+}] \ast \mathbf{1}[A_{k_2}^{+}] \ast \mathbf{1}[A_{k_3}^{+}]\right)(a\overline{pum}) \\
  &\quad + \frac{1}{T_{\overline{k}}}  \sum_{p \in \mathcal{Q}^{\Delta_4}_{B_4}} \sum_{u \in \mathcal{U}^{\Delta_5'}_{B_5', -k_1}} \sum_{m \in \mathcal{M}^{\Delta_6}_{B_6, -k_2-k_3}} \left(\mathbf{1}[A_{k_1}^{-}] \ast \mathbf{1}[A_{k_2}^{-}] \ast \mathbf{1}[A_{k_3}^{-}]\right)(a\overline{pum}).
\end{align*}
On the first line on the right-hand side the argument of the convolution $a\overline{pum}$ is in $a\overline{b_4 H b_5 H b_6 H}  = b^+H$ and on the second line on the right-hand side it is in $a\overline{b_4 H b'_5 H b_6 H} = b^- H$. 

Recalling (A1),~\eqref{eq:Tkbardef}, and~\eqref{eq:fDelta56av}--\eqref{eq:QlowGb4}, we see that

\begin{align*}
  &T^{\overline{\Delta}}_{B_4, B_5, B_6, \overline{k}}(a) + T^{\overline{\Delta'}}_{B_4, B_5', B_6, \overline{k}}(a) \gg \frac{1}{q \log^2 Q_1} \cdot \frac{(q, P(Q_1))}{\varphi((q, P(Q_1)))} \cdot \frac{|\mathcal{U}^{\Delta_5}_{b_5 H, -k_1}| + |\mathcal{U}^{\Delta_5'}_{b_5' H, -k_1}|}{Ue^{-k_1}}.
  \end{align*}
By Lemma~\ref{lemma:le:multfunctsigns} and (A2), we have 
\begin{align*}
|\mathcal{U}^{\Delta_5}_{b_5 H, -k_1}| + |\mathcal{U}^{\Delta_5'}_{b_5' H, -k_1}| = \sum_{\substack{u \in I_U(-k_1) \\ \sgn(h(u))\chi(u) = \Delta_5 \chi(b_5)}} |\mu(u)| \gg \frac{\varphi(q)}{q} \cdot \frac{Ue^{-k_1}}{Q_1^{1/100}}.
\end{align*}

  Hence either
\begin{align*}
T^{\overline{\Delta}}_{B_4, B_5, B_6, \overline{k}}(a) \gg \frac{1}{q\log^2 Q_1} \cdot \frac{(q, P(Q_1))}{\varphi((q, P(Q_1)))} \left(\frac{|\mathcal{U}^{\Delta_5}_{b_5 H, -k_1}|}{Ue^{-k_1}} +\frac{\varphi(q)}{q} \cdot \frac{1}{Q_1^{1/100}}\right)
  \end{align*}
  or
  \begin{align*}
T^{\overline{\Delta'}}_{B_4, B_5', B_6, \overline{k}}(a) \gg \frac{1}{q\log^2 Q_1} \cdot \frac{(q, P(Q_1))}{\varphi((q, P(Q_1)))} \left(\frac{|\mathcal{U}^{\Delta'_5}_{b'_5 H, -k_1}|}{Ue^{-k_1}} + \frac{\varphi(q)}{q} \cdot \frac{1}{Q_1^{1/100}}\right),
  \end{align*}
and the claim follows.
\end{proof}

\section{Proof of Theorem~\ref{thm:general}: Final case analysis}\label{sec:general2}
\begin{proof}[Proof of Theorem~\ref{thm:general}]
If~\eqref{eq:thmCondition} holds for some character $\chi$ of order at most two, there is nothing to prove. Hence we can assume that Lemma~\ref{le:Tlower(iii)}(A2) and Lemma~\ref{le:Tlower(iv)}(A2) hold, so that in order to apply these lemmas it suffices to show that (A1) holds.

By Lemma~\ref{le:reductiontoBoundingS}, it suffices to show that, for every $a \in \mathbb{Z}_q^\times$ and $\Delta \in \{+, -\}$, there exist $\overline{\Delta} = (\Delta_1, \dotsc, \Delta_6), \, \mathcal{K} \subseteq (\mathbb{Z} \cap [-K, K])^3,$ and $B_4, B_5, B_6 \subseteq \mathbb{Z}_q^\times$ such that $\Delta = \Delta_1 \dotsm \Delta_6$ and
\begin{equation} \label{eq:Sclaim}
  S_{B_4, B_5, B_6}^{\overline{\Delta}, \mathcal{K}}(a) \gg \frac{\varphi(q)}{q} \cdot \frac{(\log q)^3}{qQ_1^{1/100} \log^2 Q_1}.
\end{equation}
We split into three cases. 

\textbf{Case 1:} There is a sign $\Delta_1 \in \{+, -\}$ and a set $\mathcal{K}_1 \subseteq \mathbb{Z} \cap [-K, K]$ of size at least $K/20$ such that, for every $k_1 \in \mathcal{K}_1$, there exists a set $\mathcal{K}_2(k_1) \subseteq (\mathbb{Z} \cap [-K, K])^2$ of size at least $K^2/400$ such that, for every $(k_2, k_3) \in \mathcal{K}_2(k_1)$ and every $b \in \mathbb{Z}_q^\times$, 
\begin{align*}
(\mathbf{1}[A_{k_1}^{\Delta_1}] \ast \mathbf{1}[A_{k_2}^{\Delta_1}] \ast \mathbf{1}[A_{k_3}^{\Delta_1}])(b) \gg \varphi(q)^2.
\end{align*}
Thus Lemma~\ref{le:Tlower(iii)}(A1) holds for all such $(k_1, k_2, k_3)$. By Lemma~\ref{le:Tlower(iii)} and the pigeonhole principle, adjusting the sets $\mathcal{K}_1, \mathcal{K}_2(k_1)$ (that now have sizes $\geq K/160$ and $K^2/3200$), there exist $\Delta_4, \Delta_5,\Delta_6 \in \{+, -\}$ such that $\Delta = \Delta_1 \dotsm \Delta_6$ and, for $\overline{\Delta} = (\Delta_1, \dotsc, \Delta_6)$, we have
\[
T_{G, G, G, \overline{k}}^{\overline{\Delta}}(a) \gg \frac{1}{q \log^2 Q_1} \cdot  \frac{(q, P(Q_1))}{\varphi((q, P(Q_1)))} \left(\frac{|\mathcal{U}_{G, -k_1}^{\Delta_5}|}{Ue^{-k_1}} +  \frac{\varphi(q)}{q} \cdot Q_1^{-1/100}\right).
\]
whenever $\overline{k} = (k_1, k_2, k_3)$ with $k_1 \in \mathcal{K}_1$ and $(k_2, k_3) \in \mathcal{K}_2(k_1)$. Consequently, writing $\mathcal{K} = \mathcal{K}_1 \times (\mathbb{Z} \cap [-K, K])^2$,
\begin{align*}
T_{G, G, G}^{\overline{\Delta}, \mathcal{K}}(a)
&\gg \frac{K^2}{3200} \sum_{k_1 \in \mathcal{K}_1}
\frac{1}{q \log^2 Q_1} \cdot  \frac{(q, P(Q_1))}{\varphi((q, P(Q_1)))}
\left(\frac{|\mathcal{U}_{G, -k_1}^{\Delta_5}|}{Ue^{-k_1}} +  \frac{\varphi(q)}{q} \cdot  Q_1^{-1/100}\right) \\
&\gg \frac{K^2}{q \log^2 Q_1} \cdot  \frac{(q, P(Q_1))}{\varphi((q, P(Q_1)))}
\left(
\sum_{k_1 \in \mathcal{K}_1}\frac{|\mathcal{U}_{G, -k_1}^{\Delta_5}|}{Ue^{-k_1}}
+K \frac{\varphi(q)}{q} Q_1^{-1/100}
\right).
\end{align*}
Since $K \gg \log q$, the claim~\eqref{eq:Sclaim} follows from Lemma~\ref{lemma:le:STcomparison}.

\textbf{Case 2:} There exists a sign $\Delta_1 \in \{+, -\}$ and a set $\mathcal{K}_1 \subseteq \mathbb{Z} \cap [-K, K]$ of size at least $\frac{1}{10} K$ such that, for every $k \in \mathcal{K}_1$, 
\begin{align*}
|A_k^{\Delta_1}| \geq \left(\frac{1}{2}+\frac{1}{100}\right) \varphi(q).
\end{align*}
By Lemma~\ref{le:convolution}(i), for any triple $(k_1, k_2, k_3) \in \mathcal{K}_1^3$, we have
\[
\left(\mathbf{1}[A_{k_1}^{\Delta_1}] \ast \mathbf{1}[A_{k_2}^{\Delta_1}] \ast \mathbf{1}[A_{k_3}^{\Delta_1}]\right)(b) \gg \varphi(q)^2
\]
for every $b \in \mathbb{Z}_q^\times$. Hence we are actually in Case 1 with $\mathcal{K}_2(k_1) = \mathcal{K}_1^2$.

\textbf{Case 3}: We are not in Cases 1 or 2.
Since we are not in Case 2, by Lemma~\ref{le:Akdeltabasics} there exists a set $\mathcal{K} \subseteq \mathbb{Z} \cap [-K, K]$ of size at least $\frac{3}{2}K$ such that, for every $k \in \mathcal{K}$,
\begin{align*}
\left(\frac{1}{2}-\frac{1}{50}\right)\varphi(q) \leq |A_k^{+}|, |A_k^{-}| \leq \left(\frac{1}{2}+ \frac{1}{100}\right)\varphi(q).    
\end{align*}
By Lemma~\ref{le:KneserAppl} and the assumption that we are not in Case 1, there are $\geq K$ integers $k_1 \in \mathcal{K}$ such that there are $\geq K^2$ pairs $(k_2, k_3) \in \mathcal{K}^2$ such that, for every $i \in \{1, 2, 3\}$ and $\Delta \in \{+, -\}$, there are subgroups $H_{\overline{k},i}^\Delta \leq \mathbb{Z}_q^\times$ of index $2$ and elements $b_{\overline{k}, i}^\Delta \in \mathbb{Z}_q^\times$ such that 
\begin{align*}
|A_{k_i}^\Delta \cap b_{\overline{k}, i}^\Delta H_{\overline{k}, i}^\Delta| \geq |A_{k_i}^\Delta|-\frac{\varepsilon}{2}\varphi(q) \geq \left(\frac{1}{2}-\frac{1}{50}-\frac{\varepsilon}{2}\right)\varphi(q).
\end{align*}
Now, for each $A_{k_i}^\Delta$, this can happen only for one coset $b_{\overline{k}, i}^\Delta H_{\overline{k},i}^\Delta$ (since the intersection of two different cosets of subgroups of index $2$ has size at most $\varphi(q)/4$). Hence actually there exists a subset $\mathcal{H} \subseteq \mathcal{K}$ of size at least $K$ such that, for every $k \in \mathcal{H}$ and $\Delta \in \{+, -\}$, there exist a subgroup $H_k^\Delta \leq \mathbb{Z}_q^\times$ of index $2$ and an element $b_k^\Delta \in \mathbb{Z}_q^\times$ such that 
\begin{align}\label{eq:AR0}
|A_k^\Delta \cap b_k^\Delta H_k^\Delta| \geq |A_k^\Delta|-\frac{\varepsilon}{2}\varphi(q) \geq \left(\frac{1}{2}-\frac{1}{50}-\frac{\varepsilon}{2}\right)\varphi(q).
\end{align}

Let us first show that, for every $k \in \mathcal{H}$, we must have that
\begin{equation} \label{eq:b+neq-}
H_k^+ = H_k^- \quad \text{and} \quad  b^+_k H_k^+ \neq b_k^- H_k^-.
\end{equation}
If either of these fails, then, for $b_0 \not \in b^+ H_k^+,$ we have
\[
|b_0 H_k^+ \cap b^+ H_k^+| = 0 \quad \text{and} \quad |b_0 H_k^+ \cap b^- H_k^-| \in \left\{0,  \frac{\varphi(q)}{4}\right\}.
\]
Thus by~\eqref{eq:AR0}
\[
|A_k^+ \cap b_0 H_k^+| + |A_k^- \cap b_0 H_k^+| \leq \frac{\varepsilon}{2} \varphi(q) + \left(\frac{1}{4}+\frac{\varepsilon}{2}\right)\varphi(q) = \left(\frac{1}{4}+\varepsilon\right) \varphi(q),
\]
which by Lemma~\ref{le:Aprop}(ii) contradicts Lemma~\ref{le:P2}. Write $H_k = H_k^+ = H_k^-$. We split into two more cases.

\textbf{Case 3.1}: There exist a subgroup $H \leq \mathbb{Z}_q^\times$ of index $2$ and a set $\mathcal{H}_0 \subseteq \mathcal{H}$ of size at least $K/2$ such that $H_k = H$ for every $k \in \mathcal{H}_0$.
Let $\Delta_1 \in \{+, -\}$. By Lemma~\ref{le:convolution}(ii) we have 
\[
\left(\mathbf{1}[A_{k_1}^{\Delta_1}] \ast \mathbf{1}[A_{k_2}^{\Delta_1}] \ast \mathbf{1}[A_{k_3}^{\Delta_1}]\right)(b) \gg \varphi(q)^2
\]
for every $(k_1, k_2, k_3) \in \mathcal{H}_0^3$ and $b \in b_{k_1}^{\Delta_1}b_{k_2}^{\Delta_1}b_{k_3}^{\Delta_1}H$. Since $H$ has index $2$,~\eqref{eq:b+neq-} implies that $b_{k_1}^{+}b_{k_2}^{+}b_{k_3}^{+}H \neq b_{k_1}^{-}b_{k_2}^{-}b_{k_3}^{-}H.$ Thus Lemma~\ref{le:Tlower(iv)}(A1) holds for all $\overline{k} \in \mathcal{H}_0^3$.  Now we can finish the proof similarly to Case 1, but using Lemma~\ref{le:Tlower(iv)} in place of Lemma~\ref{le:Tlower(iii)}.

\textbf{Case 3.2:} We are not in Case 3.1. In this case we can find subsets $\mathcal{H}_1, \mathcal{H}_2 \subseteq \mathcal{H}$ with sizes $\geq K/4$ such that if $k_1 \in \mathcal{H}_1$ and $k_2 \in \mathcal{H}_2$, then $H_{k_1} \neq H_{k_2}$. Let $(k_1, k_2, k_3) \in \mathcal{H}_1 \times \mathcal{H}_2 \times \mathcal{H}_1$ and apply Lemma~\ref{le:KneserAppl} to $A_{k_1}^+, A_{k_2}^+, A_{k_3}^+$. Now Lemma~\ref{le:KneserAppl}(ii) does not hold and thus Lemma~\ref{le:KneserAppl}(i) must hold and we are actually in Case 1. 
\end{proof}

\bibliography{refs.bib}

@article {matomaki-radziwill,
    AUTHOR = {Matom\"aki, K. and Radziwi{\l}{\l} , M.},
     TITLE = {Multiplicative functions in short intervals},
   JOURNAL = {Ann. of Math. (2)},
  FJOURNAL = {Annals of Mathematics. Second Series},
    VOLUME = {183},
      YEAR = {2016},
    NUMBER = {3},
     PAGES = {1015--1056},
      ISSN = {0003-486X},
   MRCLASS = {11K65 (11N64)},
  MRNUMBER = {3488742},
MRREVIEWER = {Eugenijus Manstavi\v cius},
       URL = {https://doi.org/10.4007/annals.2016.183.3.6},
  display-string = {(2016) Matomaki, K. and Radziwill, M. - Multiplicative functions in short intervals},
}

@article{FordRad,
      title={Sign changes of the Liouville function in arithmetic progressions}, 
      author={Kevin Ford and Maksym Radziwiłł},
      year={2026},
      journal = {arXiv e-prints},
      pages = {arXiv:2605.03349},
eprint={2605.03349},
      archivePrefix={arXiv},
      primaryClass={math.NT},
      url={https://arxiv.org/abs/2605.03349},
  display-string = {(2026) Kevin Ford and Maksym Radziwiłł - Sign changes of the Liouville function in arithmetic progressions}, 
}

@article {kor,
    AUTHOR = {Korol\"ev, M. A.},
     TITLE = {Kloosterman sums with multiplicative coefficients},
   JOURNAL = {Izv. Ross. Akad. Nauk Ser. Mat.},
  FJOURNAL = {Izvestiya Rossiiskoi Akademii Nauk. Seriya Matematicheskaya},
    VOLUME = {82},
      YEAR = {2018},
    NUMBER = {4},
     PAGES = {3--17},
      ISSN = {1607-0046,2587-5906},
   MRCLASS = {11L05 (11L07)},
  MRNUMBER = {3833472},
MRREVIEWER = {Ioulia\ N.\ Baoulina},
       DOI = {10.4213/im8633},
       URL = {https://doi.org/10.4213/im8633},
  display-string = {(2018) Korolev, M. A. - Kloosterman sums with multiplicative coefficients},
}

@article {Chowla-AP,
AUTHOR = {S.~Chowla},
TITLE= {On the least prime in an arithmetical progression},
JOURNAL={J. Indian Math. Soc. (N.S.)},
VOLUME = {1},
YEAR = {1934},
PAGES = {1--3},
  display-string = {(1934) S.~Chowla - On the least prime in an arithmetical progression},}

@article {MatoTeraEkq,
    AUTHOR = {Matom\"aki, K. and Ter\"av\"ainen, J.},
     TITLE = {Products of primes in arithmetic progressions},
   JOURNAL = {J. Reine Angew. Math.},
  FJOURNAL = {Journal f\"ur die Reine und Angewandte Mathematik. [Crelle's
              Journal]},
    VOLUME = {808},
      YEAR = {2024},
     PAGES = {193--240},
      ISSN = {0075-4102,1435-5345},
   MRCLASS = {11N13 (11L07)},
  MRNUMBER = {4708120},
       DOI = {10.1515/crelle-2023-0096},
       URL = {https://doi.org/10.1515/crelle-2023-0096},
  display-string = {(2024) Matomaki, K. and Teravainen, J. - Products of primes in arithmetic progressions},
}

@article {MSsieve,
    AUTHOR = {Matom\"aki, K. and Shao, X.},
     TITLE = {When the sieve works {II}},
   JOURNAL = {J. Reine Angew. Math.},
  FJOURNAL = {Journal f\"ur die Reine und Angewandte Mathematik. [Crelle's
              Journal]},
    VOLUME = {763},
      YEAR = {2020},
     PAGES = {1--24},
      ISSN = {0075-4102,1435-5345},
   MRCLASS = {11N35 (11B25 11B30)},
  MRNUMBER = {4104277},
MRREVIEWER = {S\'andor\ Z.\ Kiss},
       DOI = {10.1515/crelle-2018-0034},
       URL = {https://doi.org/10.1515/crelle-2018-0034},
  display-string = {(2020) Matomaki, Kaisa and Shao, Xuancheng - When the sieve works II},
}

@article {GKM,
    AUTHOR = {Granville, A. and Koukoulopoulos, D. and Matom\"aki,
              K.},
     TITLE = {When the sieve works},
   JOURNAL = {Duke Math. J.},
  FJOURNAL = {Duke Mathematical Journal},
    VOLUME = {164},
      YEAR = {2015},
    NUMBER = {10},
     PAGES = {1935--1969},
      ISSN = {0012-7094,1547-7398},
   MRCLASS = {11N35 (11B30)},
  MRNUMBER = {3369306},
MRREVIEWER = {S\'andor\ Z.\ Kiss},
       DOI = {10.1215/00127094-3120891},
       URL = {https://doi.org/10.1215/00127094-3120891},
  display-string = {(2015) Granville, Andrew and Koukoulopoulos, Dimitris and Matomaki,
              Kaisa - When the sieve works},
}

@book{iw-kow,
    AUTHOR = {Iwaniec, H. and Kowalski, E.},
     TITLE = {Analytic number theory},
    SERIES = {American Mathematical Society Colloquium Publications},
    VOLUME = {53},
 PUBLISHER = {American Mathematical Society, Providence, RI},
      YEAR = {2004},
     PAGES = {xii+615},
      ISBN = {0-8218-3633-1},
   MRCLASS = {11-02 (11Fxx 11Lxx 11Mxx 11Nxx)},
  MRNUMBER = {2061214 (2005h:11005)},
MRREVIEWER = {K. Soundararajan},
  display-string = {(2004) Iwaniec, H. and Kowalski, E. - Analytic number theory},
}

@article {bgs,
    AUTHOR = {Balog, A. and Granville, A. and Soundararajan, K.},
     TITLE = {Multiplicative functions in arithmetic progressions},
   JOURNAL = {Ann. Math. Qu\'e.},
  FJOURNAL = {Annales Math\'ematiques du Qu\'ebec},
    VOLUME = {37},
      YEAR = {2013},
    NUMBER = {1},
     PAGES = {3--30},
      ISSN = {2195-4755,2195-4763},
   MRCLASS = {11N13 (11L40 11N37)},
  MRNUMBER = {3117735},
MRREVIEWER = {S.\ W.\ Graham},
       DOI = {10.1007/s40316-013-0001-z},
       URL = {https://doi.org/10.1007/s40316-013-0001-z},
  display-string = {(2013) Balog, A. and Granville, A. and Soundararajan, K. - Multiplicative functions in arithmetic progressions},
}

@book {friedlander,
    AUTHOR = {Friedlander, J. and Iwaniec, H.},
     TITLE = {Opera de cribro},
    SERIES = {American Mathematical Society Colloquium Publications},
    VOLUME = {57},
 PUBLISHER = {American Mathematical Society, Providence, RI},
      YEAR = {2010},
     PAGES = {xx+527},
      ISBN = {978-0-8218-4970-5},
   MRCLASS = {11N35 (11-02 11N05 11N13 11N25 11N36)},
  MRNUMBER = {2647984 (2011d:11227)},
MRREVIEWER = {D. R. Heath-Brown},
  display-string = {(2010) Friedlander, J. and Iwaniec, H. - Opera de cribro},
}

@article {Heath-Brown2,
    AUTHOR = {Heath-Brown, D. R.},
     TITLE = {Zero-free regions for {D}irichlet {$L$}-functions, and the
              least prime in an arithmetic progression},
   JOURNAL = {Proc. London Math. Soc. (3)},
      YEAR = {1992},
  display-string = {(1992) Heath-Brown, D. R. - Zero-free regions for Dirichlet $L$-functions, and the
              least prime in an arithmetic progression},
}

@article {KMT,
    AUTHOR = {Klurman, O. and Mangerel, A. P. and
              Ter\"av\"ainen, J.},
     TITLE = {Multiplicative functions in short arithmetic progressions},
   JOURNAL = {Proc. Lond. Math. Soc. (3)},
  FJOURNAL = {Proceedings of the London Mathematical Society. Third Series},
    VOLUME = {127},
      YEAR = {2023},
    NUMBER = {2},
     PAGES = {366--446},
      ISSN = {0024-6115,1460-244X},
   MRCLASS = {11N64 (11N13)},
  MRNUMBER = {4626713},
MRREVIEWER = {Tsz\ Ho\ Chan},
       DOI = {10.1112/plms.12546},
       URL = {https://doi.org/10.1112/plms.12546},
  display-string = {(2023) Klurman, O. and Mangerel, A. P. and
              Teravainen, J. - Multiplicative functions in short arithmetic progressions},
}

@book {montgomery-topics,
    AUTHOR = {Montgomery, H. L.},
     TITLE = {Topics in multiplicative number theory},
    SERIES = {Lecture Notes in Mathematics},
    VOLUME = {Vol. 227},
 PUBLISHER = {Springer-Verlag, Berlin-New York},
      YEAR = {1971},
     PAGES = {ix+178},
   MRCLASS = {10H30},
  MRNUMBER = {337847},
MRREVIEWER = {H.-E.\ Richert},
  display-string = {(1971) Montgomery, H: L. - Topics in multiplicative number theory},
}

@article {ramare-walker,
    AUTHOR = {Ramar\'{e}, O. and Walker, A.},
     TITLE = {Products of primes in arithmetic progressions: a footnote in
              parity breaking},
   JOURNAL = {J. Th\'{e}or. Nombres Bordeaux},
  FJOURNAL = {Journal de Th\'{e}orie des Nombres de Bordeaux},
    VOLUME = {30},
      YEAR = {2018},
    NUMBER = {1},
     PAGES = {219--225},
      ISSN = {1246-7405},
   MRCLASS = {11N13 (11A41 11B13 11N37)},
  MRNUMBER = {3809717},
MRREVIEWER = {Vilius Stakenas},
       URL = {http://jtnb.cedram.org/item?id=JTNB_2018__30_1_219_0},
  display-string = {(2018) Ramare, O. and Walker, A. - Products of primes in arithmetic progressions: a footnote in
              parity breaking},
}

@Article{Xylouris,
 Author = {Xylouris, T.},
 Title = {Linnik's constant is less than 5},
 FJournal = {Chebyshevski{\u{\i}} Sbornik},
 Journal = {Chebyshevski{\u{\i}} Sb.},
 ISSN = {2226-8383},
 Volume = {19},
 Number = {3(67)},
 Pages = {80--94},
 Year = {2018},
 Language = {German},
 DOI = {10.22405/2226-8383-2018-19-3-80-94},
 zbMATH = {7199802},
 Zbl = {1439.11244},
  display-string = {(2018) Xylouris, T. - Linnik's constant is less than 5}
}

@article {grynkiewicz,
    AUTHOR = {Grynkiewicz, D. J.},
     TITLE = {On extending {P}ollard's theorem for {$t$}-representable sums},
   JOURNAL = {Israel J. Math.},
  FJOURNAL = {Israel Journal of Mathematics},
    VOLUME = {177},
      YEAR = {2010},
     PAGES = {413--439},
      ISSN = {0021-2172},
   MRCLASS = {11B30 (05B10 20K99)},
  MRNUMBER = {2684428},
MRREVIEWER = {Michael A. Freeze},
       DOI = {10.1007/s11856-010-0053-6},
       URL = {https://doi.org/10.1007/s11856-010-0053-6},
  display-string = {(2010) Grynkiewicz, D. J. - On extending Pollard's theorem for $t$-representable sums},
}

@article {Jha,
    AUTHOR = {Jha, A.},
     TITLE = {Smallest totient in a residue class},
   JOURNAL = {Bull. Lond. Math. Soc.},
  FJOURNAL = {Bulletin of the London Mathematical Society},
    VOLUME = {57},
      YEAR = {2025},
    NUMBER = {6},
     PAGES = {1908--1917},
      ISSN = {0024-6093,1469-2120},
   MRCLASS = {11N64 (11B50 11L40)},
  MRNUMBER = {4918369},
MRREVIEWER = {Vilius\ Stakenas},
       DOI = {10.1112/blms.70069},
       URL = {https://doi.org/10.1112/blms.70069},
  display-string = {(2025) Jha, A. - Smallest totient in a residue class},
}

@article {Linnik1,
    AUTHOR = {Linnik, U. V.},
     TITLE = {On the least prime in an arithmetic progression. {I}. {T}he
              basic theorem},
   JOURNAL = {Rec. Math. [Mat. Sbornik] N.S.},
    VOLUME = {15(57)},
      YEAR = {1944},
     PAGES = {139--178},
   MRCLASS = {10.0X},
  MRNUMBER = {0012111},
MRREVIEWER = {H. Davenport},
  display-string = {(1944) Linnik, U. V. - On the least prime in an arithmetic progression. I. The
              basic theorem},
}

@article {Linnik2,
    AUTHOR = {Linnik, U. V.},
     TITLE = {On the least prime in an arithmetic progression. {II}. {T}he
              {D}euring-{H}eilbronn phenomenon},
   JOURNAL = {Rec. Math. [Mat. Sbornik] N.S.},
    VOLUME = {15(57)},
      YEAR = {1944},
     PAGES = {347--368},
   MRCLASS = {10.0X},
  MRNUMBER = {0012112},
MRREVIEWER = {H. Davenport},
  display-string = {(1944) Linnik, U. V. - On the least prime in an arithmetic progression. II. The
              Deuring-Heilbronn phenomenon},
}

@book {MVBook,
    AUTHOR = {Montgomery, H. L. and Vaughan, R. C.},
     TITLE = {Multiplicative number theory. {I}. {C}lassical theory},
    SERIES = {Cambridge Studies in Advanced Mathematics},
    VOLUME = {97},
 PUBLISHER = {Cambridge University Press, Cambridge},
      YEAR = {2007},
     PAGES = {xviii+552},
      ISBN = {978-0-521-84903-6; 0-521-84903-9},
   MRCLASS = {11-02 (11-01 11M06 11M26 11M45 11N05 11N37)},
  MRNUMBER = {2378655},
MRREVIEWER = {Wolfgang\ Schwarz},
  display-string = {(2007) Montgomery, H. L. and Vaughan, R. C. - Multiplicative number theory. I. Classical theory},
}

@article {klurman,
    AUTHOR = {Klurman, O.},
     TITLE = {Correlations of multiplicative functions and applications},
   JOURNAL = {Compos. Math.},
  FJOURNAL = {Compositio Mathematica},
    VOLUME = {153},
      YEAR = {2017},
    NUMBER = {8},
     PAGES = {1622--1657},
      ISSN = {0010-437X},
   MRCLASS = {11L40},
  MRNUMBER = {3705270},
       DOI = {10.1112/S0010437X17007163},
       URL = {http://dx.doi.org/10.1112/S0010437X17007163},
  display-string = {(2017) Klurman, O. - Correlations of multiplicative functions and applications},
}

@article {mrt,
    AUTHOR = {Matom\"aki, K. and Radziwi{\l}{\l}, M. and Tao, T.},
     TITLE = {An averaged form of {C}howla's conjecture},
   JOURNAL = {Algebra Number Theory},
  FJOURNAL = {Algebra \& Number Theory},
    VOLUME = {9},
      YEAR = {2015},
    NUMBER = {9},
     PAGES = {2167--2196},
      ISSN = {1937-0652},
   MRCLASS = {11P32},
  MRNUMBER = {3435814},
MRREVIEWER = {Martin Mereb},
       URL = {https://doi.org/10.2140/ant.2015.9.2167},
  display-string = {(2015) Matomaki, K. and Radziwill, M. and Tao, T. - An averaged form of Chowla's conjecture},
}

@book {tao-vu,
    AUTHOR = {Tao, T. and Vu, V.},
     TITLE = {Additive combinatorics},
    SERIES = {Cambridge Studies in Advanced Mathematics},
    VOLUME = {105},
 PUBLISHER = {Cambridge University Press, Cambridge},
      YEAR = {2006},
     PAGES = {xviii+512},
      ISBN = {978-0-521-85386-6; 0-521-85386-9},
   MRCLASS = {11-02 (05-02 05D10 11B13 11P70 11P82 28D05 37A45)},
  MRNUMBER = {2289012},
MRREVIEWER = {Serge\u\i  V. Konyagin},
       URL = {https://doi.org/10.1017/CBO9780511755149},
  display-string = {(2006) Tao, T. and Vu, V. - Additive combinatorics},
}
\bibliographystyle{plain}

\end{document}